\documentclass[a4paper,11pt]{article}
\usepackage[top=1in, left=1in, right=1in, bottom=1in]{geometry}
\usepackage[colorlinks=true, linkcolor=blue, citecolor=cyan, urlcolor=blue]{hyperref}

\usepackage{graphicx,enumitem}
\usepackage{color}
\usepackage{latexsym}
\usepackage{amsthm, amsmath, amssymb, mathrsfs,mathtools}
\usepackage{algorithmic}
\usepackage{algorithm}
\usepackage{comment}  
\usepackage{float}

\newtheorem{theorem}{Theorem}[section]
   
\newtheorem{remark}{Remark}
\newtheorem{proposition}[theorem]{Proposition}
\newtheorem{lemma}[theorem]{Lemma}
\newtheorem{problem}{Problem}
\newtheorem{corollary}[theorem]{Corollary}

\usepackage{bm}
\newcommand{\vect}[1]{\boldsymbol{\mathbf{#1}}}
\newcommand{\vertiii}[1]{{\left\vert\kern-0.25ex\left\vert\kern-0.25ex\left\vert #1 
    \right\vert\kern-0.25ex\right\vert\kern-0.25ex\right\vert}}
\usepackage{xcolor}

\newcommand{\dive}{\operatorname{div}}
\DeclareMathOperator*{\argmin}{arg\,min}
\newcommand{\tannabla}{\nabla_{\varGamma}}
\newcommand{\tandive}{\operatorname{div}_{\varGamma}}
\newcommand{\dotu}{\dot{u}}  
\newcommand{\cv}{\overline{v}}

\newcommand{\cphi}{\overline{\varphi}}

\newcommand{\ur}{u_{1}}

\newcommand{\ui}{u_{2}}

\newcommand{\vr}{p_{1}}
\newcommand{\vi}{p_{2}}
\newcommand{\qr}{q_{1}}
\newcommand{\qi}{q_{2}}
\newcommand{\rw}{w_{1}}
\newcommand{\iw}{w_{2}}
\newcommand{\rlam}{\Lambda_{1}}
\newcommand{\ilam}{\Lambda_{2}}

\newcommand{\Id}{\operatorname{Id}}
\newcommand{\intervalI}{I}
\newcommand{\Jact}{j_{t}}
\newcommand{\aaa}{{\mathsf{a}}}
\newcommand{\aat}{{\mathsf{a}}_t}
\newcommand{\GG}{G} 
\newcommand{\LL}{\vect{\mathsf{L}}}
\newcommand{\QQ}{\vect{\mathsf{Q}}}
\newcommand{\SSig}{\vect{\mathsf{S}}}

\newcommand{\sfTheta}{\Theta} 
\newcommand{\sfj}{F}

\newcommand{\HH}{\vect{\mathsf{H}}}

\newcommand{\VV}{\theta}
\newcommand{\Vn}{\theta_{n}}

\newcommand{\nn}{\vect{n}}

\newcommand{\dn}[1]{\partial_{\nn}{#1}}

\newcommand{\intO}[1]{\int_{\varOmega}{#1}{\, {d}{x}}} 
  
\newcommand{\intS}[1]{\int_{\varSigma}{#1}{\, {d}{s}}} 
\newcommand{\intG}[1]{\int_{\varGamma}{#1}{\, {d}{s}}}

\newcommand{\ADMMpara}{\gamma}
\begin{document}

\title{Cavity shape reconstruction with a homogeneous Robin condition via a constrained coupled complex boundary method with ADMM}
\author{Mustapha Essahraoui$^{\dagger}$   \and  El Mehdi Cherrat$^{\dagger}$ \and Lekbir Afraites$^{\dagger}$ \and Julius Fergy Tiongson Rabago$^{\ddagger}$\thanks{Corresponding author}}
\date{%
	{\footnotesize
        $^{\dagger}$Laboratory of Mathematics $\&$ Informatics and their Interactions (LM2I)\\%
        Faculty of Sciences and Techniques\\%
        Sultan Moulay Slimane University,
        Beni Mellal, Morocco\\\vspace{-2pt}
        \texttt{essahraouimoustapha@gmail.com, \ cherrat.elmehdi@gmail.com, \ l.afraites@usms.ma,\ lekbir.afraites@gmail.com}\\[2ex]
	$^{\ddagger}$Faculty of Mathematics and Physics\\%
	 Institute of Science and Engineering\\%
         Kanazawa University, Kanazawa 920-1192, Japan\\\vspace{-2pt}
        \texttt{jfrabago@gmail.com,\ jftrabago@gmail.com}}\\[2ex]        
    \today
}

\maketitle
\begin{abstract}
We revisit the problem of identifying an unknown portion of a boundary subject to a Robin condition, based on a pair of Cauchy data on the accessible part of the boundary.
It is known that a single measurement may correspond to infinitely many admissible domains.
Nonetheless, numerical strategies based on shape optimization have been shown to yield reasonable reconstructions of the unknown boundary.
In this study, we propose a new application of the coupled complex boundary method to address this class of inverse boundary identification problems. The overdetermined problem is reformulated as a complex boundary value problem with a complex Robin condition that couples the Cauchy data on the accessible boundary. The reconstruction is achieved by minimizing a cost functional constructed from the imaginary part of the complex-valued solution.
To improve stability with respect to noisy data and initialization, we augment the formulation with inequality constraints through prior admissible bounds on the state, leading to a constrained shape optimization problem. The shape derivative of the complex state and the corresponding shape gradient of the cost functional are derived, and the resulting problem is solved using an alternating direction method of multipliers (ADMM) framework. The proposed approach is implemented using the finite element method and validated through various numerical experiments.

\medskip

\end{abstract}


\section{Introduction}
\label{sec:Introduction} 
This study considers the classical problem of identifying a cavity inside a conducting body.  
The domain $\varOmega$ is assumed to be doubly connected, with an accessible exterior boundary $\varSigma$ and an interior boundary $\varGamma$, which do not intersect (i.e., $\varSigma \cap \varGamma = \emptyset$), to be identified from a single boundary measurement of a harmonic function $u$ on $\varSigma$.  
On the inaccessible boundary $\varGamma$, the function $u$ satisfies a homogeneous Robin condition.  
Given a Cauchy pair $(f,g)$ on $\varSigma$, we consider the following overdetermined boundary value problem:
\begin{equation}
\label{eq:gip}
-\,\Delta u = 0 \ \text{in } \varOmega, \qquad 
u = f, \ \partial_{\nn}u = g \ \text{on } \varSigma, \qquad 
\partial_{\nn}u + \alpha u = 0 \ \text{on } \varGamma,
\end{equation}
where $\alpha$ is a fixed non-negative Lipschitz function satisfying $\alpha(x) \geqslant \alpha_{0} > 0$ for all $x \in \varGamma$, and $\nn$ denotes the outward unit normal to $\partial\varOmega$.  
In the inverse problem framework, the objective is to determine the unknown boundary $\varGamma$ from the prescribed Cauchy data $(f,g)$ on $\varSigma$.
\begin{problem}\label{eq:inverse_problem}
Given Dirichlet data $f$ and corresponding Neumann data $g \coloneqq \partial_{\nn}u$ on $\varSigma$, where $u$ satisfies
    \begin{equation}
    \label{eq:state}
    -\,\Delta u = 0 \ \text{in } \varOmega, \qquad
    u = f \ \text{on } \varSigma, \qquad
    \partial_{\nn}u + \alpha u = 0 \ \text{on } \varGamma,
    \end{equation}
determine the unknown interior boundary $\varGamma$.
\end{problem}
In Problem~\ref{eq:inverse_problem}, the choices $\alpha = 0$ and $\alpha = \infty$ correspond, respectively, to homogeneous Neumann and homogeneous Dirichlet boundary conditions imposed on $\varGamma$. 
When $\alpha(x) \in (0,\infty)$ for $x \in \varGamma$, the boundary condition is of homogeneous Robin type. According to Newton’s cooling law, this condition describes uniform environmental interactions (thermal or electrostatic), which are taken to vanish in the present setting for simplicity. In the literature, the parameter associated with the Robin condition is sometimes called the impedance or the admittance. Although these quantities differ physically, being reciprocals of one another, they are mathematically interchangeable because the Robin condition may equivalently be expressed as $\alpha^{-1} \partial_{\nn} u + u = 0$. 
For clarity of presentation, the analysis in this work is restricted to the case of a strictly positive constant $\alpha$, while the methodology can also be adapted to more general settings.

The primary goal of this work is to enhance the identification of concave or non-convex portions of the cavity enclosed by $\varGamma$ (see Figure~\ref{fig:reconstruction} for an illustration) by employing a nonstandard shape optimization strategy (Subsection~\ref{subsec:shape_optimization_formulation}) for the associated inverse problem.

\begin{figure}[htbp]
    \centering
    \includegraphics[width=0.5\textwidth]{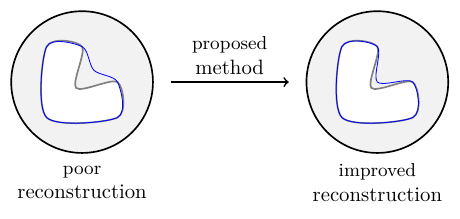}
    \caption{illustration of enhanced concavity detection.}
    \label{fig:reconstruction}
\end{figure}

Problem~\ref{eq:inverse_problem} arises in a broad class of inverse boundary value problems, including the identification of an inaccessible portion of the boundary from thermal or electrostatic measurements collected on an accessible subset $\varSigma \subset \partial\varOmega$.
In the electrostatic setting, the state variable $u$ denotes the electric potential in a conducting domain $\varOmega$, while measurements are restricted to $\varSigma$.
For simplicity, the conductivity in \eqref{eq:gip} and \eqref{eq:state} is assumed to be scalar, homogeneous, and normalized to unity.
The inverse problem is to reconstruct the unknown boundary component $\varGamma$ from the Cauchy data $(u,\partial_{\nn}u)$ prescribed on $\varSigma$.
Related formulations and physical motivations may be found in \cite{AlessandriniDelPieroRondi2003,AlessandriniSincich2007,ChaabaneaJaoua1999,CakoniKress2007,FangLu2004,RabagoAzegami2018}, whereas analytical derivations and qualitative properties of the governing equations are discussed in \cite{FasinoInglese2007,Inglese1997,KaupSantosa1995,KaupSantosaVogelius1996}.
For numerical methods addressing the reconstruction of $\varGamma$ from boundary measurements, assuming that the impedance coefficient $\alpha$ is known a priori, we refer to \cite{AfraitesRabago2025,CakoniKress2007,CakoniKressSchuft2010a,KressRundell2005,Loh1987,Rundell2008,FangLinMa2019,FangZeng2009}.

Identifiability in the inverse Robin problem has been extensively studied. 
Local uniqueness of $\varGamma$ can hold for specific domains \cite{IngleseMariani2004}, but in general, for a fixed constant $\alpha$, a single Cauchy pair on $\varSigma$ may correspond to infinitely many domains, making $\varGamma$ or $\alpha$ indeterminate \cite{CakoniKress2007}. 
Two linearly independent Cauchy pairs, with one positive, guarantee uniqueness of both $\varGamma$ and $\alpha$ \cite{Bacchelli2009,PaganiPeerotti2009}, and stability results for two measurements are available \cite{Sincich2010}.  

For homogeneous Dirichlet or Neumann conditions, a single Cauchy pair uniquely determines the missing boundary \cite[Thm. 2.3]{CakoniKress2007}. 
In particular, for Neumann data on $\varGamma$, the harmonic conjugate of $u$ ensures uniqueness if $f$ is non-constant \cite[Rem. 2.4]{CakoniKress2007}.

Reconstructing the unknown boundary $\varGamma$ from boundary measurements is an ill-posed inverse problem. 
Existing studies \cite{AfraitesRabago2025,RabagoAzegami2018,CaubetDambrineKateb2013} indicate that, even when the Robin coefficient $\alpha$ is known, recovering concave portions of the boundary can be particularly challenging. Classical shape optimization methods, typically based on boundary mismatch or energy-gap functionals \cite{AfraitesRabago2025,RabagoAzegami2018,CaubetDambrineKateb2013,Fang2022}, may exhibit reduced sensitivity to concave features, which can affect the accuracy of the reconstruction, especially with respect to the depth of the concavity.

In this work, we employ a non-standard shape optimization framework based on the coupled complex boundary method (CCBM), first proposed in \cite{Chengetal2014} for inverse source problems. Since its introduction, CCBM has been extended to various inverse and shape reconstruction problems, including inverse conductivity \cite{Gongetal2017}, parameter identification for elliptic equations \cite{ZhengChengGong2020}, geometric inverse source problems \cite{Afraites2022,AfraitesMasnaouiNachaoui2022}, Bernoulli free boundary problems \cite{Rabago2023b}, free surface flows \cite{RabagoNotsu2024}, inverse Cauchy problems for the Stokes system \cite{Ouiassaetal2022}, obstacle reconstruction in viscous fluids \cite{RabagoAfraitesNotsu2025}, tumor localization \cite{Rabago2025}, source recovery in time-fractional PDEs \cite{HriziPrakashNovotny2025}, bioluminescence tomography \cite{WuGongGongZhangZhu2026}, and metal--semiconductor contact reconstruction \cite{AfraitesHadriHriziRabago2026}. These works demonstrate that CCBM provides an effective and flexible framework for shape identification problems, both theoretically and computationally.

The present work is motivated by the difficulty of accurately reconstructing concave or nonconvex portions of an unknown Robin boundary from limited boundary measurements. Numerical investigations in \cite{AfraitesRabago2024,AfraitesRabago2025} considered this geometric inverse problem within classical shape optimization frameworks: the single-measurement setting was studied in \cite{AfraitesRabago2024}, while \cite{AfraitesRabago2025} demonstrated that the use of multiple measurements can significantly improve the recovery of such features. These studies also indicate that least-squares and Kohn--Vogelius-type formulations may suffer from deterioration in reconstruction quality near concave regions, particularly in the single-measurement case.
In contrast, the CCBM framework leads to a more regular and structurally stable state system, thereby providing a favorable setting for variational analysis and numerical approximation. However, to the best of our knowledge, the recovery of Robin boundaries has not yet been investigated within the CCBM framework.
Previous studies have also demonstrated that CCBM-based methods yield reconstruction quality comparable to, and in certain settings more robust than, Kohn--Vogelius-type approaches for free boundary and surface identification problems; see, e.g., \cite{Rabago2023b,RabagoNotsu2024}.

In the present work, however, we focus on the more challenging single-measurement setting. 
Motivated by the aforementioned difficulties in recovering concave boundary features, we incorporate an inequality constraint into the CCBM-based shape optimization formulation in order to improve the reconstruction of such regions. 
The resulting constrained problem is solved using the alternating direction method of multipliers (ADMM), following ideas introduced in \cite{RabagoHadriAfraitesHendyZaky2024} and further developed in \cite{CherratAfraitesRabago2025b,CherratAfraitesRabago2026} for related inverse problems with homogeneous Dirichlet boundary conditions. 
In contrast, the present study considers a Robin boundary condition on the unknown interface, which is known to yield a substantially more challenging reconstruction problem \cite{AfraitesRabago2025}.

\color{black}
The remainder of this paper is organized as follows. 
In Section~\ref{sec:problem_setting}, we introduce the problem setting (equation~\ref{eq:shape_problem}) and the associated shape optimization formulation (equation~\ref{eq:shape_problem}). 
Section~\ref{sec:Shape_Derivatives} is devoted to the shape sensitivity analysis and the derivation of the corresponding shape derivatives (see Prop.~\ref{rem:Jmaps} and Remark~\ref{prop:Jmaps}). 
In Section~\ref{sec:Numerical_Approximation}, we discuss the numerical approximation and implementation of the proposed approach within the CCBM framework. 
Numerical examples are presented to illustrate the difficulty of accurately recovering concave regions, even for relatively large values of the Robin coefficient (see~Figure~\ref{fig:all_reg_radius_0.3}). 
In Section~\ref{sec:inequality_constrained_shape_reconstruction}, we introduce the inequality-constrained formulation (Problem~\ref{prob:optimal_shape_problem}) and present the corresponding ADMM-based algorithm (Algorithm~\ref{algo:ADMM-SGBD}) together with numerical results, with particular emphasis on the case of unit Robin coefficient. 
Finally, Section~\ref{sec:conclusion} summarizes the main findings and concludes the paper.
\medskip

%
%
\section{Problem setting and shape optimization formulation}
\label{sec:problem_setting}
\subsection{The main problem} 
Let $D$ be a ${C}^{1,1}$ smooth, open, bounded planar set\footnote{The extension to three dimensions is straightforward; we restrict to two dimensions for simplicity of presentation.}, and let ${d_{\circ}} > 0$ be fixed.
Define the admissible class of subdomains
\begin{equation}
\label{eq:admissible_domains}
\mathcal{A} \coloneqq \big\{ \omega \Subset D \ \big| \ \omega \text{ is connected, of class } C^{1,1}, \ \operatorname{dist}(\omega,\partial D) > {d_{\circ}}, \text{ and } D \setminus \overline{\omega} \text{ is connected} \big\}.
\end{equation}
We denote $\varOmega \coloneqq D \setminus \overline{\omega}$, $\varSigma \coloneqq \partial D$, and $\varGamma \coloneqq \partial \omega$.
We say that $\varOmega$ is admissible if $\varOmega = D \setminus \overline{\omega}$ for some $\omega \in \mathcal{A}$.
With a slight abuse of notation, we also write $\varOmega \in \mathcal{A}$.

The main problem is then to
\begin{equation}\label{eq:shape_problem}
	\text{find $\omega \in \mathcal{A}$ and $u$ satisfying the overdetermined system \eqref{eq:gip}}.
\end{equation}
We assume $f \in H^{3/2}(\varSigma)$ with $f \not\equiv 0$, and let $g \in H^{1/2}(\varSigma)$ denote the corresponding Neumann measurement, namely
\[
g=\Lambda_{\varSigma}(f)=\partial_{\nn}u,
\]
where $u$ is the solution of \eqref{eq:state}.

\begin{remark}\label{rem:data_regularity}
The above regularity assumptions are stronger than necessary. In fact, it is sufficient to assume that $\varSigma$ is Lipschitz, $f \in H^{1/2}(\varSigma)$, and $g \in H^{-1/2}(\varSigma)$. The higher regularity is adopted only to streamline the presentation and proofs, ensure that the shape derivative of the state variable belongs to $H^1(\varOmega)$, and permit a boundary integral representation of the shape gradient, which is convenient for both analysis and computation.
\end{remark}

%
%
%
\subsection{Shape optimization setting}
\label{subsec:shape_optimization_formulation}
We present the CCBM formulation of \eqref{eq:gip} and its shape optimization reformulation \eqref{eq:shape_problem}, based on a least-squares fitting of the imaginary part of the complex PDE solution.

The approach begins by recasting \eqref{eq:gip} as the complex system
\begin{equation}
\label{eq:ccbm_state}
	-\Delta u = 0 \ \text{in} \ \varOmega, \qquad
	\partial_{\nn} u + i {\rho} u = g + i  {\rho}  f \ \text{on} \ \varSigma, \qquad
	\partial_{\nn} u + \alpha u = 0 \ \text{on} \ \varGamma,
\end{equation}
with $i = \sqrt{-1}$, for some fixed constant ${\rho} > 0$.
The introduction of the free parameter ${\rho}$ is not \textit{strictly} required in the CCBM formulation. 
In the majority of existing works \cite{Rabago2023b,RabagoNotsu2024,RabagoAfraitesNotsu2025,Rabago2025,CherratAfraitesRabago2026}, the choice ${\rho} = 1$ is sufficient for shape identification. 
However, selecting values of ${\rho}$ different from unity may improve the reconstruction results; see, for instance, \cite{AfraitesHadriHriziRabago2026}.
For simplicity, we present the formulation with ${\rho} = 1$ and reintroduce ${\rho}$ in the numerical experiments.

Before presenting the weak form of the complex PDE system \eqref{eq:ccbm_state} and its well-posedness, we first introduce some function space notations.  

Let $W^{m,p}(\varOmega)$ be the standard real Sobolev space with norm $\|\cdot\|_{W^{m,p}(\varOmega)}$, and $W^{0,p}(\varOmega) = L^p(\varOmega)$. In particular, $H^m(\varOmega) = W^{m,2}(\varOmega)$ with inner product $(\cdot,\cdot)_{m,\varOmega}$ and norm $\|\cdot\|_{H^m(\varOmega)}$.  
The complex space $\HH^m(\varOmega)$ has inner product $(\!(u,v)\!)_{m,\varOmega} = (u,\overline{v})_{m,\varOmega}$\footnote{$(\cdot,\cdot)_{m,\varOmega}$ is the real Sobolev inner product extended componentwise.} and norm $\vertiii{v}_{\HH^m(\varOmega)} = \sqrt{(\!(v,v)\!)_{m,\varOmega}}$. Similarly, $Q = L^2(\varOmega)$, $\QQ = \LL^2(\varOmega)$, $S = L^2(\varSigma)$, and $\SSig = \LL^2(\varSigma)$.

We define the sesquilinear form $\aaa$ and linear form $l$ on $\HH^{1}(\varOmega)$ by
\[
\aaa(u,v) = \intO{\nabla u \cdot \nabla \overline{v}} + i \intS{u \overline{v}} + \alpha \intG{u \overline{v}}, \quad 
l(v) = \intS{g \overline{v}} + i \intS{f \overline{v}}, \quad u,v \in \HH^1(\varOmega).
\]

The variational formulation of \eqref{eq:ccbm_state} is then:
\begin{equation}\label{eq:state_weak_form}
\text{find } u \in \HH^{1}(\varOmega) \text{ such that } \aaa(u,v) = l(v), \ \forall v \in \HH^{1}(\varOmega).
\end{equation}

Existence and uniqueness of the solution follow from the complex Lax-Milgram lemma \cite[p. 376]{DautrayLionsv21998} (see also \cite[Lem. 2.1.51, p. 40]{SauterSchwab2011}):  

\begin{lemma}
For given $f \in H^{1/2}(\varSigma)$ and $g \in (H^{1/2}(\varSigma))'$, there exists a unique weak solution $u \in \HH^{1}(\varOmega)$ of \eqref{eq:state_weak_form}.
\end{lemma} 
Writing $u = \ur + i \ui$, the real and imaginary parts satisfy
\begin{align}
&-\Delta \ur = 0 \ \text{in} \ \varOmega, \qquad
\partial_{\nn} \ur = \ui + g \ \text{on} \ \varSigma, \qquad
\partial_{\nn} \ur + \alpha \ur = 0 \ \text{on} \ \varGamma; \label{eq:real_state} \\
&-\Delta \ui = 0 \ \text{in} \ \varOmega, \qquad
\partial_{\nn} \ui = -\ur + f \ \text{on} \ \varSigma, \qquad
\partial_{\nn} \ui + \alpha \ui = 0 \ \text{on} \ \varGamma. \label{eq:imaginary_state}
\end{align}
To reformulate the original setting in terms of shape optimization, observe that if $\ui = 0$ in $\varOmega$, then continuity and harmonicity imply $\ui = 0$ on $\varSigma$ and $\varGamma$, and by Green's formula, $\partial_{\nn} \ui = 0$ on $\varSigma$.
Hence, $(\varOmega, \ur)$ solves the original problem \eqref{eq:gip}. Conversely, any solution $(\varOmega,u)$ of \eqref{eq:gip} satisfies \eqref{eq:real_state} and \eqref{eq:imaginary_state}. 
Therefore, the inverse shape problem is equivalent to the following:
\begin{problem}
\label{eq:main_problem}
Find $\omega \in \mathcal{A}$ such that 
\[
\ui = 0 \quad \text{in} \ \varOmega = D \setminus \overline{\omega},
\]
where $u = \ur + i \ui$ solves the well-posed complex PDE system \eqref{eq:ccbm_state}.
\end{problem}
To solve Problem~\ref{eq:main_problem}, we consider the problem
\begin{equation}\label{eq:shape_problem}
\inf_{\varOmega \in \mathcal{A}} \; \left\{
J(\varOmega)
= \frac12 \intO{|\ui|^2} \right\},
\end{equation}
where $\ui = \Im\{u\}$ and $u$ solves \eqref{eq:ccbm_state}, with $\mathcal{A}$ defined in \eqref{eq:admissible_domains}.

Since $\varSigma$ is fixed, only $\varGamma$ is perturbed; hence, Lipschitz regularity of $\varSigma$ suffices to derive shape derivatives in distributed form.
To obtain a boundary integral form of the shape derivative of the functional in \eqref{eq:shape_problem}, in accordance with the Hadamard--Zol\'{e}sio structure theorem \cite[Thm.~3.6, p.~479]{DelfourZolesio2011}, we require higher regularity assumptions on the data and on the annular domain $\varOmega$ (recall Remark~\ref{rem:data_regularity} and \eqref{eq:admissible_domains}).

The minimization problem \eqref{eq:shape_problem} is solved numerically using a gradient-based method combined with the finite element method (FEM), requiring the computation of the shape derivative of the cost functional.
This derivative will be derived in the next section using shape calculus \cite{DelfourZolesio2011,HenrotPierre2018,MuratSimon1976,Simon1980,SokolowskiZolesio1992}.

\section{Shape sensitivity analysis}
\label{sec:Shape_Derivatives}
For completeness, we briefly recall some preliminary concepts related to shape derivatives from shape calculus; see \cite{DelfourZolesio2011,HenrotPierre2018,SokolowskiZolesio1992}.
\subsection{Some elements of Shape Calculus}\label{subsec:elements_of_shape_calculus}
	Let $D_{d_{\circ}}$ be an open set with a ${C}^{\infty}$ boundary, such that $\{ x \in D \mid \text{$d(x,\partial D) > {d_{\circ}}/2$}\} \subset D_{{{d_{\circ}}}} \subset \{ x\in D \mid \text{$d(x,\partial D) > {d_{\circ}}/3$}\}$. 
	Let $\VV \in {C}^{k,1}(\mathbb{R}^{d})$, $k \in \mathbb{N}$, be a smooth vector field with compact support in $\overline{D}_{d_{\circ}}$ and let $\sfTheta^k$ denote the set of all such admissible deformations.
We denote the normal component by $\Vn \coloneqq \langle \VV, \nn \rangle = \VV \cdot \nn$, and let $\Id$ be the identity matrix in $\mathbb{R}^{d}$.

For small $t_0 > 0$, define the perturbation map
\[
T_t = \Id + t \VV, \qquad t \in \intervalI \coloneqq [0,t_0), \qquad \VV \in \sfTheta^k,
\]
so that $\varOmega_t = T_t(\varOmega)$, $\varGamma_t = T_t(\varGamma)$, and $\varSigma_t = \varSigma$ for all $t$ since $\operatorname{supp}(\VV) \subset \overline{\varOmega}_{d_{\circ}}$.  

Set
\[
A_t \coloneqq {\Jact} ({D}T_t^{-1})({D}T_t)^{-\top}, \qquad 
B_t \coloneqq {\Jact} |({D}T_t)^{-\top} \nn|, \qquad {\Jact} \coloneqq \det \,{D}T_t > 0,
\]
with bounds $0<{b}_1 \leqslant B_t \leqslant {b}_2$ and $0<{b}_3 |\xi|^2 \leqslant A_t \xi \cdot \xi \leqslant {b}_4 |\xi|^2$ for all $\xi \in \mathbb{R}^d$.

We assume that the interval $\intervalI$ is sufficiently small so that, for all $t \in \intervalI$, the mappings
\begin{equation}\label{eq:continuity_of_the_maps}
[t \mapsto {\Jact}] \in {C}^1(\intervalI,{C}(\overline{\varOmega})), \qquad
[t \mapsto A_t] \in {C}^1(\intervalI,{C}(\overline{\varOmega})^{d\times d}),\qquad
[t \mapsto B_t] \in {C}^1(\intervalI,{C}(\partial{\varOmega}))
\end{equation}
are well defined and possess the stated regularity.

At $t=0$, $j_0 = 1$, $A_0 = \Id$, $B_0 = 1$, with derivatives
\begin{equation}\label{eq:derivatives}
\frac{d}{dt} {\Jact} \big|_{0} = \dive \VV, \qquad 
\frac{d}{dt} A_t \big|_{0} = (\dive \VV)\Id - D\VV - (D\VV)^\top \eqqcolon A, \qquad
\frac{d}{dt} B_t \big|_{0} = \dive_\varGamma \VV.
\end{equation}

Let $u_t$ be the state defined on $\varOmega_t$. 
The material derivative of $u$ in the direction $\VV$ is defined as
\[
\dot{u} \coloneqq \lim_{t\searrow 0} \frac{u^t - u}{t},
\quad \text{where } u^t \coloneqq u_t \circ T_t.
\]

For a given $\VV \in \sfTheta^{1}$, the Eulerian semiderivative of $\sfj(\varOmega)$ in the direction of $\VV$ is defined as
\[
d\sfj(\varOmega)[\VV] \coloneqq \lim_{t\searrow0} \frac{\sfj(\varOmega_t) - \sfj(\varOmega)}{t}.
\]

If $d\sfj(\varOmega)[\VV]$ exists for all $\VV \in \sfTheta^{1}$ and the map $\VV \mapsto d\sfj(\varOmega)[\VV]$ is linear and continuous on $\sfTheta^{1}$, then $\sfj$ is said to be shape differentiable at $\varOmega$, and $d\sfj(\varOmega)[\VV]$ is called the \textit{shape derivative}.
Moreover, under suitable regularity assumptions, the shape derivative admits the boundary representation
\[
d\sfj(\varOmega)[\VV] = \int_{\partial \varOmega} G \, \nn \cdot \VV\, ds,
\]
and $G$ is called the \textit{shape gradient} of $J$ at $\varOmega$ \cite[Def.~3.4, p.~479]{DelfourZolesio2011}.
\subsection{Computation of the shape gradient}
\label{subsec:shape_derivative_of_the_cost_using_the_Eulerian_derivatives}
In this section, we compute the shape derivative of $J$ (Proposition~\ref{prop:Jmaps}) using the rearrangement method \cite{IKP2008}. 
This approach derives the shape gradient via the material derivative of the state without requiring its explicit form (cf. \cite{Rabago2023b}) and only assumes H\"older, rather than ${C}^{1}$, continuity of $t \mapsto u^{t}$ near $0$.

Hereinafter, we fix $k = 1$ in $\sfTheta^{k}$, and $\varOmega$ and $\theta$ are always assumed admissible; that is, $\varOmega = D \setminus \overline{\omega}$ with $\omega \in \mathcal{A}$ and $\theta \in \sfTheta^{1}$, without further notice, and this will usually be omitted for brevity.
Moreover, for economy of space, we adopt, for any complex-valued functions $\varphi, \psi \in \HH^{2}(\varOmega)$, the following notations:\footnote{Here $\tannabla$ denotes the tangential gradient operator on $\varGamma$. An intrinsic definition can be found, for instance, in \cite[Chap.~5, Sec.~5.1, p.~492]{DelfourZolesio2011}.}
\begin{align*}
\Psi_{\varGamma}(\varphi, \psi) 
&= \big( \tannabla \varphi_{1} \cdot \tannabla \psi_{2} 
- \tannabla \varphi_{2} \cdot \tannabla \psi_{1} \big)
+ (\alpha^{2} - \alpha \kappa)(\psi_{1}\, \varphi_{2} - \psi_{2}\, \varphi_{1});\\
\Psi^{\dagger}_{\varGamma}(\varphi, \psi) 
&= \big( \tannabla \varphi_{1} \cdot \tannabla \psi_{1} 
+ \tannabla \varphi_{2} \cdot \tannabla \psi_{2} \big)
- (\alpha^{2} - \alpha \kappa)(\varphi_{1}\, \psi_{1} + \varphi_{2}\, \psi_{2}).
\end{align*}
%

The main result of this subsection is as follows:
\begin{proposition}
\label{prop:Jmaps}
The functional $J$ is shape differentiable, and its shape derivative in the direction $\VV$ is given by
\begin{equation}\label{eq:shape_gradient}
{d}J(\varOmega)[\VV] 
= \intG{ \GG_{1}(u,p) \, \nn \cdot \VV }
\coloneqq \intG{ \left( \frac{1}{2} |\ui|^2 
+
\Psi_{\varGamma}(u, p) \right)\nn \cdot \VV }.
\end{equation}
Here, $\kappa$ denotes the mean curvature of $\varGamma$, $u = \ur + i \ui$ is the unique solution to~\eqref{eq:ccbm_state}, and $p = \vr + i \vi$ is the unique solution of the adjoint system
\begin{equation}\label{eq:adjoint_system}
-\Delta p = \ui \quad \text{in } \varOmega, 
\qquad
\dn{p} - i p = 0 \quad \text{on } \varSigma, 
\qquad
\dn{p} + \alpha p = 0 \quad \text{on } \varGamma.
\end{equation}
\end{proposition}
The weak formulation of \eqref{eq:adjoint_system} reads as follows:	
\begin{equation}\label{eq:adjoint_equation}
	\text{	find $p \in \HH^{1}(\varOmega)$ such that} \ \ \intO{\nabla p \cdot \nabla \overline{\varphi}} - i \intS{p \overline{\varphi}} + \alpha \intG{p \overline{\varphi}} = \intO{ \ui {\cphi} }, \ \ \forall \varphi \in \HH^{1}(\varOmega).
\end{equation}
The existence and uniqueness of a solution to \eqref{eq:adjoint_equation} is again a consequence of the Lax–Milgram lemma.
Note that, given the regularity assumptions on the data and the domain, the state and adjoint solutions enjoy higher regularity; that is, it can be shown that $u \in \HH^{2}(\varOmega)$ and, similarly, $p \in \HH^{2}(\varOmega)$.

The proof of Proposition \ref{prop:Jmaps} relies on the following lemmas, which follow from standard arguments (see, e.g., \cite{Rabago2023b}) and are therefore omitted.
\begin{lemma}\label{lem:sesquiliner_form}
	The sesquilinear form $\aat$ defined on $\HH^{1}(\varOmega) \times \HH^{1}(\varOmega)$ by
	\[
		\aat({u},{v}) \coloneqq \intO{A_t \nabla {u} \cdot \nabla {\cv}} + i \intS{{u} {\cv}} + \alpha \intG{B_t {u} {\cv}}, \quad \forall {u}, {v} \in \HH^{1}(\varOmega),
	\]
	is bounded and coercive on $\HH^{1}(\varOmega) \times \HH^{1}(\varOmega)$ for $t \in \intervalI$.
\end{lemma} 
\begin{lemma}
	\label{lem:transported_problem}
	The function $u^{t}=u_{1}^{t} + i \ui^{t}$ uniquely solves in $\HH^{1}(\varOmega)$ the equation
	\begin{equation}\label{eq:transformed_state}
		\aat({u^{t}},{v}) = l({v}), \quad \forall {v} \in \HH^{1}(\varOmega).	
	\end{equation}
\end{lemma}	
\begin{lemma}\label{lem:holder_continuity}
	The solution ${u^{t}}$ of \eqref{eq:transformed_state} is (uniformly) bounded in the neighborhood of $0$ and 
	\[
		\vertiii{u^{t} - u}_{\HH^{1}(\varOmega)} = O(t).
	\] 
\end{lemma}
\begin{proof}
	It is clear that, with $v = u^{t}$ in \eqref{eq:transformed_state} and by Lemma \ref{lem:sesquiliner_form}, $\vertiii{u^{t}}_{\HH^{1}(\varOmega)}$ is bounded for all $t\in \intervalI$.
	Meanwhile, the variational equation $\aat(u^{t},v) - \aaa(u,v) = 0$, for all ${v} \in \HH^{1}(\varOmega)$, is equivalent to
	%
	\begin{equation}\label{eq:difference_equation}
	\aaa(u^{t} - u, v) = - \intO{(A_{t} - \Id) \nabla {u^{t}} \cdot \nabla {\cv}} - \alpha \intG{ (B_{t} - 1) {u^{t}} {\cv}},\quad \forall {v} \in \HH^{1}(\varOmega).
	\end{equation}
	By taking $v = u^{t} - u$ above, and then applying the continuity of the maps given in \eqref{eq:continuity_of_the_maps}, it can be shown without difficulty that the following limit holds
	\[
		\lim_{t \searrow 0} \frac{1}{t} \vertiii{u^{t} - u}_{\HH^{1}(\varOmega)}^{2} = 0,	
	\]
	which verifies the assertion.
\end{proof}
\begin{remark}
Note that, by dividing both sides of \eqref{eq:difference_equation} by $t$, and then letting $t \searrow 0$, one recovers the equation (in variational form) satisfied by the material derivative $\dotu$ of the state variable $u$.
\end{remark}
Finally, to complete our preparations, we recall the following identity (see, e.g., \cite{DelfourZolesio2011}) which is needed to write the shape derivative of $J$ in terms of a boundary integral.
The $\HH^2(\varOmega)$ regularity of ${u}$, which holds true since $\varOmega$ is assumed to be of class ${C}^{1,1}$ and $\VV \in \sfTheta^1$, will again be used subsequently without further notice.	
\begin{lemma}
\label{eq:integral_equivalence}
The state solution $u$ of \eqref{eq:ccbm_state} and the adjoint solution $p$ to \eqref{eq:adjoint_system} satisfy the following identity
	\begin{equation} \label{eq:domain_identity}
	\begin{aligned}
	\intO{A \nabla {u} \cdot \nabla \overline{p}}
	&= - \intO{ \ui \VV \cdot \nabla {u} } 
		+ \alpha \intG{ \left( {u}(\VV \cdot \nabla \overline{{p}}) + \overline{p} (\VV \cdot \nabla {u}) \right) } 
		+ \intG{(\nabla \overline{{p}} \cdot \nabla {u})\Vn}.
	\end{aligned}
	\end{equation}
\end{lemma}
%
%
%
Now we prove Proposition \ref{prop:Jmaps} as follows.
\begin{proof}[Proof of Proposition \ref{prop:Jmaps}]
	The proof proceeds in two steps: first, we derive an expression for the shape derivative of $J$ involving domain and boundary integrals, and then we express the whole expression in terms of just a boundary integral. 

	\underline{Step 1.} We consider the difference
	\begin{align*}
		J(\varOmega_{t}) - J(\varOmega) 
		= \frac12 \intO{ ({\Jact} - 1) |\ui^{t}|^{2} } + \frac12 \intO{ (|\ui^{t}|^{2} - |\ui|^{2})}
		\eqqcolon I_{1}(t) + I_{2}(t),
	\end{align*}
	and then look at the limit of $\dfrac{1}{t}\left(I_{1}(t) + I_{2}(t)\right)$ as $t \searrow 0$.

	For the limit of $I_{1}(t)/t$ as $t \searrow 0$, we have
	\begin{equation}\label{expression1}
		\frac{I_{1}(t)}{t}  \longrightarrow \frac12 \intO{ \dive \VV |\ui|^{2} } \eqqcolon I_{1} \quad \text{as $t \longrightarrow 0$}.
	\end{equation}
	Meanwhile, to get the limit of $I_{2}(t)/t$, let us consider
	\[
		I_{2}(t) = \frac12 \intO{ |\ui^{t} - \ui|^{2}} + \intO{ \ui (\ui^{t} - \ui)}
			\eqqcolon I_{21}(t) + I_{22}(t).
	\]
	Using Lemma \ref{lem:holder_continuity}, we see that $I_{21}(t)/t \to 0$ as $t \searrow 0$.
	On the other hand, to get the limit of $I_{22}(t)/t$ as $t \searrow 0$, we make use of \eqref{eq:adjoint_equation} with $\varphi = u^{t} - u$ and \eqref{eq:difference_equation} with $v = p$ which are given as follows: 
	\[
	\begin{aligned}
		&\intO{\nabla p \cdot \nabla \overline{(u^{t} - u)}} - i \intS{p \overline{(u^{t} - u)}} + \alpha \intG{p \overline{(u^{t} - u)}} = \intO{ \ui \overline{(u^{t} - u)} }\\
		\quad& \Longleftrightarrow \quad 
		\intO{ \ui (u^{t} - u) }
		= \aaa(u^{t} - u,{p}) 
		= - \intO{(A_{t} - \Id) \nabla {u^{t}} \cdot \nabla \overline{p}} - \alpha \intG{ (B_{t} - 1) {u^{t}} \overline{p}}.	
	\end{aligned}
	\]
	After dividing the latter equation by $t$ and taking the imaginary part, we obtain the following identity:
	\[
	\frac{I_{22}(t)}{t} = \Im \left\{ \intO{ \ui \left(\frac{u^{t} - u}{t}\right) }  \right\}
	= \Im\left\{ - \intO{\left(\frac{A_{t} - \Id}{t}\right) \nabla {u^{t}} \cdot \nabla \overline{p}} - \alpha \intG{ \left(\frac{B_{t} - 1}{t}\right) {u^{t}} \overline{p}} \right\}.
	\]

	Thus, using \eqref{eq:derivatives}, we get
	\begin{equation}\label{expression2}
		\frac{I_{22}(t)}{t} \longrightarrow \Im\left\{ -\intO{A \nabla {u} \cdot \nabla \overline{p}} - \alpha \intG{ \tandive\! \VV {u} \overline{p}} \right\} \eqqcolon I_{2}\quad \text{as $t \longrightarrow 0$}.
	\end{equation}
	Putting \eqref{expression1} and \eqref{expression2} together, we get an expression for the shape derivative of $J$
	\[
		{d}J(\varOmega)[\VV]
			= I_{1} + I_{2}
			= \frac12 \intO{ \dive \VV |\ui|^{2} } + \Im\left\{ \intO{-A \nabla {u} \cdot \nabla \overline{p}} - \alpha \intG{ \tandive\! \VV {u} \overline{p}} \right\}.
	\] 

	\underline{Step 2.} We simplify the above shape derivative by expressing it solely as a boundary integral over $\varGamma$.
	First, we recall a version of the tangential Green's formula,\footnote{A proof of this formula can be found in \cite{MuratSimon1976}.} which is valid, for instance, when $\varGamma \in {C}^{1,1}$ and for a function $\phi \in W^{2,1}(D)$, given as follows
\begin{equation}
	\label{eq:tangential_Greens_formula}
	\intG{\left( \nabla \phi \cdot \VV + \phi \tandive \VV \right) }
		=  \intG{ \left( \dn{\phi}  +\phi \kappa \right) {\Vn}},
\end{equation} 
where $\kappa = \tandive \nn$ denotes the mean curvature of $\varGamma$.

Taking $\phi = u \overline{p}$ in equation \eqref{eq:tangential_Greens_formula}, 
it can be verified that
\[
		\alpha\intG{\left(  \overline{p} (\nabla u \cdot \VV) + u ( \nabla \overline{p} \cdot \VV \right)} + \alpha\intG{u \overline{p} \tandive \VV } 
		=\intG{ \left( -2\alpha^2 + \alpha\kappa \right) u \overline{p} {\Vn}}.
\]
In addition, we have the following identities
\begin{align*}
\frac12 \intO{  \dive \VV |\ui|^{2} }
	&= - \intO{ \ui \VV \cdot \nabla \ui } + \frac12 \intG{ |\ui|^{2} \Vn },\\
\intG{(\nabla \overline{{p}} \cdot \nabla {u})\Vn}
	&= \intG{(\tannabla \overline{{p}} \cdot \tannabla {u} + \alpha^2 u \overline{p} )\Vn}.
\end{align*}
Now, using the above identities together with \eqref{eq:domain_identity}, we get
%
	\begin{align*}
	I_{1} + I_{2}
	& = \frac12 \intG{ |\ui|^{2} \Vn } + \intG{\Im\left\{ (\alpha^2-\alpha \kappa)u\overline{p}-\tannabla \overline{{p}} \cdot \tannabla {u} \right\} \Vn}.
	\end{align*}

	Simplifying the second integral above gives the desired expression for the shape gradient $G$ given in \eqref{eq:shape_gradient}.
	This proves the proposition.
\end{proof} 
%
%
%
We record an alternative expression of the shape derivative in the following remark.
\begin{remark}
	\label{rem:Jmaps}
	Consider the adjoint variable $p = \vr + i \vi$ defined as the unique solution to
	\begin{equation}\label{eq:adjoint_system_Re}
		-\Delta p 		=i\ui \ \text{in $\varOmega$},\qquad\quad
		\dn{p} - i p 	=0 \ \text{on $\varSigma$},\qquad\quad
		\dn{p} + \alpha p=0 \ \text{on $\varGamma$}.
	\end{equation}
	This choice of adjoint yields the alternative boundary representation
	\begin{equation}
	\label{eq:GG_correct}
	{d}J(\varOmega)[\VV] 
	= \intG{\GG_2 \, \nn \cdot \VV}
	\coloneqq \intG{ \left( \frac12 |\ui|^2  -  \Psi^{\dagger}_{\varGamma}(u,p) \right) \nn \cdot \VV},
	\end{equation}
	with $u = \ur + i \ui$ denoting the unique solution to \eqref{eq:ccbm_state}.
\end{remark}

The next conclusion can be drawn easily from \eqref{eq:shape_gradient} and \eqref{eq:adjoint_system}.
\begin{corollary}[Necessary condition]\label{cor:necessary_condition}
	Let the domain $\varOmega^{\star}$ be such that $u=u(\varOmega^{\star})$ satisfies equation \eqref{eq:gip}, i.e., there holds $\ui = 0$ in $\varOmega^{\star}$ or equivalently, $u(\varOmega^{\star})|_{\varSigma} = f$ and $\dn{}u(\varOmega^{\star})|_{\varSigma} = g$.
	Then, $\varOmega^{\star}$ is a stationary solution of \eqref{eq:shape_problem}.
	That is, it fulfills the necessary optimality condition
	\begin{equation}
	\label{eq:optimality_condition}
		{d}J(\varOmega^{\star})[\VV] = 0, \quad \text{for all $\VV \in \sfTheta^{1}$}.
	\end{equation}
\end{corollary}
\begin{proof}
By the assumption that $\ui = 0$ on $\varOmega^{\star}$, we find that $p = p(\varOmega^{\star})$ on $\varOmega^{\star}$. Thus, it follows that $\GG \equiv 0$ on $\varSigma^{\star}$, which implies ${d}J(\varOmega^{\star})[\VV] = 0$ for any $\VV \in \sfTheta^1$.
\end{proof}
\begin{remark}
	We note that solutions of the necessary condition \eqref{eq:optimality_condition} might exist such that the state does not satisfy equation $\ui = 0$ in $\varOmega^{\star}$.
	However, only in the case of exact matching of boundary data a stationary domain $\varOmega^{\star}$ is a global minimum because $J(\varOmega^{\star})=0$.
\end{remark}

\section{Numerical approximation}
\label{sec:Numerical_Approximation}

The proposed shape optimization approach for \eqref{eq:gip} is implemented using a Sobolev gradient-based finite element method, following our previous work \cite{AfraitesRabago2025,CherratAfraitesRabago2025b,CherratAfraitesRabago2025}. For clarity, we briefly summarize the numerical approach.

A natural descent direction for $J$ is $\VV = -G\nn$ with $G \in L^2(\varGamma)$, $G \not\equiv 0$. However, this can lead to unstable reconstructions and poor mesh quality. To mitigate this, we employ an $H^1$-Riesz representation of the shape gradient \cite{Doganetal2007,Azegami2020}, seeking $\VV \in H_{\varSigma,0}^{1}(\varOmega)^d$ such that
\begin{equation}\label{eq:Sobolev_gradient_computation}
	\beta \intO{ \nabla \VV : \nabla \vect{\varphi}} + (1 - \beta) \intG{ \tannabla \VV : \tannabla \vect{\varphi} }
	= - \intG{ G \nn \cdot \vect{\varphi} }, \quad \forall \vect{\varphi} \in H_{\varSigma,0}^{1}(\varOmega)^d,
\end{equation}
where $\beta \in (0,1]$ is a regularization parameter. 
This construction smoothly extends the boundary deformation field $-G\nn$ into $\varOmega$, while the tangential term improves mesh regularity and suppresses spurious oscillations along the evolving free boundary \cite{Neuberger1997}. 
In particular, for nonzero values of $\beta$, the resulting boundary evolution is observed to be significantly smoother, especially near highly curved or concave regions. 
In this sense, the proposed deformation strategy yields an implicit regularization effect similar to perimeter penalization techniques commonly employed in shape optimization, such as the addition of a term of the form
\[
	\eta \operatorname{Per}(\partial\varOmega) \coloneqq \eta \int_{\partial\varOmega} 1 \, ds,
\qquad \eta > 0.
\]
Although such perimeter regularization may also be incorporated into the objective functional (for instance, in the form $J(\varOmega) + \eta \operatorname{Per}(\partial\varOmega)$), the numerical experiments presented in this work suggest that the proposed smoothing mechanism already provides adequate stabilization for the reconstruction problems under consideration (Remark~\ref{rem:no_need_for_perimeter_functional}).

The $k$th domain approximation $\varOmega^k$ is computed as follows:
\begin{enumerate}[label=\arabic*.]
    \item \textit{Initialization}: select an initial shape $\varOmega^0$.
    \item \textit{Iteration} ($k = 0,1,2,\dots, M$):
    \begin{enumerate}[label*=\arabic*.]
        \item Solve the state and adjoint systems on $\varOmega^k$.
        \item Compute the update vector $\VV^k$ and step size $t^k = \mu J(\varOmega^k)/|\VV^k|_{H^1(\varOmega^k)^d}^2$ (see \cite{RabagoAzegami2020}), with $\mu = 2$ and reduced if necessary to avoid mesh inversion.
        \item Update the domain: $\varOmega^{k+1} = (\mathrm{id} + t^k \VV^k) \varOmega^k$.
    \end{enumerate}
    \item \textit{Stop Test}: Repeat \textit{Iteration} until convergence (i.e., terminate after $M=200$ iterations).
\end{enumerate}
\begin{remark}\label{rem:no_need_for_perimeter_functional}
The presented iterative scheme naturally extends to handle noisy data. In the experiments that follow, we demonstrate that it is sufficient to achieve accurate reconstructions under moderate noise. Although regularization terms--such as perimeter penalization--can be incorporated, they do not appear necessary in our case.
\end{remark}
All computations were performed in \textsc{FreeFem++} \cite{Hecht2012} on a MacBook Pro with Apple M1 chip and 16GB RAM.

To test robustness against noise, we define $u^\delta = (1 + \delta\,\text{g.n.})\, u^\star$, where $u^\star$ is the exact solution for input $g$ and ``g.n.'' is Gaussian noise with zero mean and standard deviation $0.5$.  
The corresponding noisy measurement is then $g|_{\varSigma} \coloneqq \dn{u^\delta}$.
\subsection{Numerical examples and discussion}
\label{subsec:numerical_examples} 
We set $\alpha = 100$, take $f = 1$ on $\varSigma$, and consider a circular domain of unit radius. 
Four cavity shapes are examined: E = ellipse, K = kite, S = square, and L = L-block. Figures~\ref{fig:all_reg_radius_0.3}--\ref{fig:kite_adjoint_real_imag_beta08} present the reconstruction and field approximation results. 
In all cases, unless otherwise stated, $\rho = 1$ in \eqref{eq:ccbm_state}.
The exterior boundary $\varSigma$ is shown as a solid (black or blue) line, the exact interior boundary $\varGamma^{\star}$ as a solid magenta line, and the initial guess $\varGamma^{0}=C(x_0,y_0,r)$ as a black dotted circle (Figure~\ref{fig:all_reg_radius_0.3} only), where $C(x_0,y_0,r)$ is a circle of center $(x_0,y_0)$ and radius $r$.
Reconstructed shapes for noise levels $\delta = 0,\,0.03,\,0.06,$ and $0.09$ are plotted with dashed curves of distinct colors.

Figure~\ref{fig:all_reg_radius_0.3} illustrates the effect of the parameter $\beta$ by comparing reconstructions for $\beta = 0$ and $\beta = 0.8$.
In the noise-free case, the reconstructions closely match the exact shapes. As the noise level increases, the method remains stable, with mild smoothing near corners, particularly for S and L. 
Larger values of $\beta$ suppress oscillations and improve stability, consistent with its regularizing effect. 
However, increasing $\beta$ promotes convexity in the reconstructed shapes, which is advantageous for convex cavities but may be detrimental for non-convex geometries.

Figure~\ref{fig:cost_grad_history} illustrates the convergence of the gradient-based updates in terms of the cost functional and gradient norm. 
Although not shown for all cases, convergence is faster for smooth geometries (E and K) and slower for non-smooth ones (S and L), which violate the regularity assumptions on the unknown boundary $\varGamma$. 
In our experiments, larger values of $\beta$ yield smoother convergence histories.

Figure~\ref{fig:ellipse_state_real_imag_comparison} shows the distribution of the real and imaginary parts of the state solution \eqref{eq:ccbm_state} on the cavity boundary at the final iterate $k=200$ for the elliptical cavity. 
The real component converges to a stable distribution, while the imaginary component tends to zero almost everywhere.

Figures~\ref{fig:kite_state_real_imag_beta08} and~\ref{fig:kite_adjoint_real_imag_beta08} illustrate the evolution of the real and imaginary parts of the state and adjoint solutions on the cavity boundary at selected iterations $k=0,50,200$. 
The imaginary component remains small but nonzero in all cases, reflecting numerical effects while preserving physical consistency.

Moreover, Figure~\ref{fig:kite_adjoint_real_imag_beta08} indicates that both components of the adjoint variable decrease toward small magnitudes as the optimization proceeds. 
This behavior is consistent with the first-order necessary optimality condition in Corollary~\ref{cor:necessary_condition}, up to discretization and noise effects, and indicates that the adjoint solution approaches zero near optimality.

In summary, the method yields stable and accurate cavity reconstructions under varying regularization and noise levels for the given magnitude of the Robin coefficient $\alpha$. 
The decay of the adjoint field and the small imaginary component of the state solution are consistent with the theoretical optimality conditions (Corollary~\ref{cor:necessary_condition}).
\begin{figure}[h!]
    \centering
    \resizebox{0.23\textwidth}{!}{\includegraphics{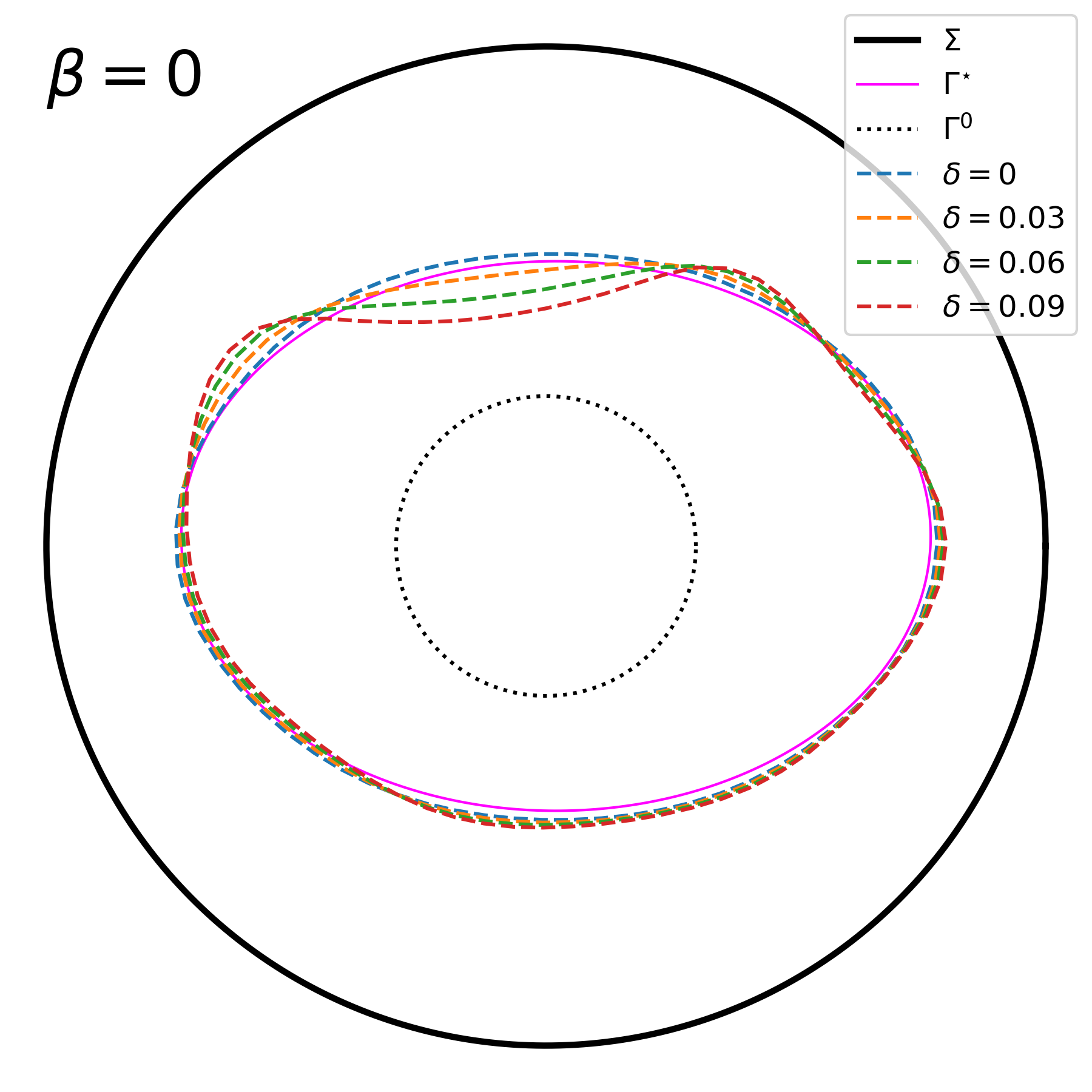}}\ 
    \resizebox{0.23\textwidth}{!}{\includegraphics{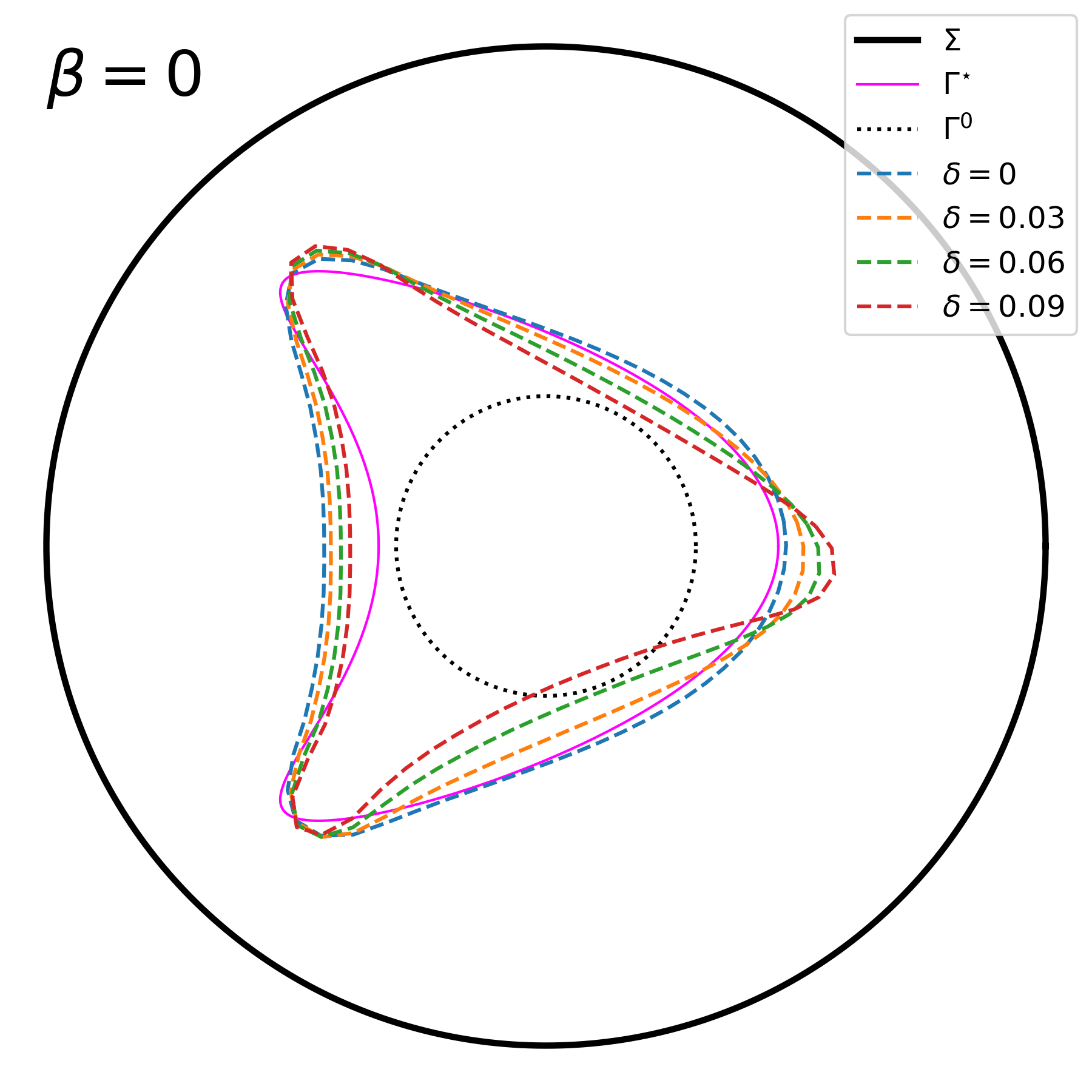}}\ 
    \resizebox{0.23\textwidth}{!}{\includegraphics{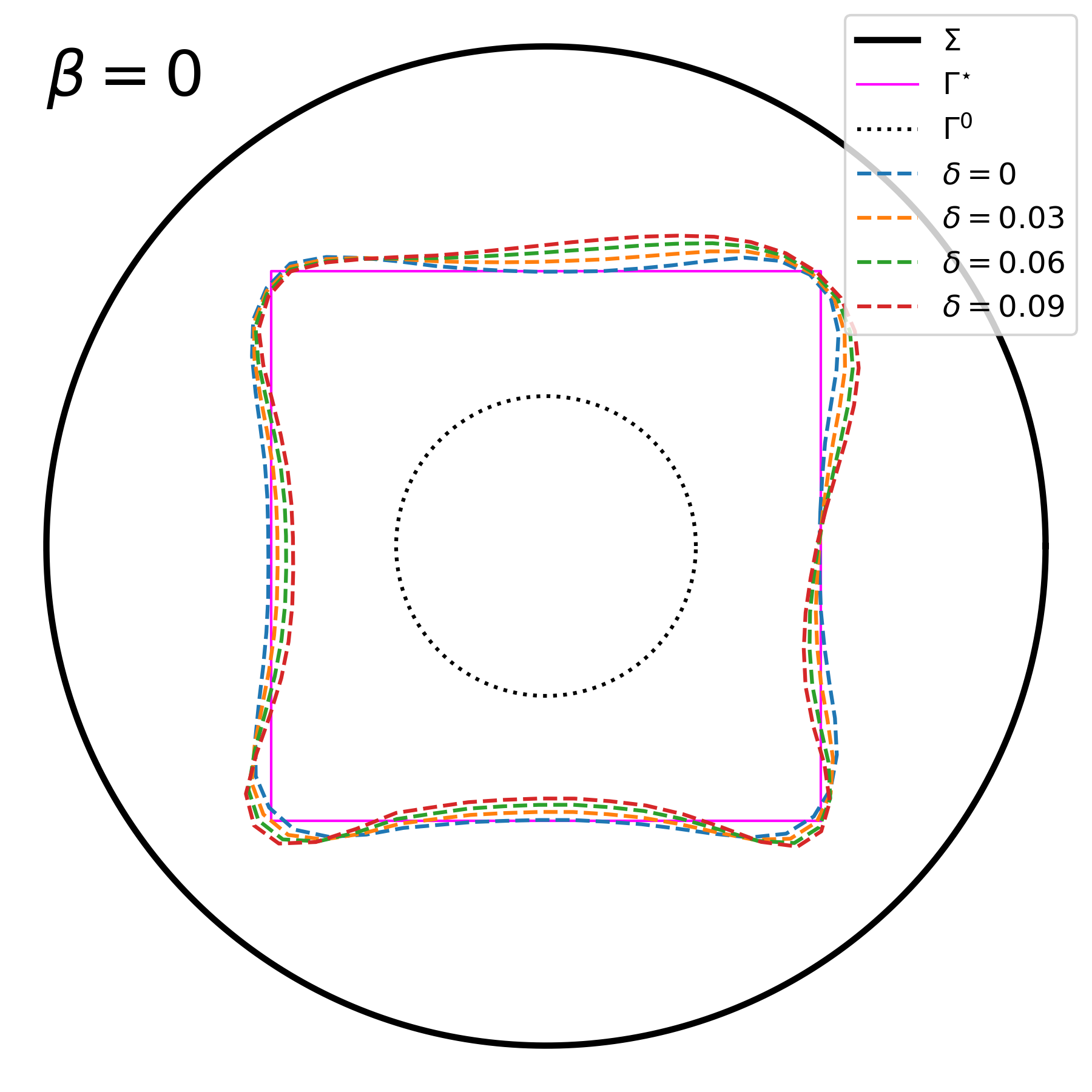}}\ 
    \resizebox{0.23\textwidth}{!}{\includegraphics{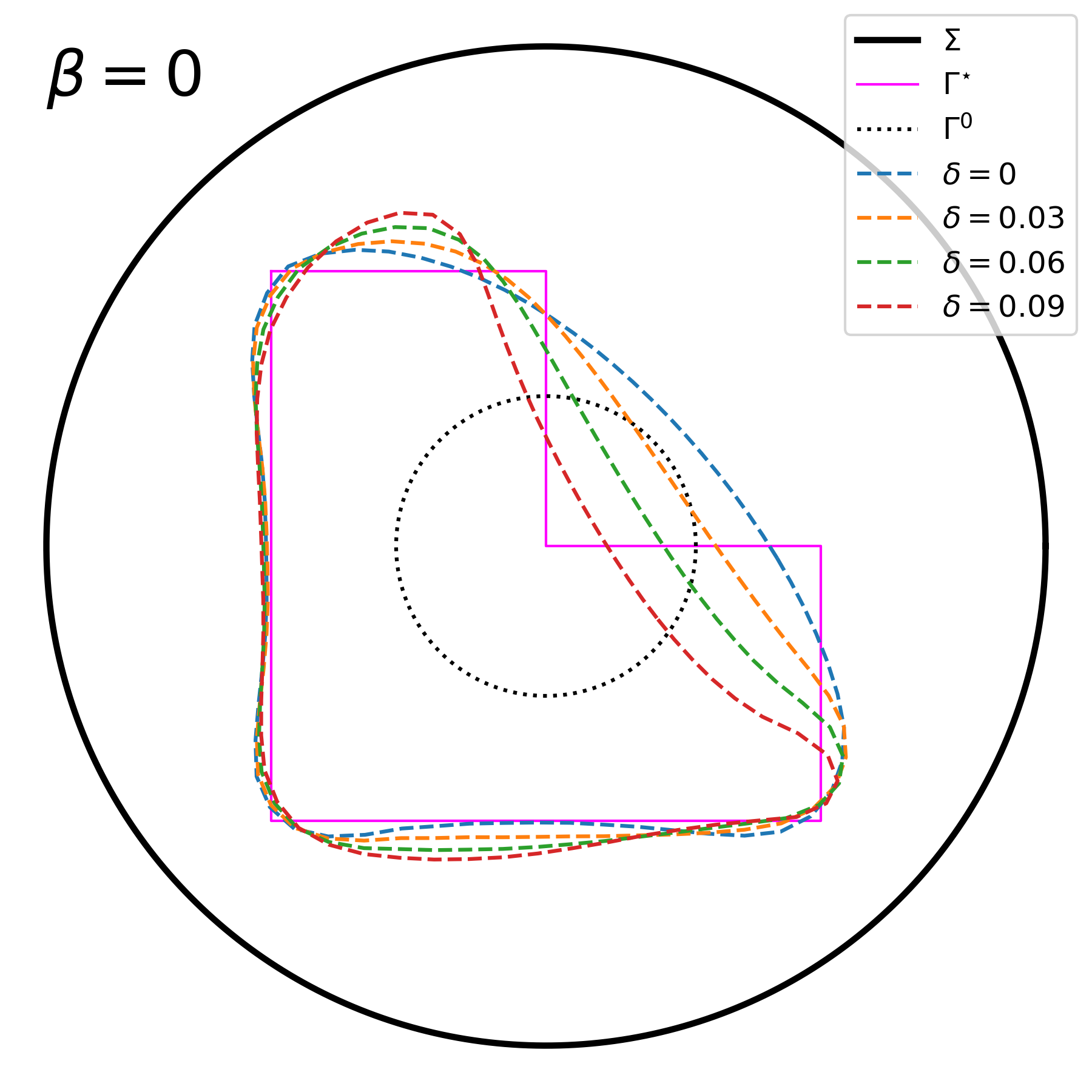}}\\[4pt]
    
    \resizebox{0.23\textwidth}{!}{\includegraphics{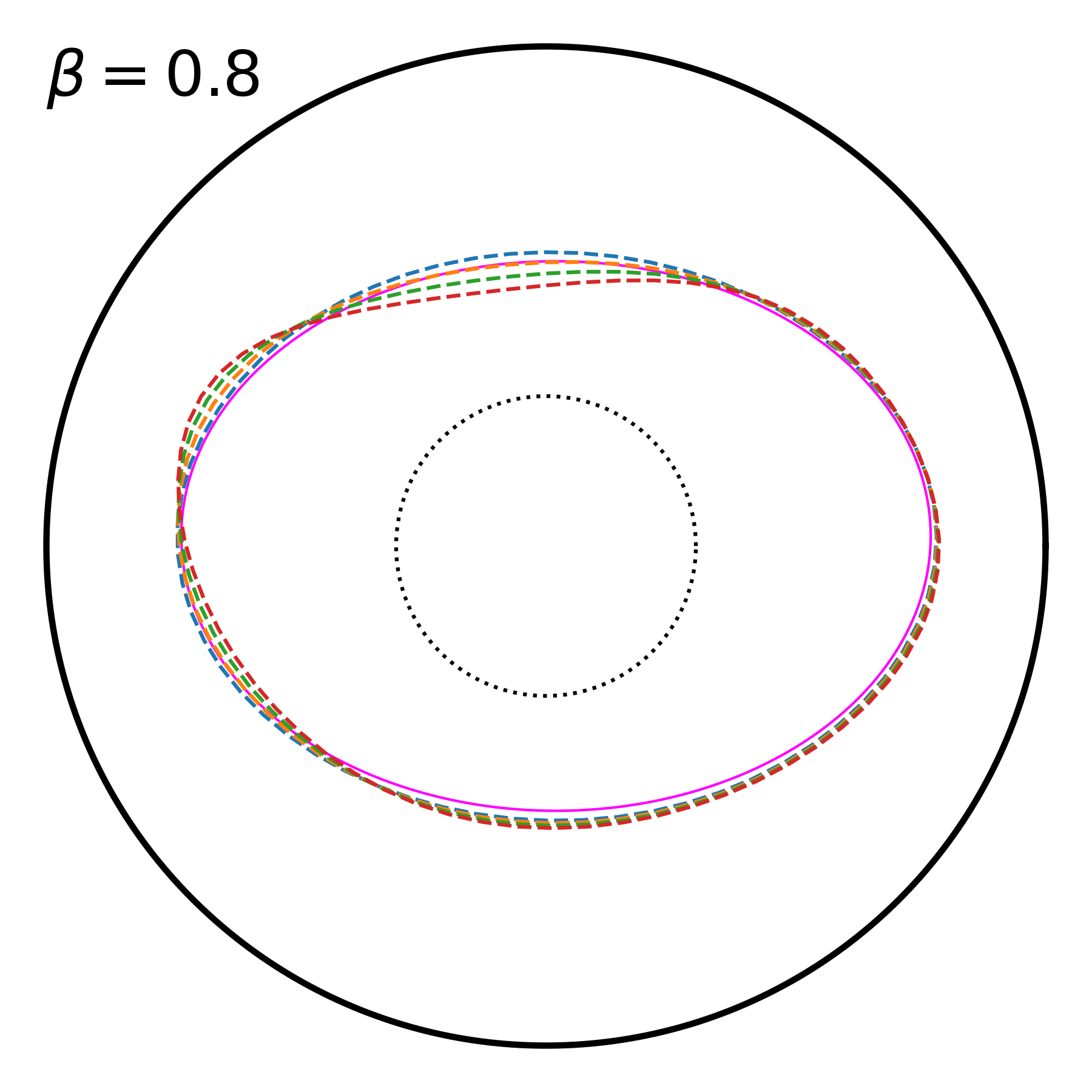}}\ 
    \resizebox{0.23\textwidth}{!}{\includegraphics{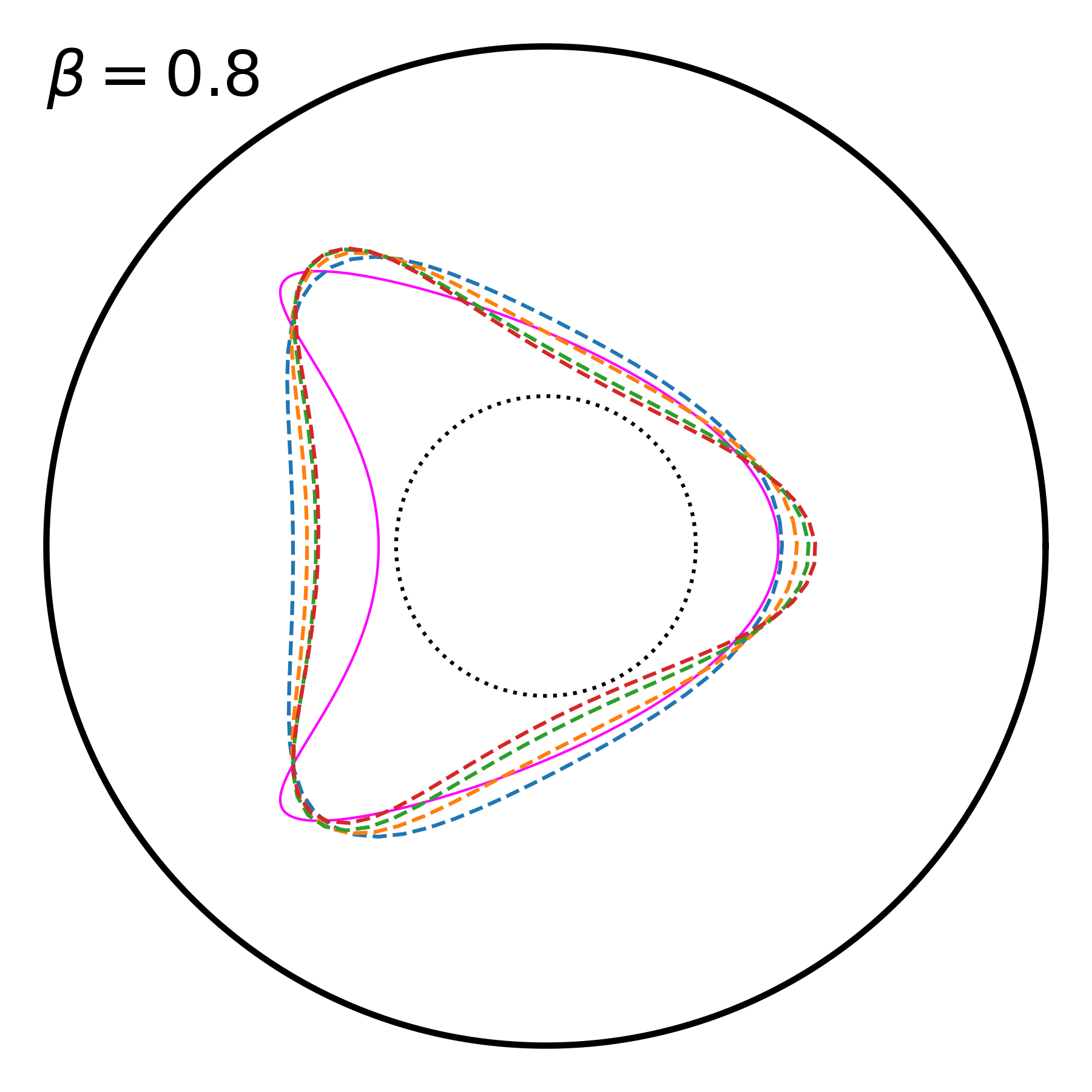}}\ 
    \resizebox{0.23\textwidth}{!}{\includegraphics{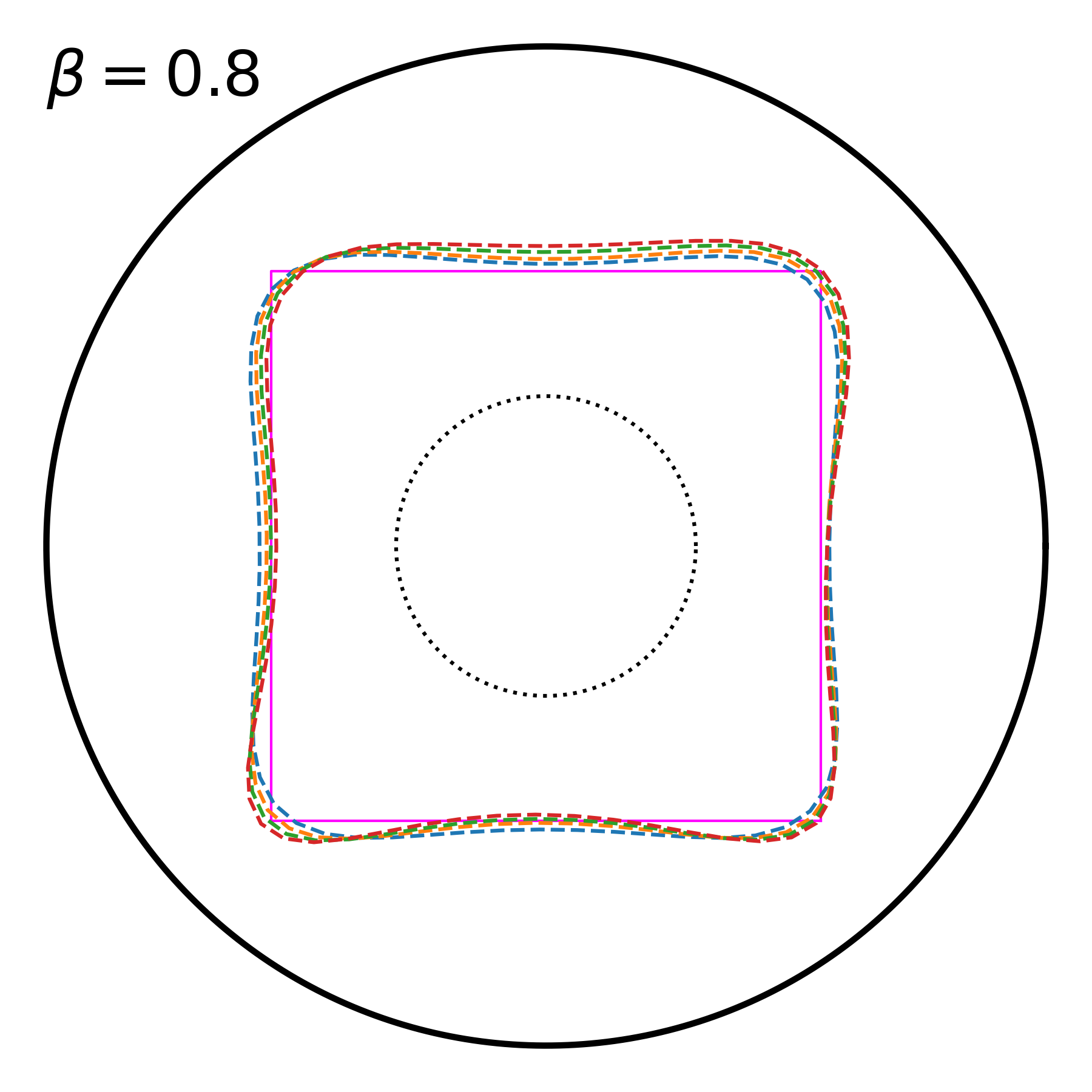}}\ 
    \resizebox{0.23\textwidth}{!}{\includegraphics{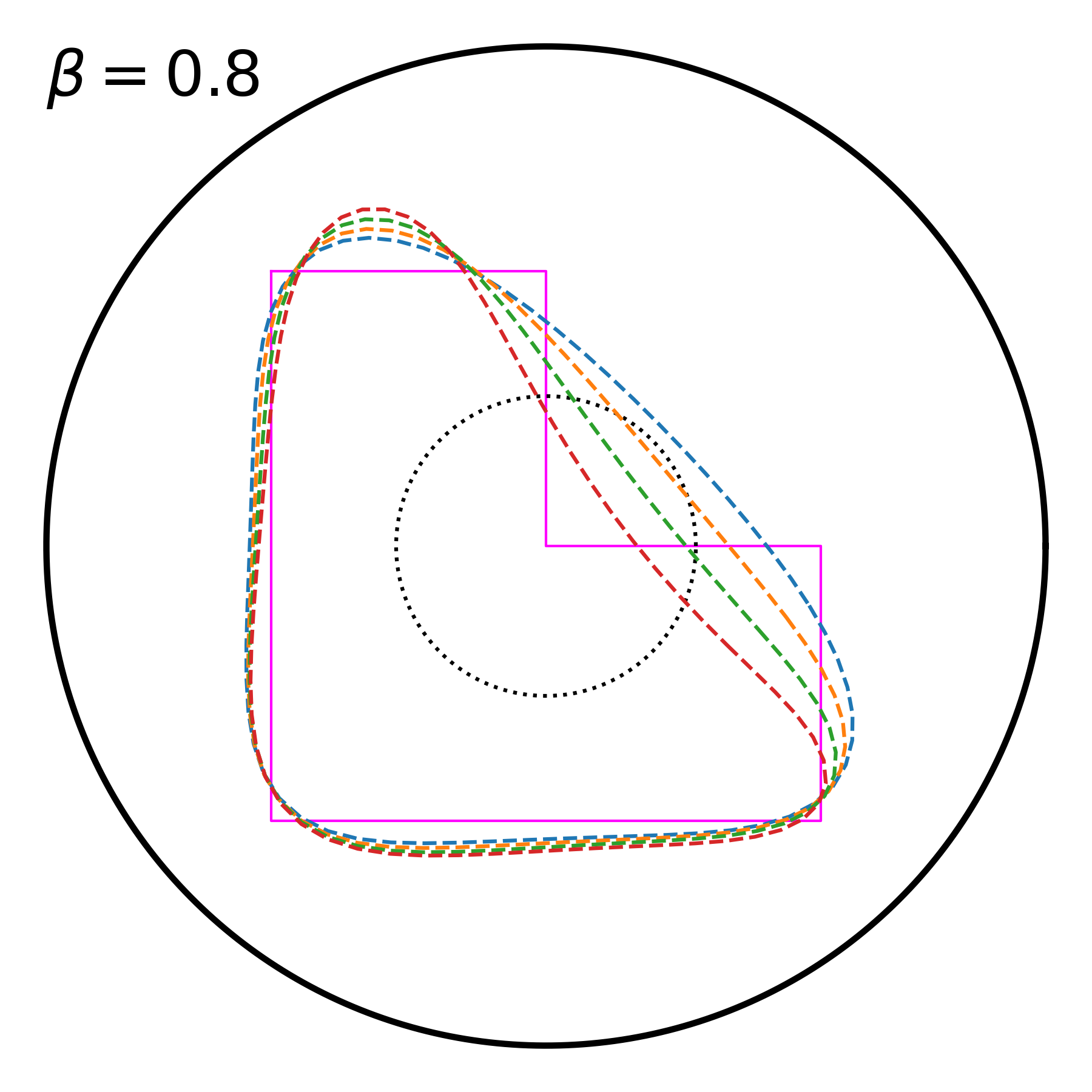}}

\caption{
Comparison of results for different regularization parameters ($\beta = 0, 0.8$) across various tests (columns: E = ellipse, K = kite, S = square, L = L-block) with radius $= 0.3$ for the initial guess and different noise levels ($\delta = 0, 0.03, 0.06, 0.09$). 
}
\label{fig:all_reg_radius_0.3}
\end{figure}

\begin{figure}[h!]
    \centering 
    \resizebox{0.24\textwidth}{!}{\includegraphics{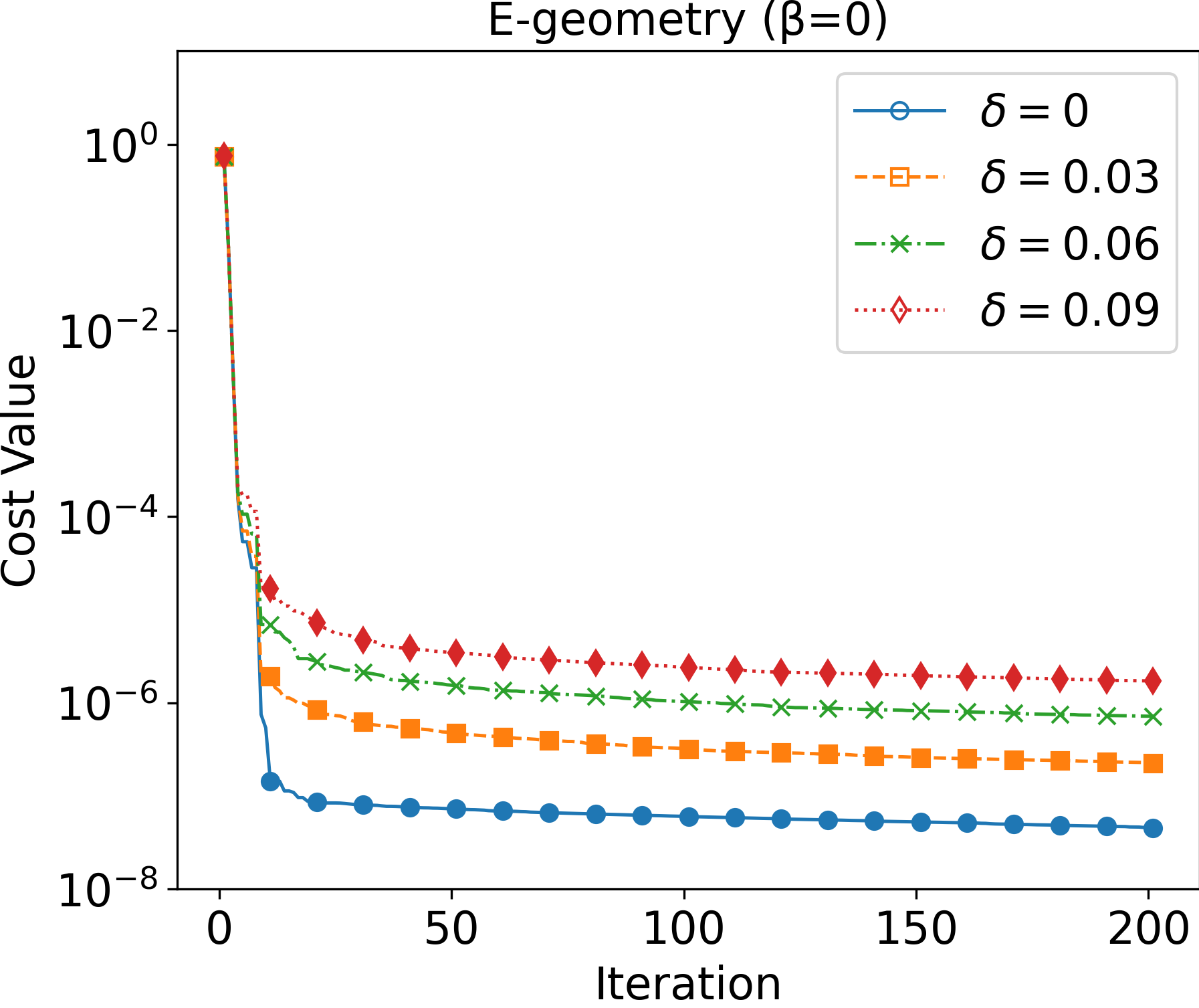}}
    \resizebox{0.24\textwidth}{!}{\includegraphics{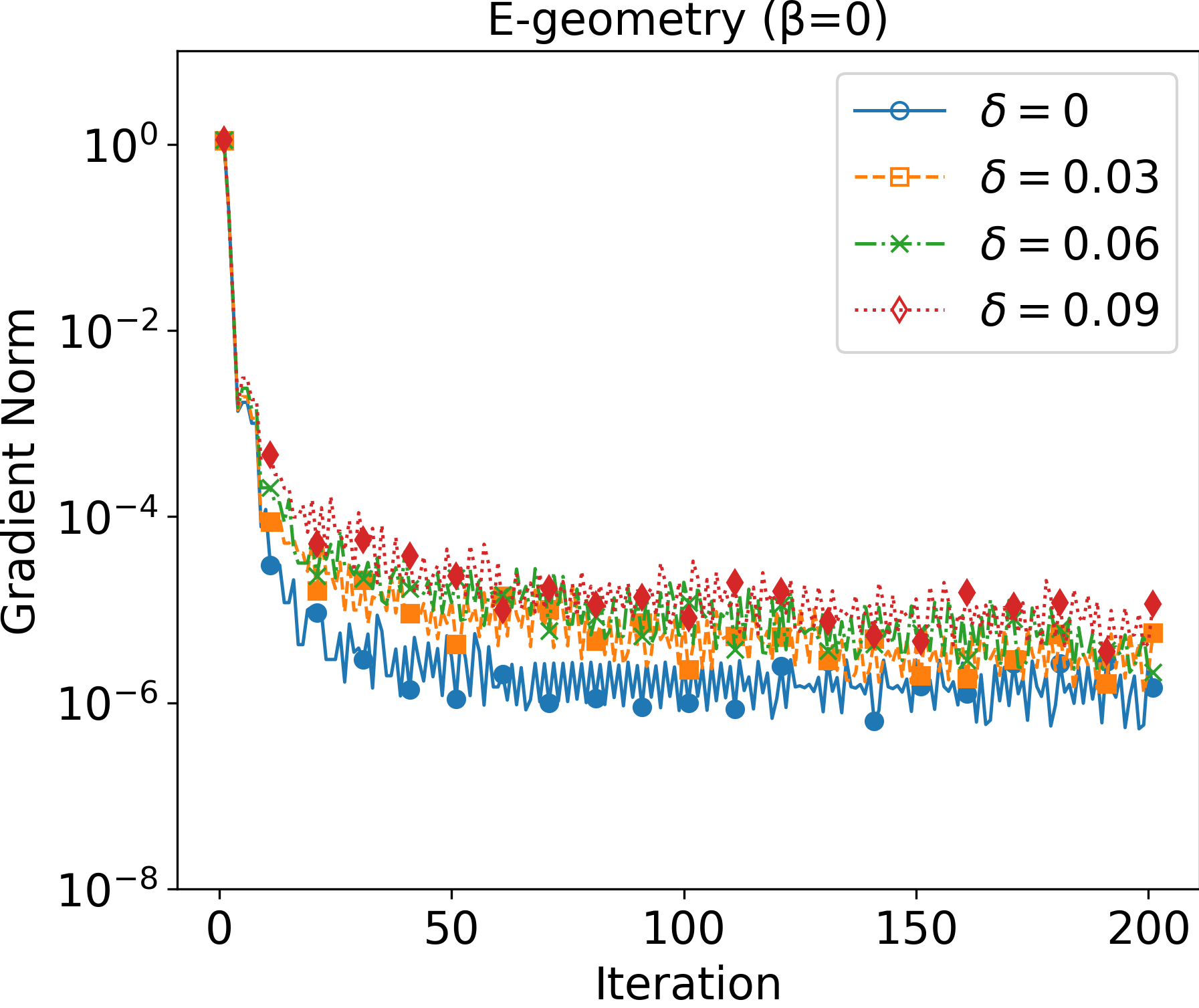}}
    \resizebox{0.24\textwidth}{!}{\includegraphics{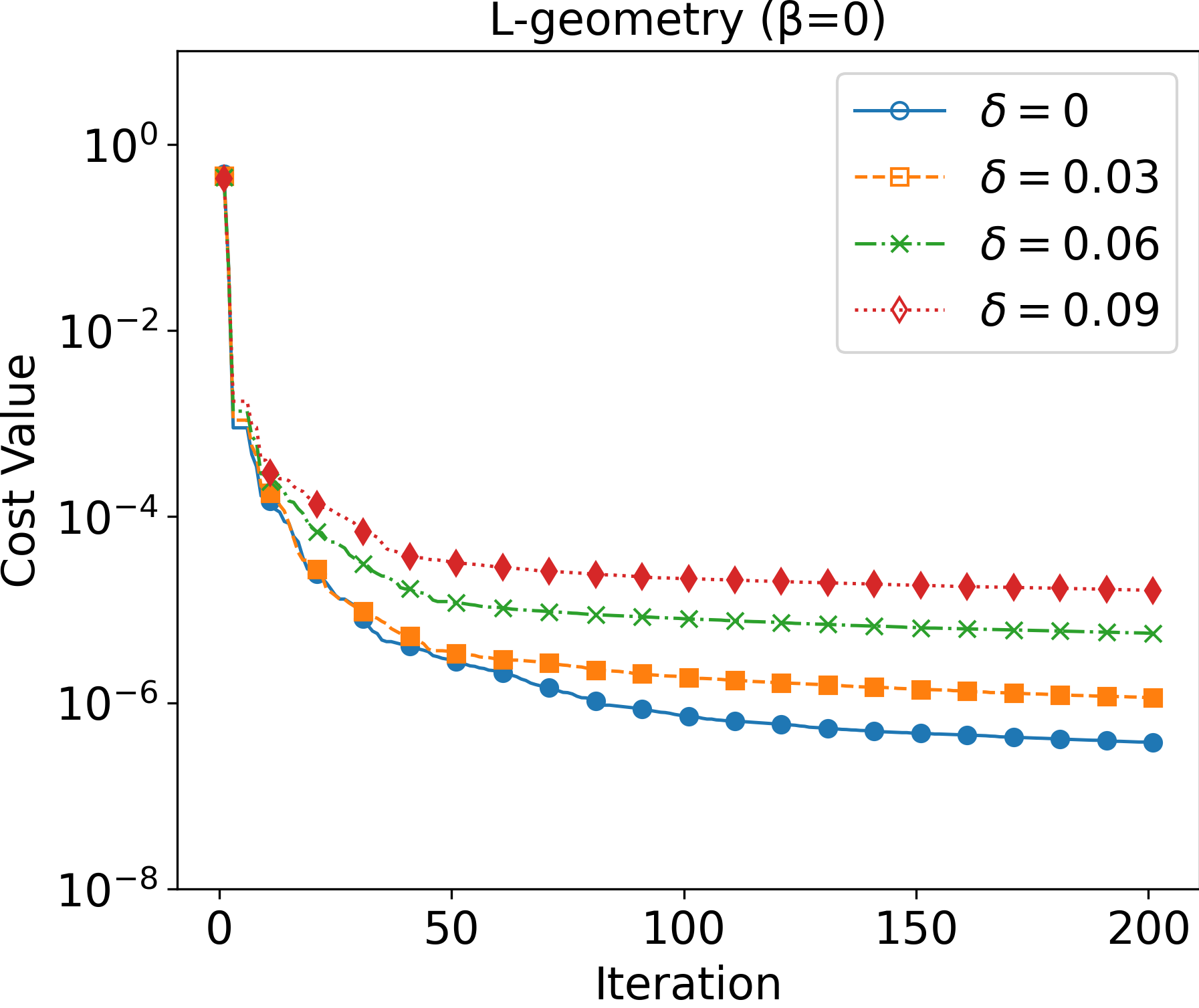}}
    \resizebox{0.24\textwidth}{!}{\includegraphics{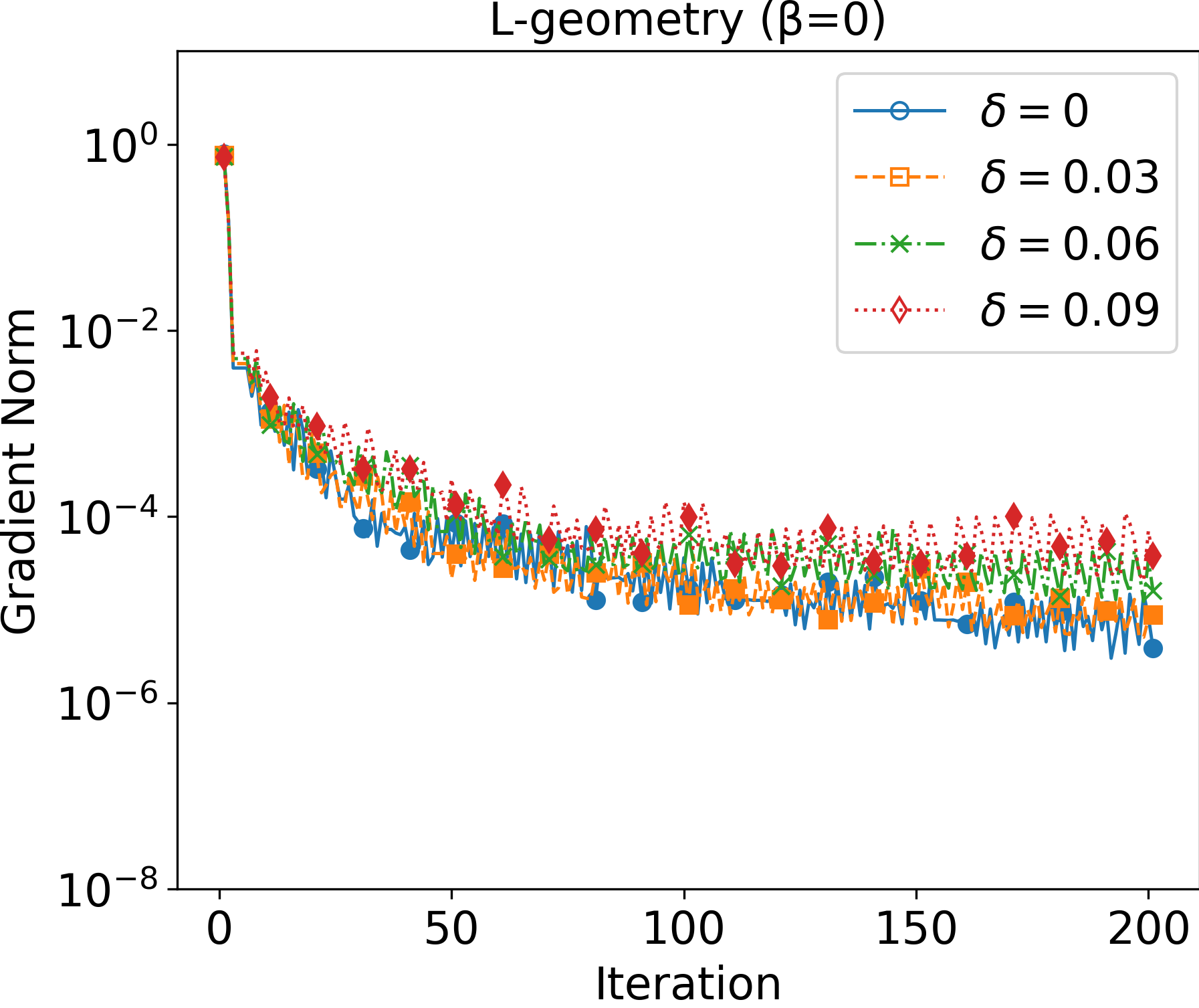}}	
    \caption{Cost function and gradient norm evolution for the Ellipse (E) and L-block (L) cases.}  
    \label{fig:cost_grad_history}
\end{figure}

\begin{figure}[h!]
    \centering 
    \resizebox{0.24\textwidth}{!}{\includegraphics{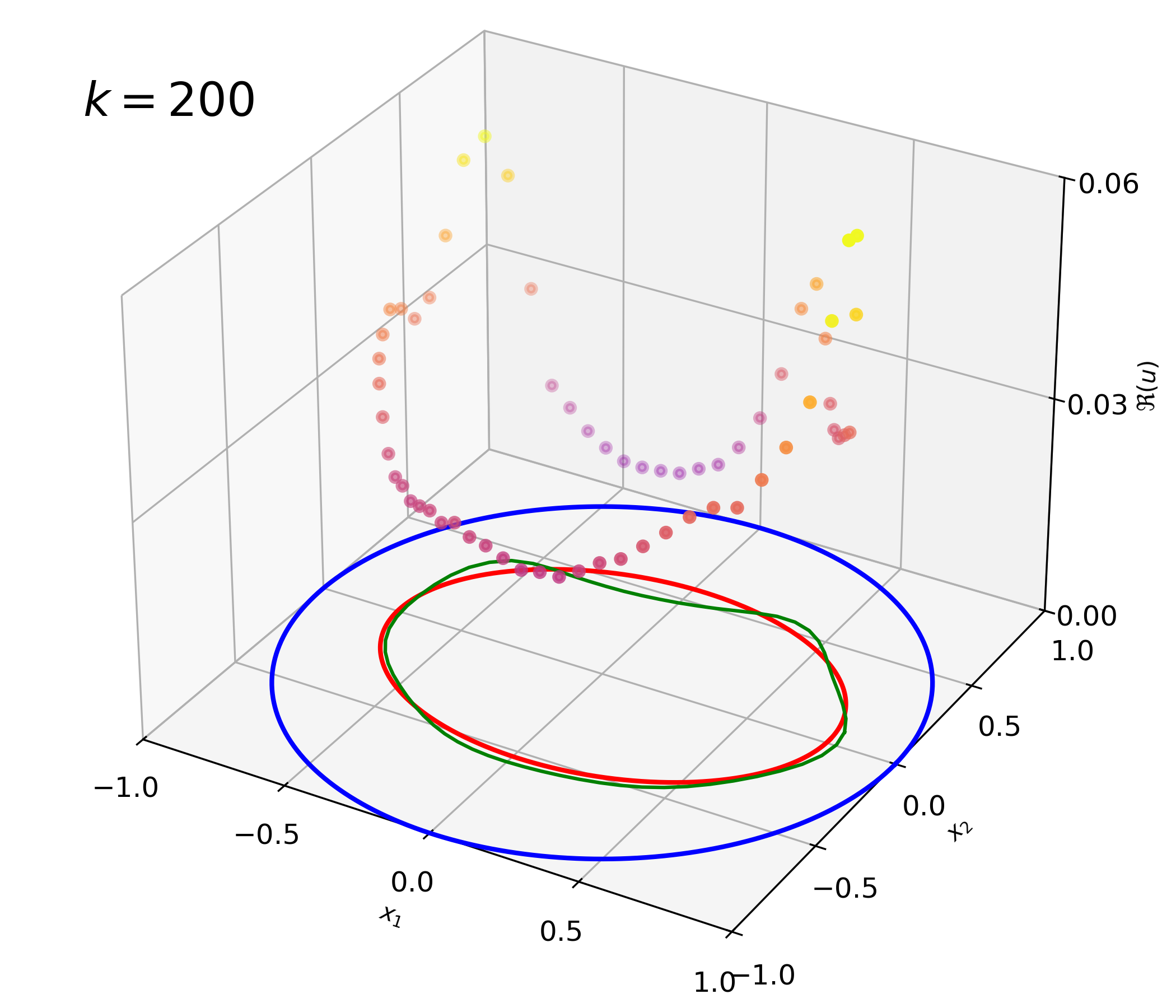}}
    \resizebox{0.24\textwidth}{!}{\includegraphics{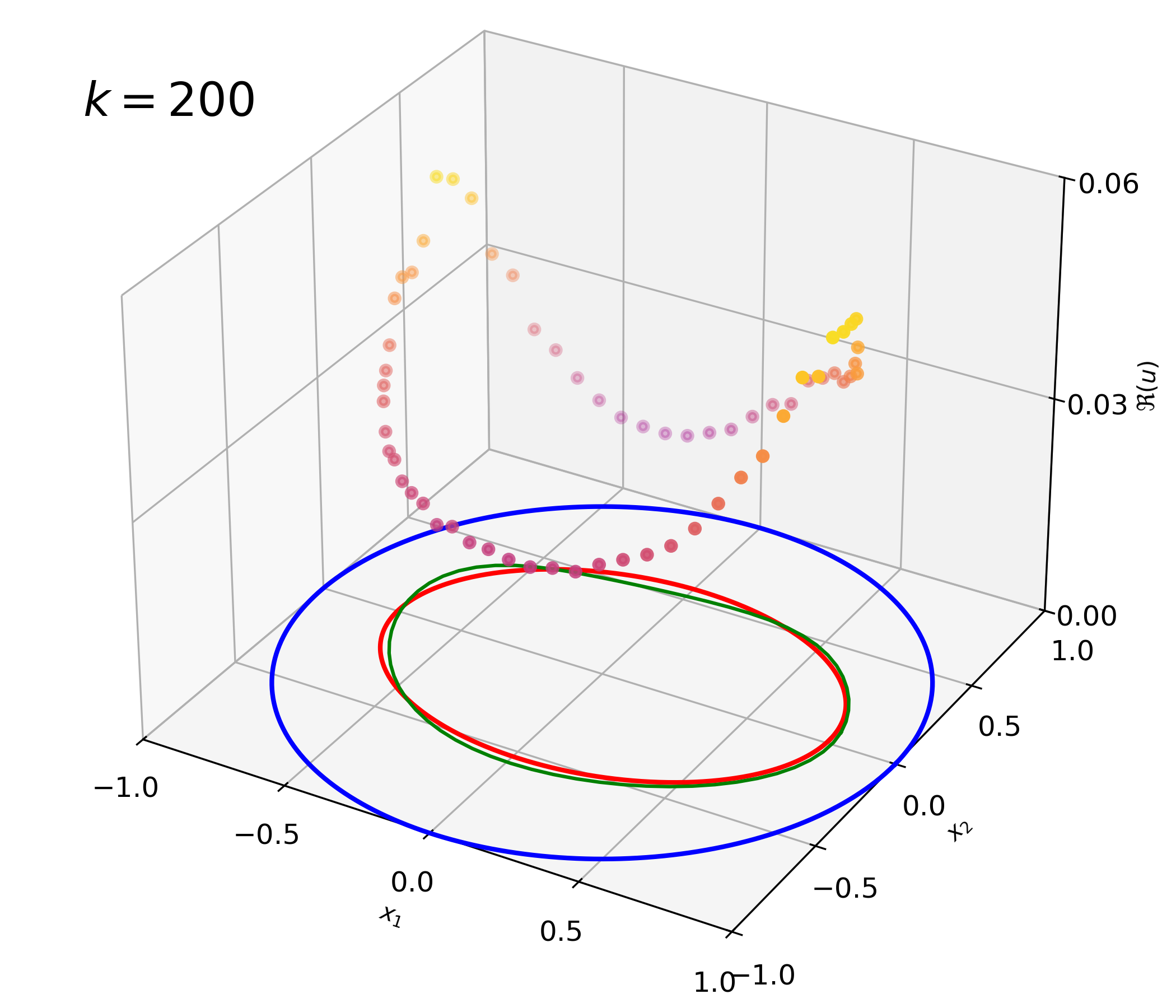}}\hfill
    \resizebox{0.24\textwidth}{!}{\includegraphics{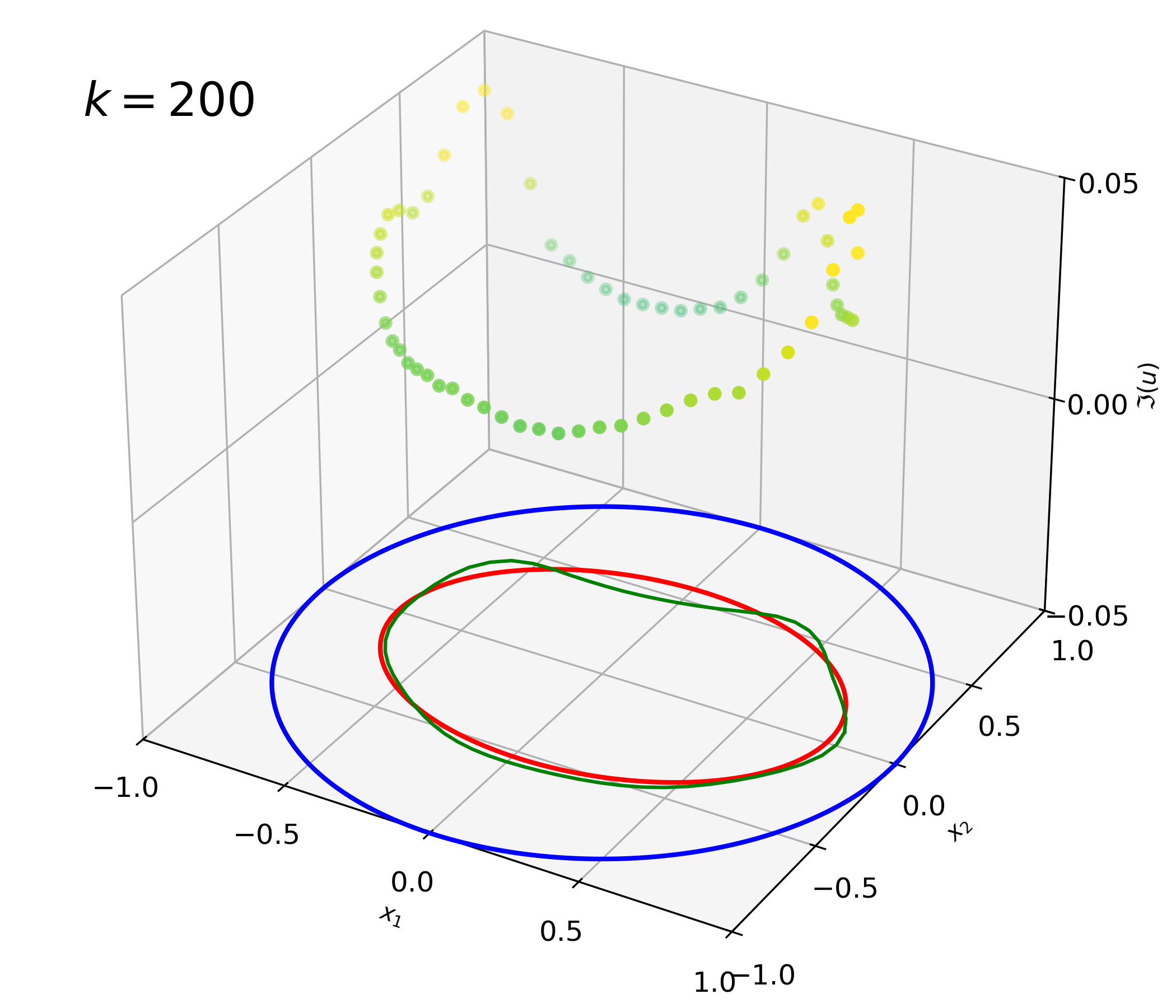}}
    \resizebox{0.24\textwidth}{!}{\includegraphics{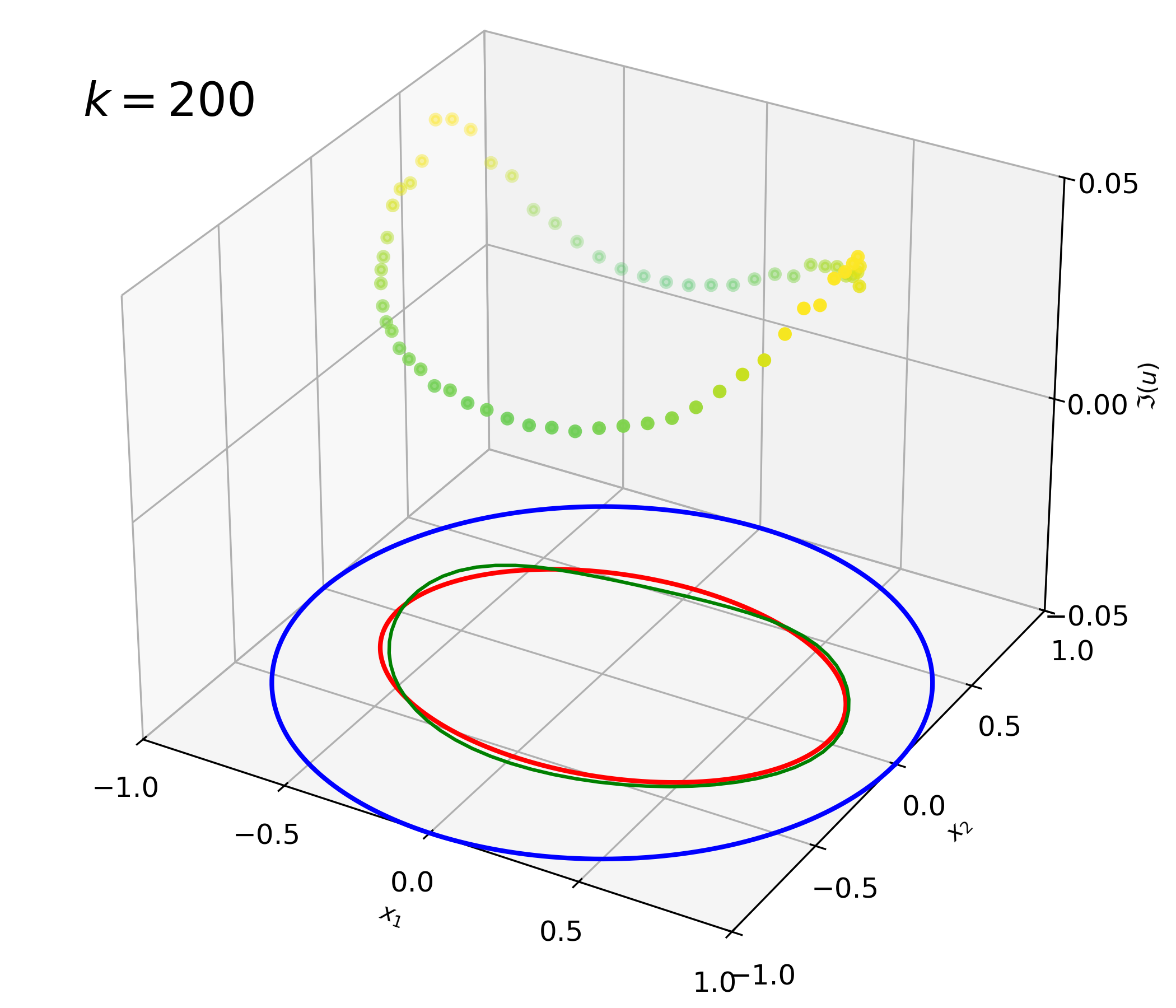}}\\[4pt]
    \resizebox{0.48\textwidth}{!}{\includegraphics{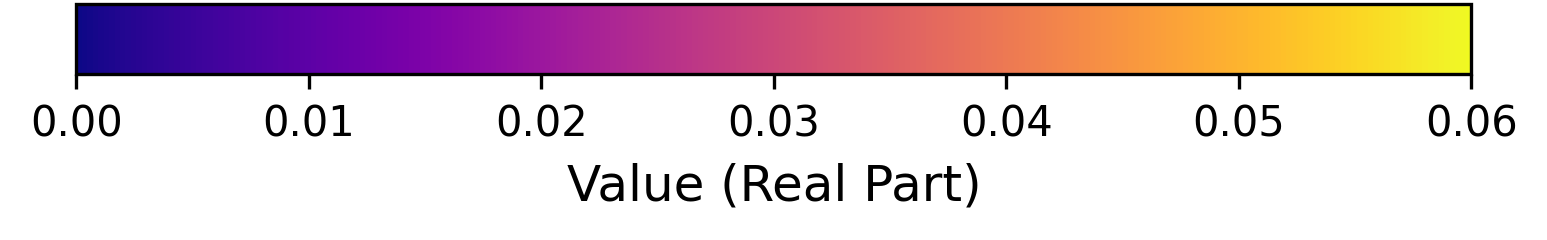}}\hfill
    \resizebox{0.48\textwidth}{!}{\includegraphics{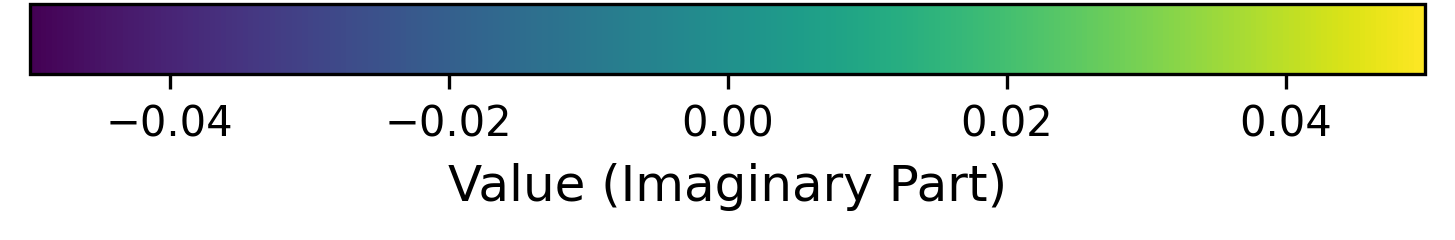}}
\caption{
Real and imaginary parts of the state solution at $k = 200$: the first two panels correspond to the real part and the last two to the imaginary part; in each pair, $\beta = 0$ (left) and $\beta = 0.8$ (right), with $\varGamma^{0} = C(0,0,0.3)$.
}
\label{fig:ellipse_state_real_imag_comparison}
\end{figure}

\begin{figure}[h!]
    \centering
    \resizebox{0.32\textwidth}{!}{\includegraphics{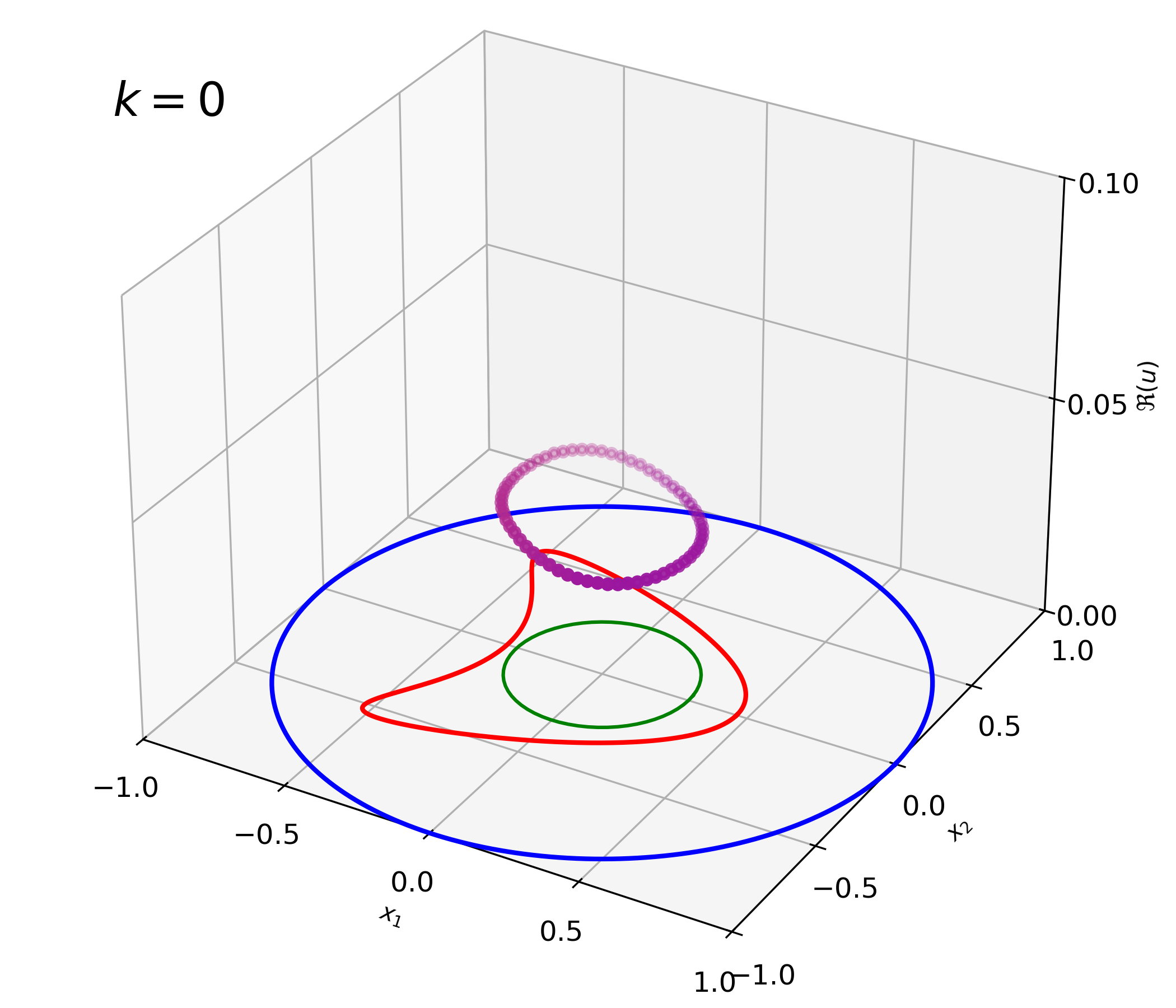}}\ 
    \resizebox{0.32\textwidth}{!}{\includegraphics{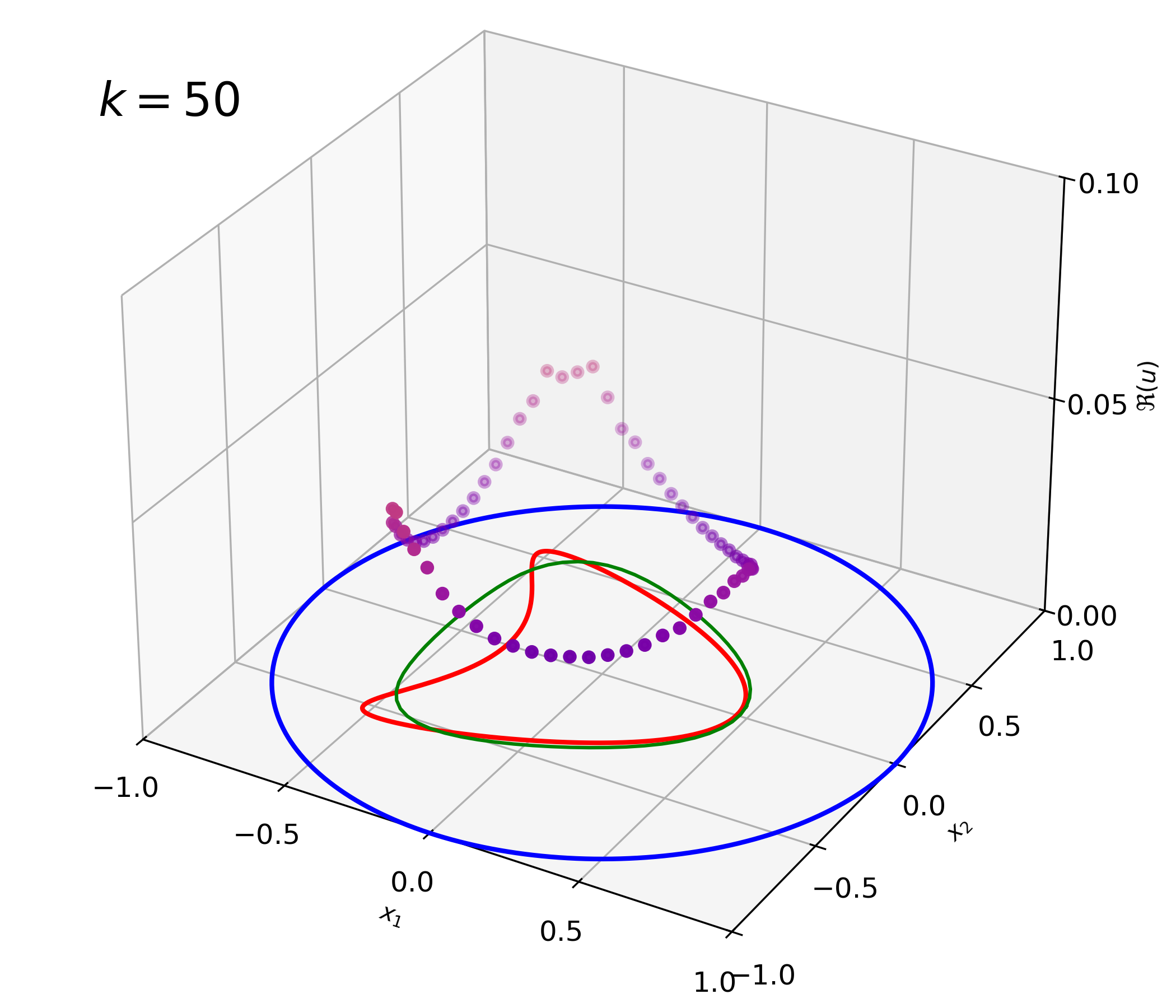}}\ 
    \resizebox{0.32\textwidth}{!}{\includegraphics{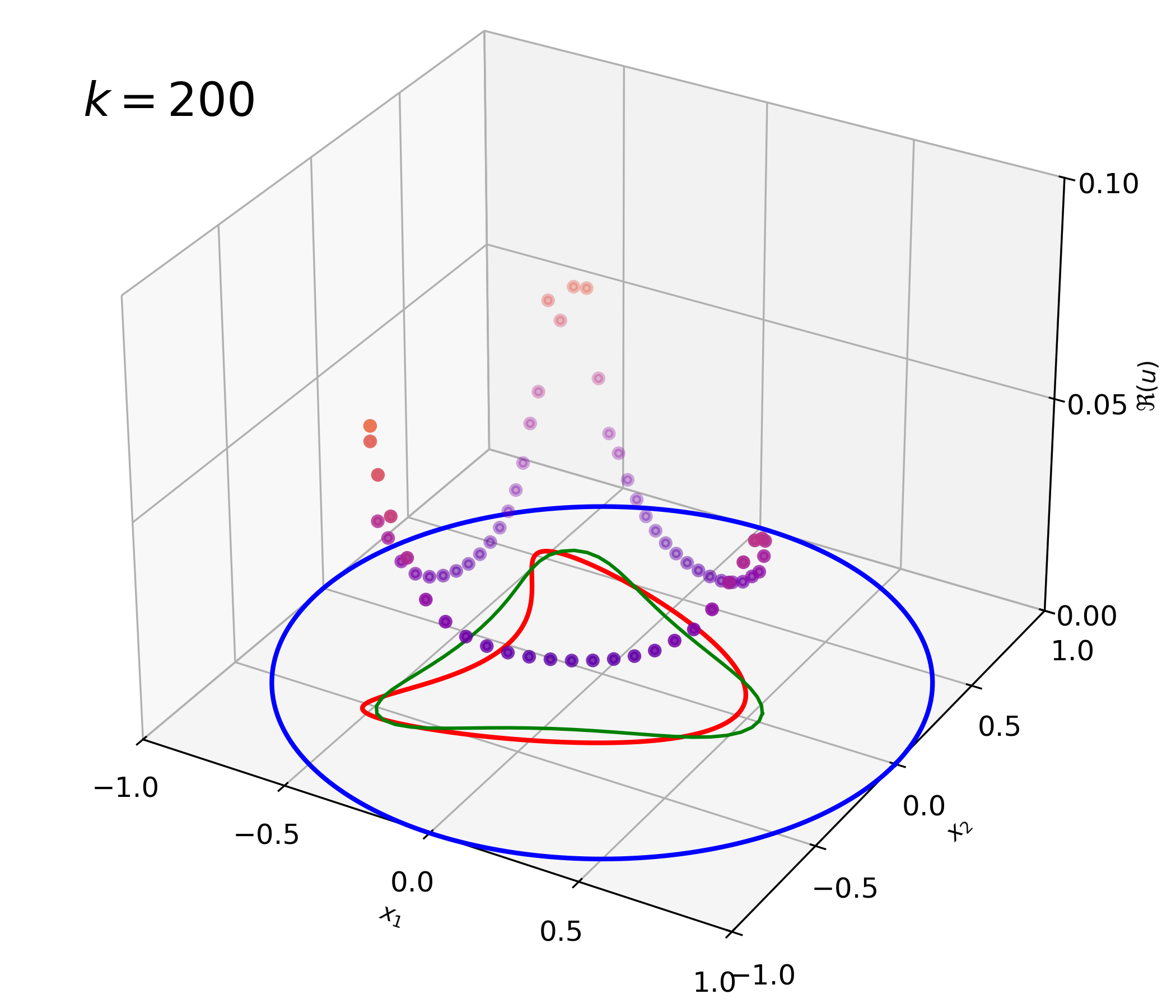}}\\[4pt]
    \resizebox{0.4\textwidth}{!}{\includegraphics{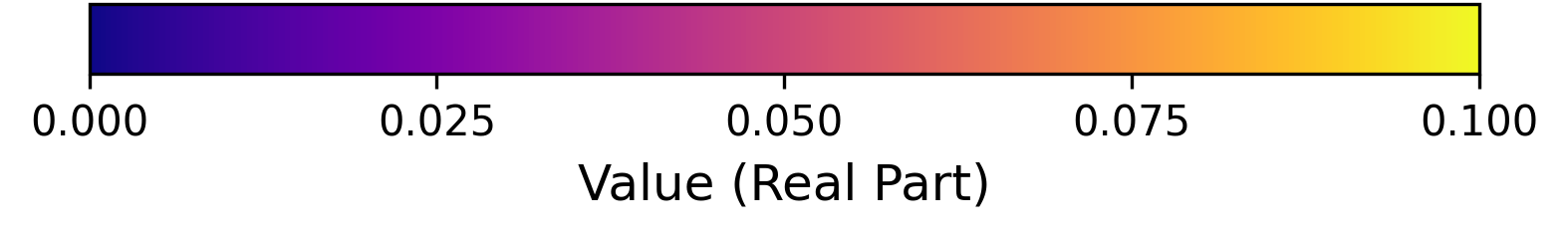}}\\[6pt]
    \resizebox{0.32\textwidth}{!}{\includegraphics{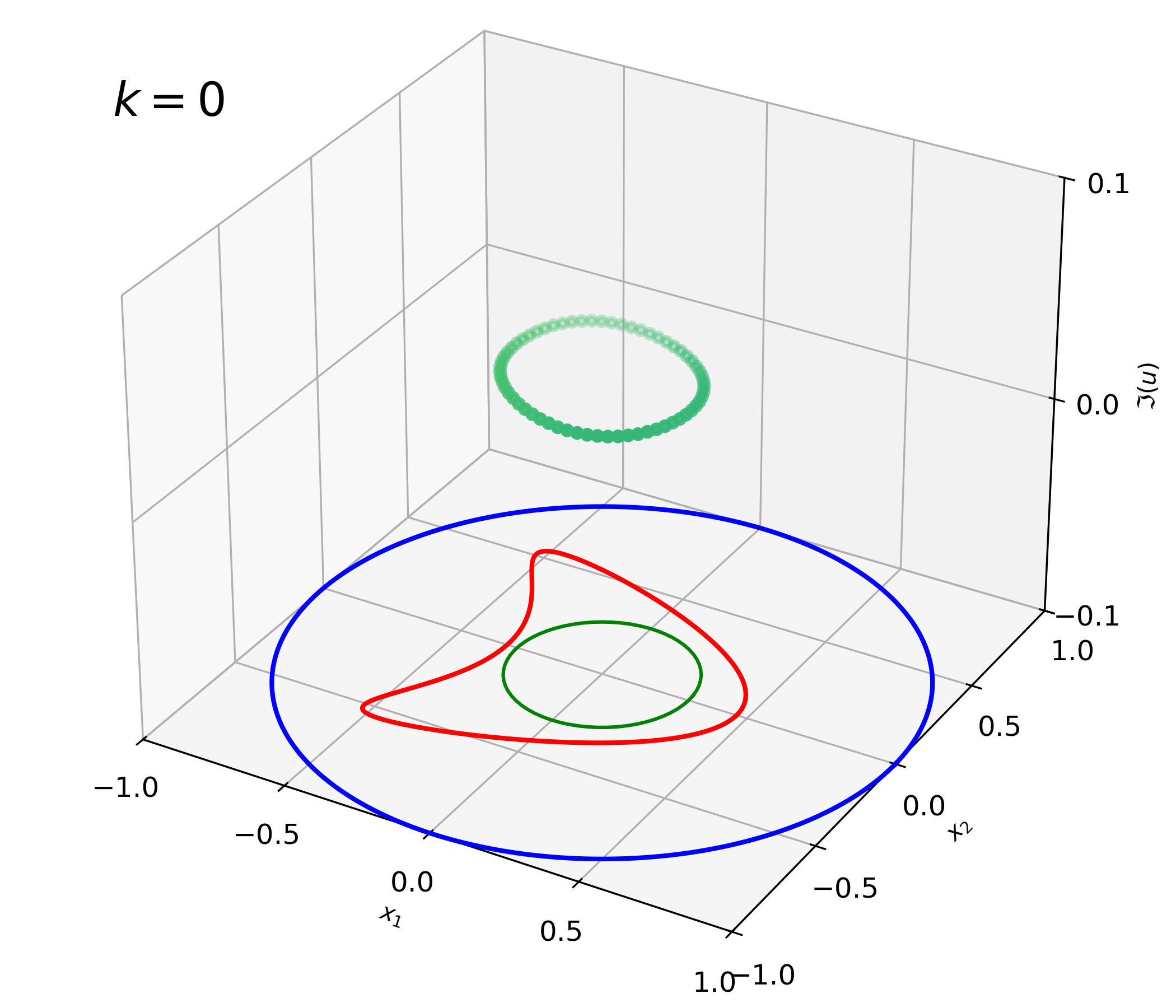}}\ 
    \resizebox{0.32\textwidth}{!}{\includegraphics{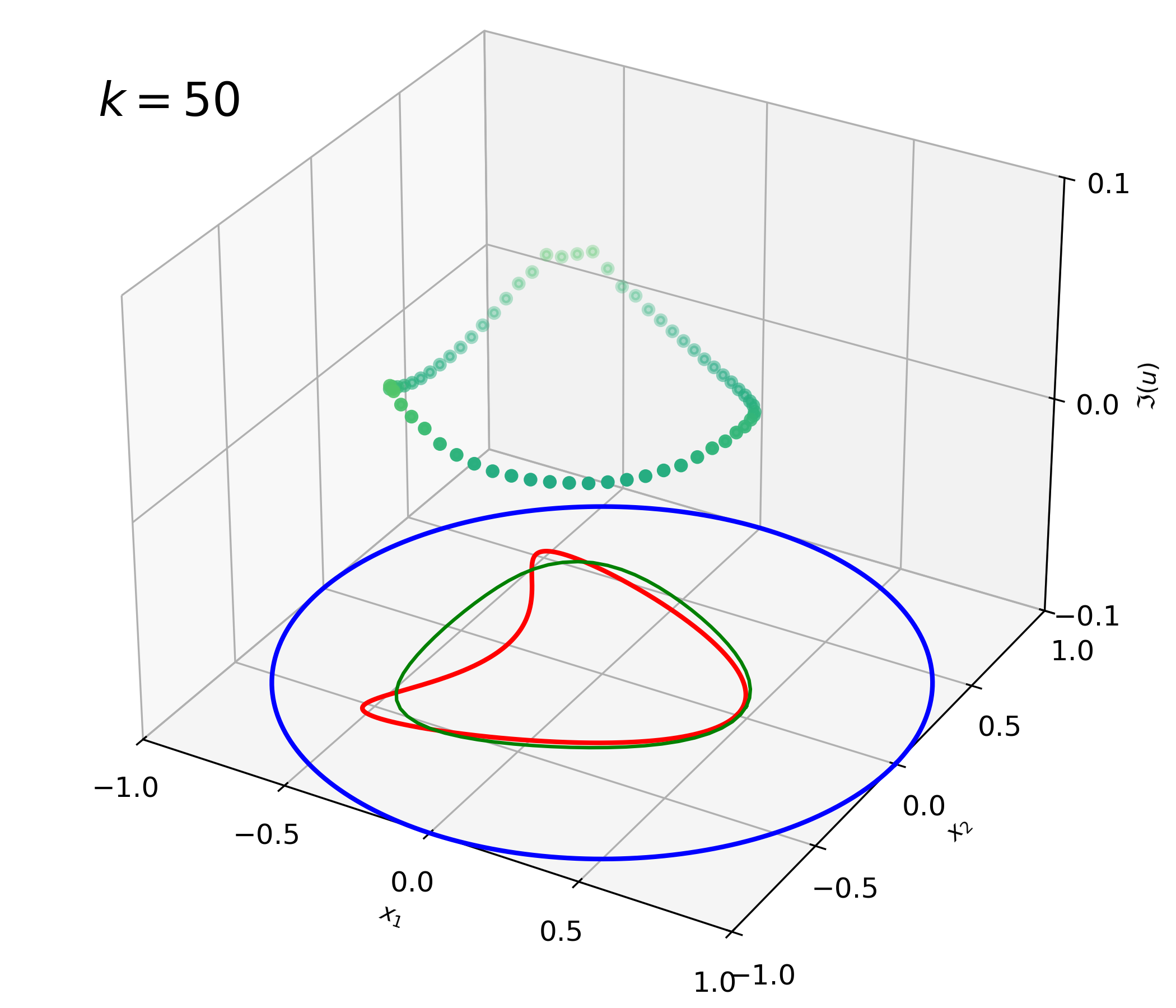}}\ 
    \resizebox{0.32\textwidth}{!}{\includegraphics{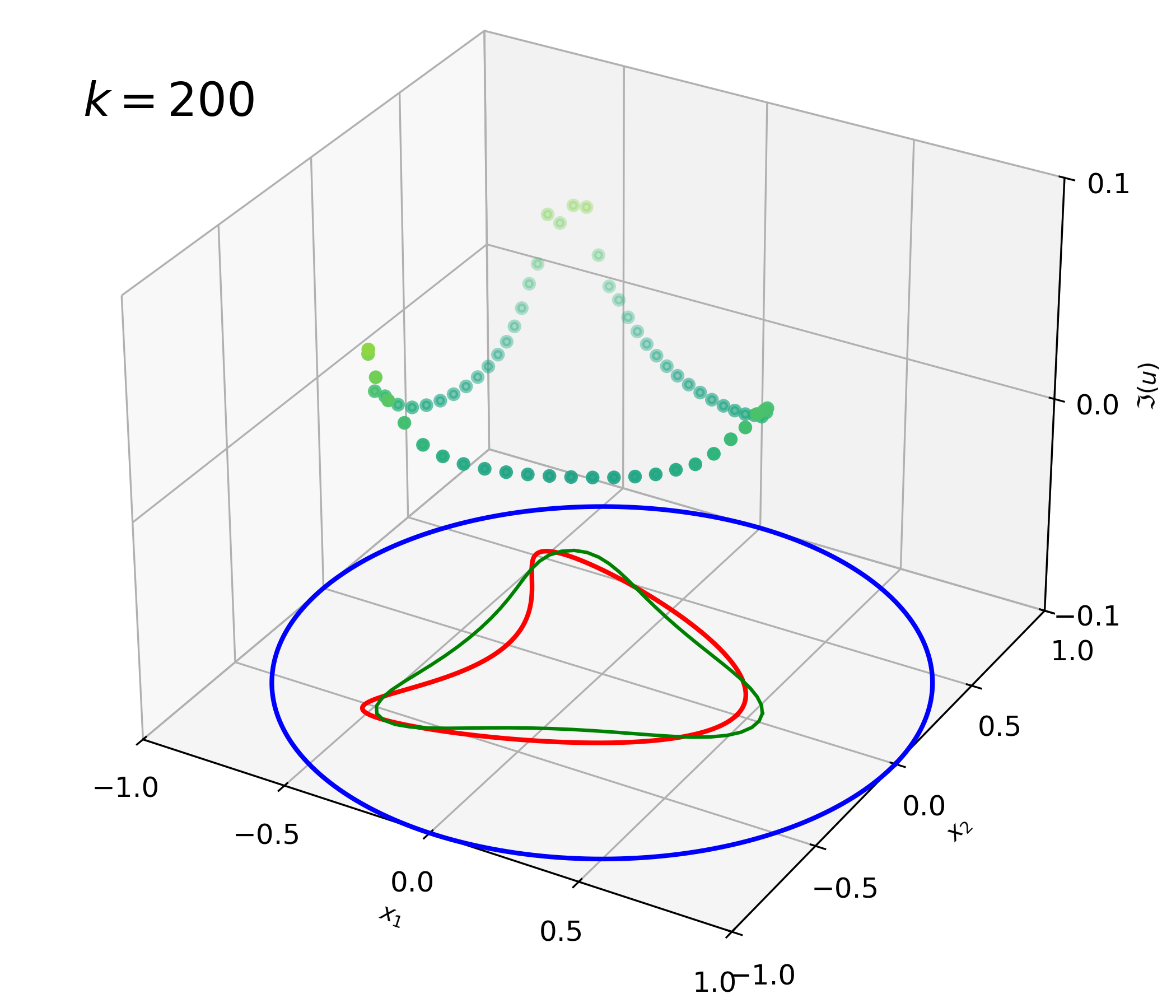}}\\[4pt]
    \resizebox{0.4\textwidth}{!}{\includegraphics{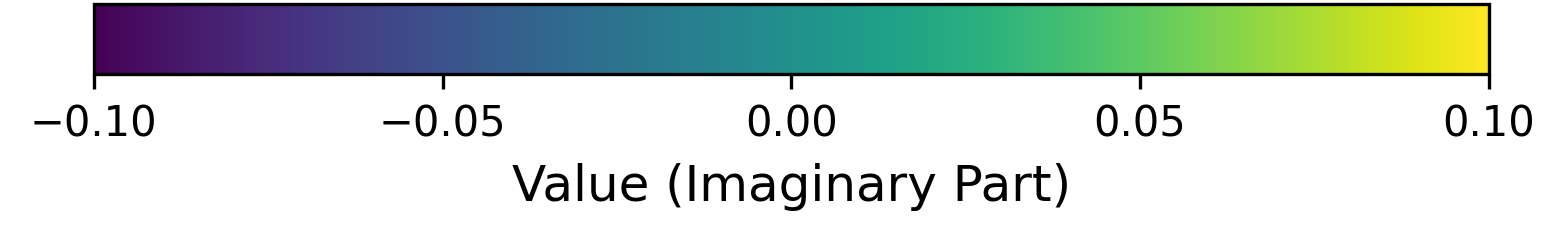}}
\caption{
Evolution of the real and imaginary parts of the state solution for $\beta = 0.8$ at $k = 0, 50, 200$, with $\varGamma^{0} = C(0,0,0.3)$.
}
\label{fig:kite_state_real_imag_beta08}
\end{figure}

\begin{figure}[h!]
    \centering
    \resizebox{0.32\textwidth}{!}{\includegraphics{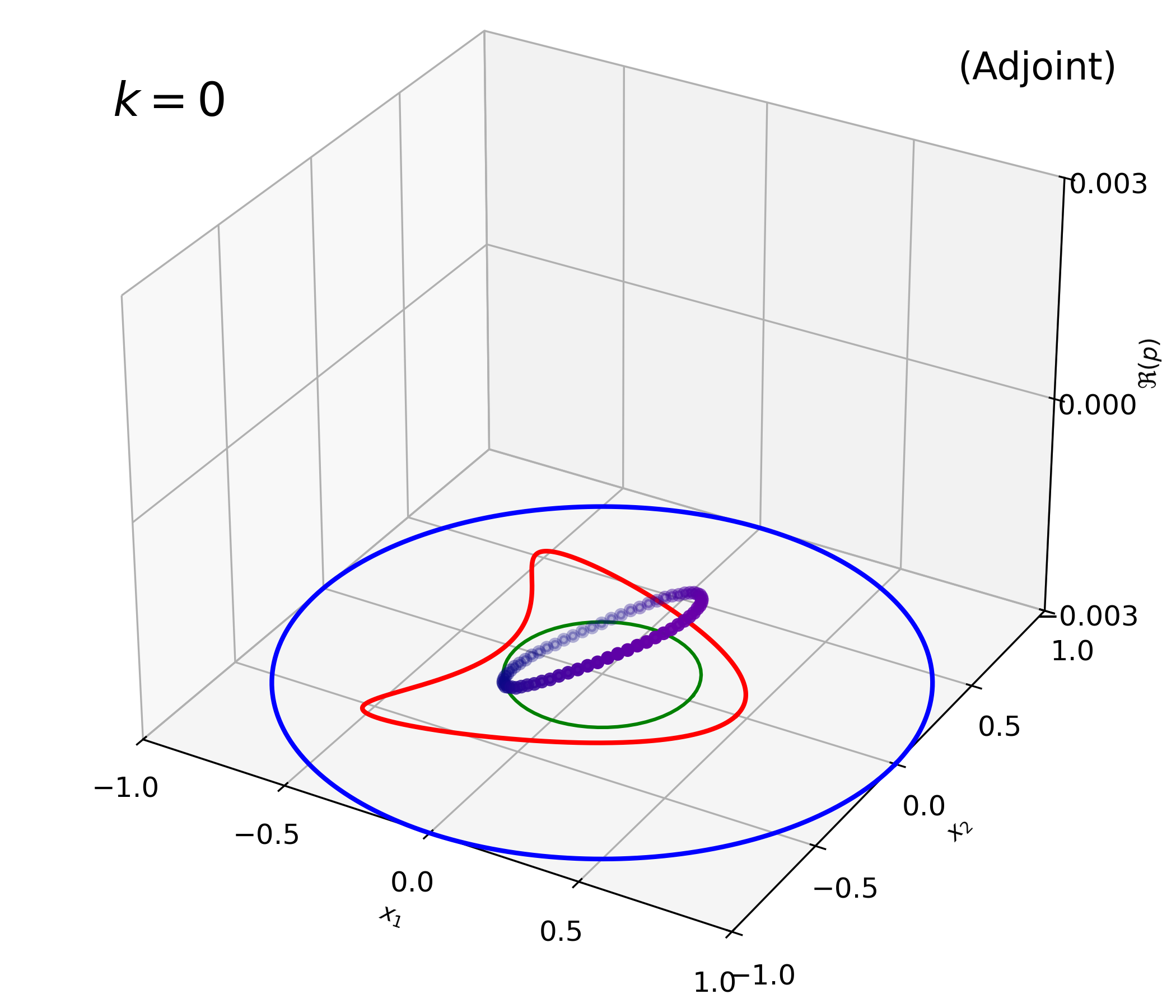}}\ 
    \resizebox{0.32\textwidth}{!}{\includegraphics{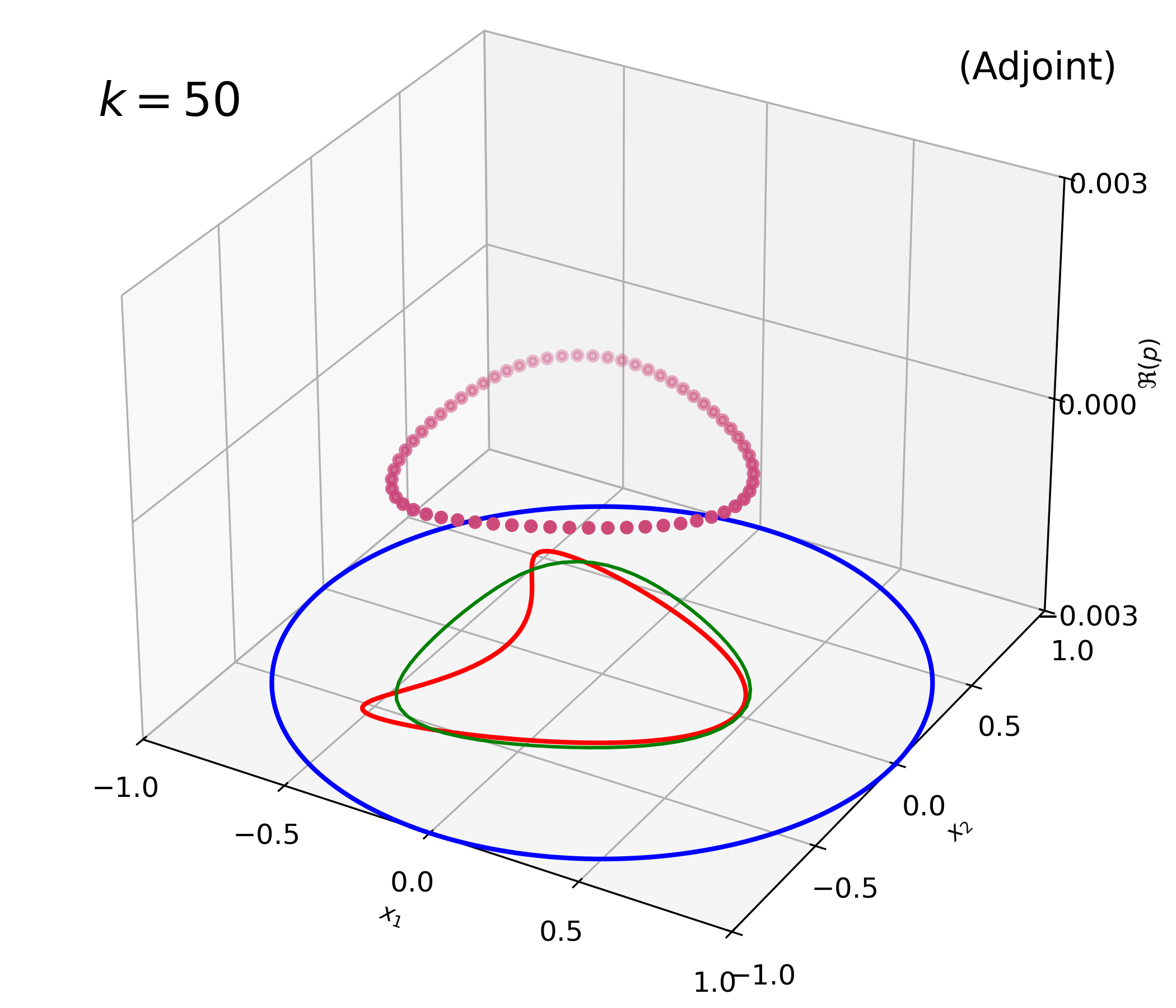}}\ 
    \resizebox{0.32\textwidth}{!}{\includegraphics{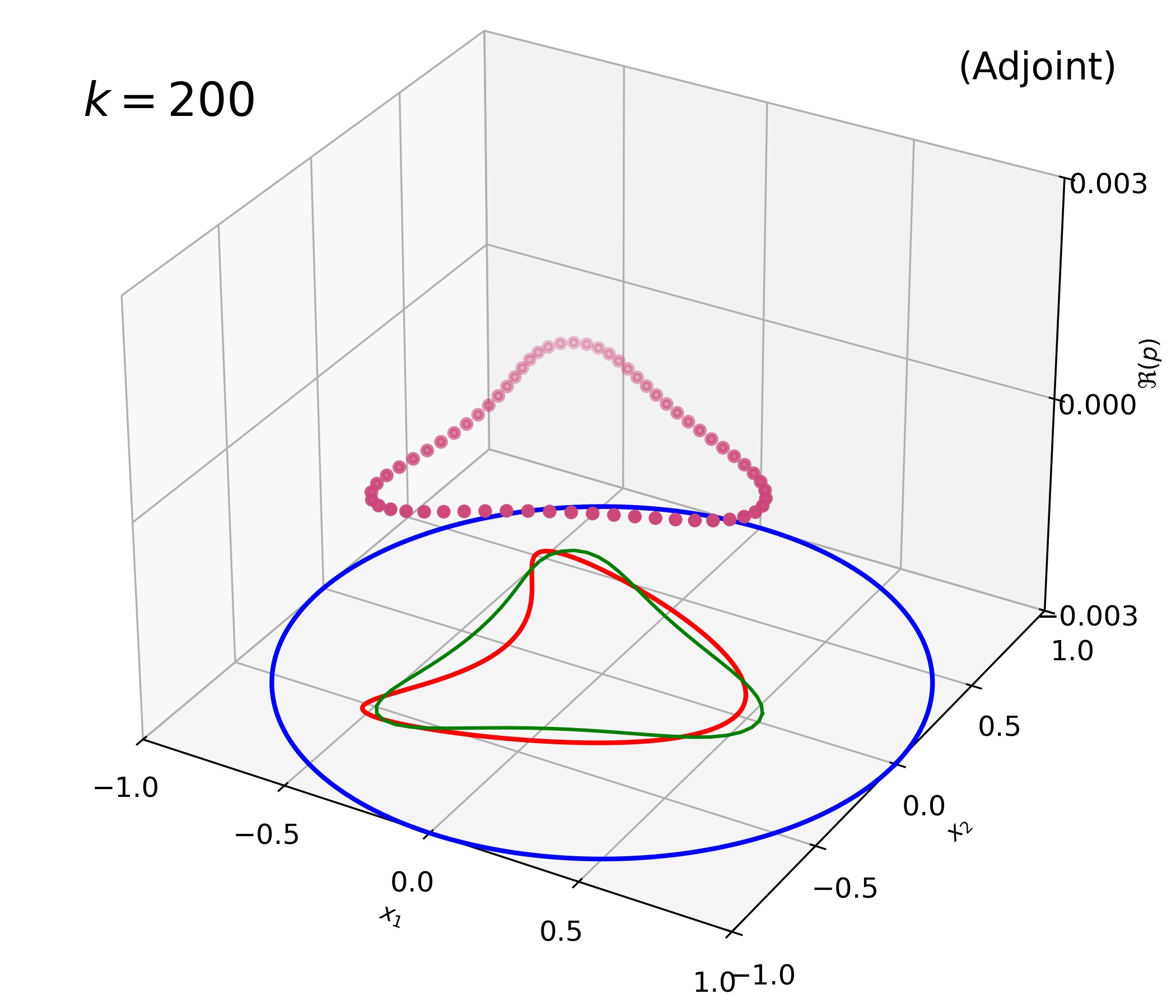}}\\[4pt]
    \resizebox{0.4\textwidth}{!}{\includegraphics{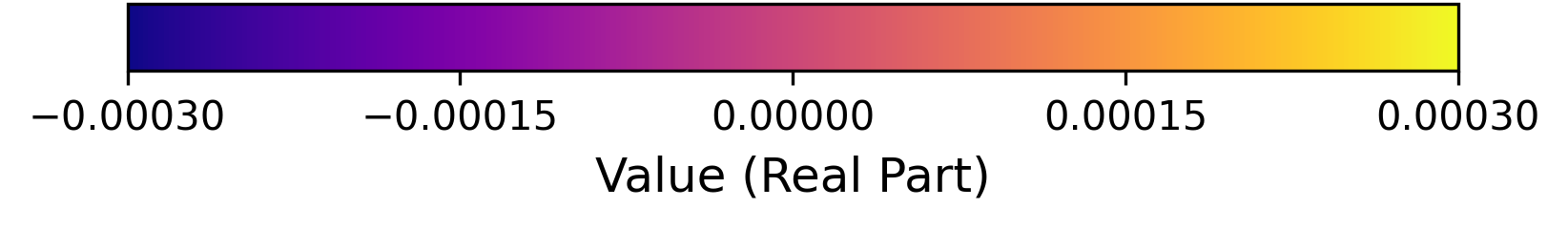}}\\[6pt]
    \resizebox{0.32\textwidth}{!}{\includegraphics{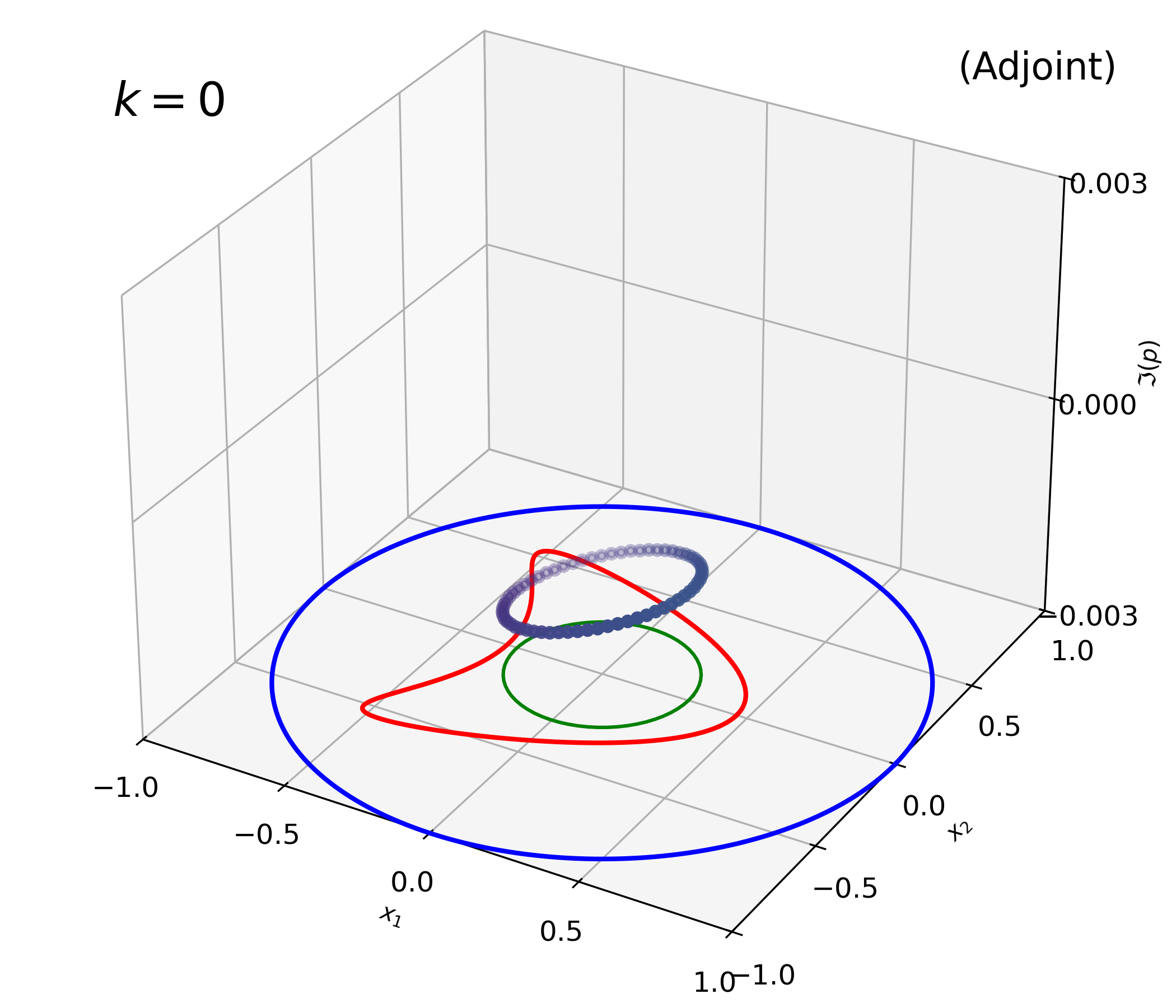}}\ 
    \resizebox{0.32\textwidth}{!}{\includegraphics{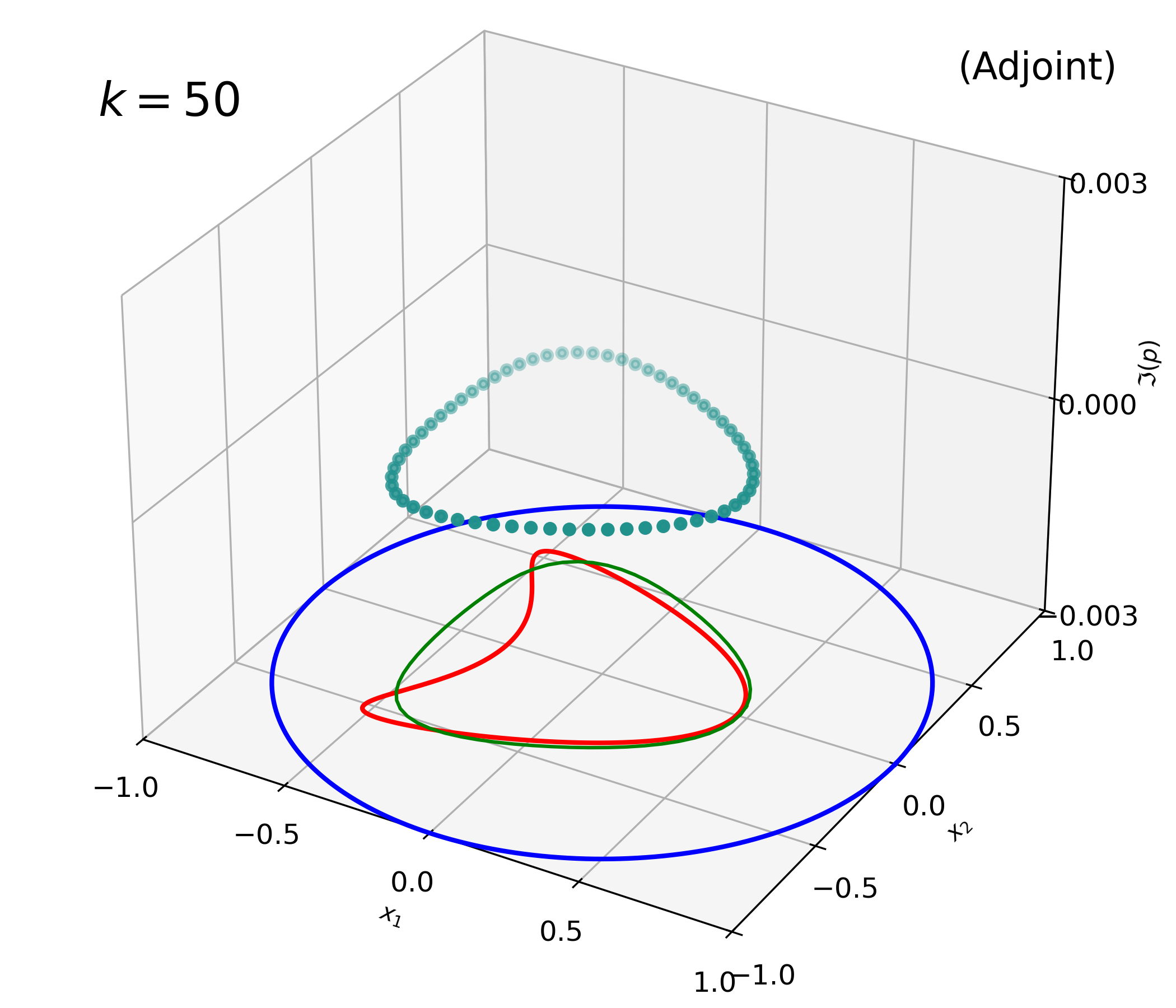}}\ 
    \resizebox{0.32\textwidth}{!}{\includegraphics{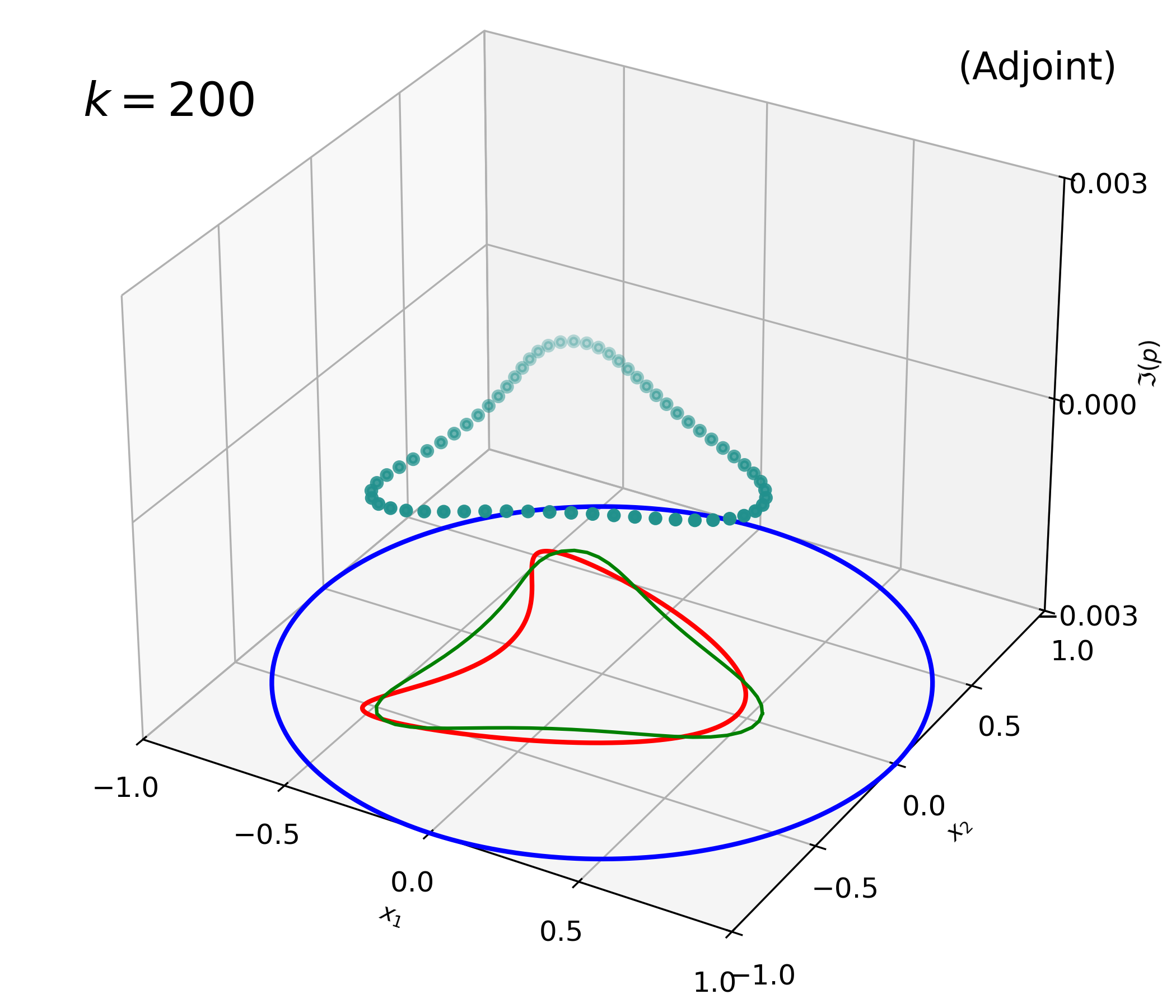}}\\[4pt]
    \resizebox{0.4\textwidth}{!}{\includegraphics{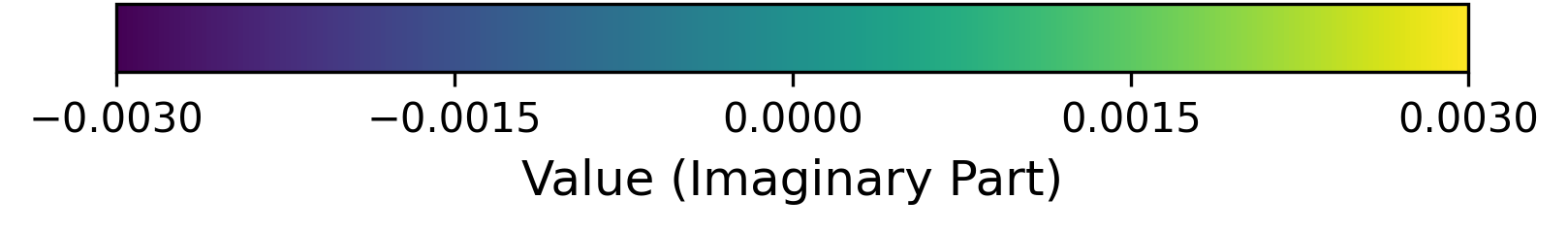}}\\[6pt]
\caption{
Evolution of the real and imaginary parts of the adjoint solution for $\beta = 0.8$ at $k = 0, 50, 200$, with $\varGamma^{0} = C(0,0,0.3)$.
}
\label{fig:kite_adjoint_real_imag_beta08}
\end{figure}

\section{Shape reconstruction with an inequality constraint}
\label{sec:inequality_constrained_shape_reconstruction}
As shown in the previous sections, the inverse boundary identification problem can be reformulated as the minimization problem \eqref{eq:shape_problem}, which yields reasonable reconstructions in practice.
Nevertheless, the problem remains inherently ill-posed and sensitive to noise.
In the numerical experiments presented thus far, the prescribed Robin coefficient is relatively large (e.g., $\alpha = 100$), which empirically leads to more stable reconstructions.
From a PDE perspective, such values correspond to stronger boundary interactions, closer to a Dirichlet-type condition, thereby providing better control of the solution.
For smaller values of $\alpha$, on the other hand, the boundary condition approaches a Neumann-type regime, resulting in weaker boundary interactions and a more severely ill-posed inverse problem.
Consequently, the reconstruction quality deteriorates, particularly in accurately resolving the depth of concavities.
This behavior reflects a fundamental difficulty of the problem and is not specific to the CCBM formulation, but is also observed in more classical approaches, such as least-squares tracking of boundary data and the Kohn--Vogelius formulation \cite{AfraitesRabago2025}.

To overcome this difficulty, we augment the CCBM shape optimization problem with an inequality constraint and solve the resulting constrained problem (Problem~\ref{prob:optimal_shape_problem}) using an Alternating Direction Method of Multipliers (ADMM) strategy.
Within this framework, the ADMM approach enhances the stability of the reconstruction while preserving the original CCBM formulation and the shape sensitivity analysis developed in Section~\ref{sec:Shape_Derivatives}.
In particular, the method enforces a physically motivated constraint on the real part of the state variable through a variable splitting technique, without introducing additional geometric regularization terms.
\begin{problem}\label{prob:optimal_shape_problem}
Let $a$ and $b$ be given fixed constants with $b \geqslant a$.
Find the shape $\omega^{\star}$ in the space of admissible set
\[
	\mathcal{U} = \left\{ \omega \in \mathcal{A} \mid \text{$a \leqslant \Re\{u\} \leqslant b$ a.e. in $\varOmega$ where $u$ solves \eqref{eq:ccbm_state}} \right\}
\]
such that
\begin{equation}\label{eq:control} 
	\omega^{\star} 
	= \argmin_{\omega \in \mathcal{U}} J({\varOmega})
	\coloneqq \argmin_{\omega \in \mathcal{U}} \left\{ \frac{1}{2} \intO{ \left( \Im\{u\} \right)^{2} } \right\}.
\end{equation}	
\end{problem}
Note that identifying $\omega$ is equivalent to identifying $\varOmega$, since $\varOmega = D \setminus \overline{\omega}$ and defining one automatically determines the other.
With a slight abuse of notation, we also write
\[
\varOmega^{\star} = \argmin_{\varOmega \in \mathcal{U}} J(\varOmega),
\]
identifying $\varOmega = D \setminus \overline{\omega}$ for $\omega \in \mathcal{U}$.
Throughout this section, we assume that $\varOmega \in \mathcal{U}$ and consider admissible deformation fields $\VV \in \sfTheta^{1}$ without further notice.
\subsection{ADMM in shape optimization setting}
\label{subsec:ADMM_algorithm}
Our immediate goal is to reformulate Problem~\ref{prob:optimal_shape_problem} into a form amenable to an ADMM-based numerical solution in the spirit of  \cite{CherratAfraitesRabago2026} (see also \cite{RabagoHadriAfraitesHendyZaky2024,CherratAfraitesRabago2025b}).
To this end, we first provide additional motivation for the introduction of the inequality constraint.
Observe from \eqref{eq:shape_problem} that although the boundary data is incorporated through the state equation, it is propagated into the domain via an elliptic PDE, which inherently smooths the data.
As a consequence, fine geometric features of the cavity, such as nonconvex regions and concavities, may be partially obscured, leading to a loss of reconstruction accuracy.
In particular, the optimization relies on a mismatch functional that reflects the data only indirectly, without enforcing the bounds on the state variable within the domain.

The imposed bounds 
\begin{equation}\label{eq:bounds}
a \leqslant \Re\{u\} \leqslant b  \quad \text{a.e. in } \varOmega
\end{equation}
provide additional physically motivated information on the state variable, which can be justified, for instance, by maximum principle arguments for mixed Dirichlet--Robin problems (see, e.g., \cite{Chicco1997}).
By explicitly enforcing these constraints, the reconstruction is guided not only by the governing PDE \eqref{eq:ccbm_state}, but also by admissible ranges of the state, thereby improving stability with respect to noisy data and enhancing sensitivity to geometric features.
To incorporate the constraints efficiently, we introduce an auxiliary variable to decouple the inequality constraint \eqref{eq:bounds} from the state equation, allowing for a more flexible treatment within the optimization procedure.
This leads to an equivalent constrained formulation in which the real part of the state variable is handled separately through a projection onto a suitable admissible set.
We begin by defining a closed convex subset of $L^{2}(\varOmega)$ that encodes the imposed bounds, which will play a central role in the subsequent ADMM construction.

Let $a<b$ be fixed real constants.
Define the admissible convex set
\begin{equation}\label{eq:admissible_set_K}
\mathcal{K}
\coloneqq
\big\{
v \in L^{2}(\varOmega)\; ;\; a \leqslant v \leqslant b \ \text{a.e. in } \varOmega
\big\},
\end{equation}
and the indicator functional $U_{\mathcal{K}}$ of $\mathcal{K}$; that is,
$U_{\mathcal{K}}(v) = 0$ if $v \in \mathcal{K}$, and $U_{\mathcal{K}}(v) = \infty$ if $v \in L^{2}(\varOmega) \setminus \mathcal{K}$.

Introducing an auxiliary variable $v$ associated with the real part $\ur$ of the complex state solution to \eqref{eq:ccbm_state}, we consider the constrained shape optimization problem
\begin{equation}\label{eq:control_Uad}
	(\omega^{\star}, v^{\star})
	=
	\argmin_{(\omega,v) \in \mathcal{E}}
	\left\{
	J({\varOmega}) + U_{\mathcal{K}}(v)
	\right\},
\end{equation}
where the admissible set $\mathcal{E}$ is defined by
\[
\mathcal{E}
\coloneqq
\Big\{
(\omega,v)\in \mathcal{U}\times L^{2}(\varOmega)
\;\Big|\;
\ur(\varOmega)=v \ \text{a.e. in } \varOmega,
\ \ u \text{ solves \eqref{eq:ccbm_state} in } \varOmega
\Big\}.
\]

For a given penalty parameter $\ADMMpara>0$\footnote{This free parameter can, in principle, be optimized via a bilevel optimization approach; however, we fix it here to simplify the discussion, as this choice is already effective for our purposes.} and a Lagrange multiplier
$\lambda \in L^{2}(\varOmega)$, we introduce the functional
\[
	\Upsilon_{\ADMMpara}(\varOmega, v, \lambda)
	= \intO{ \left( \frac{\ADMMpara}{2}  |\ur- v|^2 + \lambda (\ur - v) \right) },
\]
and the augmented Lagrangian
\begin{equation}
\label{eq:augmented_lagrangian}
	{L}_{\ADMMpara}(\omega, v; \lambda)
	=
	J({\varOmega})
	+
	U_{\mathcal{K}}(v)
	+
	\Upsilon_{\ADMMpara}(\varOmega, v, \lambda),
\end{equation}
where $\ur = \ur(\varOmega)$ denotes the real part of the solution of \eqref{eq:ccbm_state} in the annular domain $\varOmega$.

Starting from an initial guess $(\omega^{0}, v^{0}, \lambda^{0})$, the ADMM iterations proceed as follows:
\begin{align}
    \omega^{k+1}
    &= \argmin_{\omega \in \mathcal{U}}
    {L}_{\ADMMpara}(\omega, v^{k}; \lambda^{k}),
    \label{eq:controle}\tag{SP1}
    \\
    v^{k+1}
    &= \argmin_{v \in L^{2}(\varOmega)}
    {L}_{\ADMMpara}(\omega^{k+1}, v; \lambda^{k}),
    \label{eq:etat}\tag{SP2}
    \\
    \lambda^{k+1}
    &= \lambda^{k} + \ADMMpara \big(\ur^{k+1} - v^{k+1}\big),
    \label{eq:parametre1}\tag{SP3}
\end{align}
where $\ur^{k+1} \coloneqq \ur(\varOmega^{k+1})$, and $k \in \mathbb{N} \cup \{0\}$.

\begin{remark}
\begin{itemize}
\item The shape subproblem \eqref{eq:controle} is solved using the same shape derivative,
adjoint formulation, and Sobolev gradient method as in the unconstrained case; only an additional
quadratic term appears in the objective functional.
\item The auxiliary variable update \eqref{eq:etat} admits the explicit solution given by the projection formula \eqref{eq:projection_formula}, corresponding to the projection of $\ur^{k+1} + \lambda^{k}/\ADMMpara$ onto the admissible set $\mathcal{K}$.
\item The proposed ADMM scheme does not require the explicit computation of the material
derivative of the state variable.
\end{itemize}
\end{remark}

A concise description of the complete ADMM procedure is summarized in
Algorithm~\ref{algo:ADMM_algorithm}. 

\begin{algorithm}
\begin{enumerate}\itemsep0.1em 
	\item \textit{Input.} Fix $\ADMMpara$, $a$, and $b$, and define the Cauchy pair $(f, g)$.
	\item \textit{Initialization.} Choose an initial shape $\omega^{0}$ and set the initial values
	$v^{0}$ and $\lambda^{0}$.
	\item \textit{Iteration.} For $k = 0,1,2,\ldots$, compute $(\omega^{k+1}, v^{k+1}, \lambda^{k+1})$
	using equations \eqref{eq:controle}--\eqref{eq:parametre1} through the following sequential steps:
	\[
		\{v^{k}, \lambda^{k}\}
		\ \stackrel{\eqref{eq:controle}}{\longrightarrow} \ 
		\omega^{k+1}
		\ \stackrel{\eqref{eq:etat}}{\longrightarrow} \ 
		v^{k+1}
		\ \stackrel{\eqref{eq:parametre1}}{\longrightarrow} \ 
		\lambda^{k+1}.
	\]
	\item \textit{Stopping criterion.} Repeat \textit{Iteration} until convergence.
\end{enumerate}
\caption{ADMM algorithm for the solution of problem \eqref{eq:control}.}
\label{algo:ADMM_algorithm}
\end{algorithm}

The resolution of subproblems \eqref{eq:controle} and \eqref{eq:etat} is presented in the following two subsections.

\subsection{Solution of the $\omega$-subproblem \eqref{eq:controle}}
\label{subsec:omega_subproblem}

We first consider the solution of the $\omega$-subproblem \eqref{eq:controle}, which consists in
minimizing the augmented Lagrangian ${L}_{\ADMMpara}$ with respect to the shape variable $\omega$.
The $\omega$-subproblem reads
\[
 \omega^{k+1}
 =
 \argmin_{\omega \in \mathcal{U}}
 \left\{
 J({\varOmega})
 + U_{\mathcal{K}}(v^{k})
 + \Upsilon(\varOmega, v^{k}, \lambda^{k}; \ADMMpara)
 \right\}.
\]

Since the indicator function $U_{\mathcal{K}}(v^{k})$ does not depend on the shape variable $\omega$, it does not contribute to the shape derivative and can be omitted in its computation.
We therefore introduce, for $k\in \mathbb{N} \cup \{0\}$, the following shape functional:
\[
Y^{k}({\varOmega})
\coloneqq
J({\varOmega})
+
\Upsilon(\varOmega, v^{k}, \lambda^{k}; \ADMMpara).
\]
For ease of writing in the subsequent discussion, we denote
\[
	Y^{k}({\varOmega})
	= \intO{ \Phi_{\ADMMpara}(u, v^{k}, \lambda^{k}) } 
	\coloneqq \intO{
	\left( 
	\frac12 |\ui |^2
        +
        \frac{\ADMMpara}{2}|\ur -v^{k}|^2
        +
        \lambda^{k}(\ur -v^{k})
	\right)
	}.
\]

To solve \eqref{eq:controle} numerically, it is necessary to compute the shape
derivative of the functional $Y^{k}(\varOmega)$.

\begin{proposition}
\label{prop:shape_gradient_Yk}
For all $k \in \mathbb{N} \cup \{0\}$, the functional $Y^{k}(\varOmega)$ is shape differentiable, and its shape derivative at $\varOmega \in \mathcal{U}$ in the direction $\VV \in \sfTheta^{1}$ admits a boundary representation 
\[
dY^k(\Omega)[\theta] = \intG{ \GG^{k}_q \Vn},
\]
with shape gradient
\begin{equation}
\label{eq:Gq_prop}
\GG^{k}_q
=
- \Psi^{\dagger}_{\varGamma}(u, q) 
+ 
\Phi_{\ADMMpara}(u, v^{k}, \lambda^{k}),
\end{equation}
where $\kappa=\tandive\nn$ is the mean curvature of $\varGamma$,
$u=\ur +i \ui \in\HH^{2}(\varOmega)$ is the unique solution of the state problem
\eqref{eq:ccbm_state}, and the adjoint state $q=\qr +i \qi \in \HH^{2}(\varOmega)$
is the unique solution of
\begin{equation}
\label{eq:adjoint_ADMM_prop_strong}
-\,\Delta q
=
\ADMMpara(\ur -v^{k})+\lambda^{k}
+i\ui 
\ \text{in } \varOmega,
\qquad
\dn q - i q = 0 \ \text{on } \varSigma,
\qquad
\dn q + \alpha q = 0 \ \text{on } \varGamma.
\end{equation}
\end{proposition}

Expression \eqref{eq:Gq_prop} can be obtained similarly to the proof of Proposition \ref{prop:Jmaps}, so we omit it.

The following remarks summarize alternative boundary representations of the shape derivative associated with different adjoint formulations (and all other quantities are as defined in Proposition~\ref{prop:shape_gradient_Yk}.).
\begin{remark}\label{rem:shape_gradient_Yk_alt}
An alternative expression for the shape derivative can be obtained by considering a different adjoint system. 
More precisely, letting $w = \rw + i \iw$ be the unique weak solution of the adjoint system
\begin{equation}\label{eq:adjoint_ADMM_choice}
-\,\Delta w
=
i\big(\ADMMpara(\ur - v^{k}) + \lambda^{k}\big) - \ui
\ \text{in } \varOmega, \qquad
\dn w - i w = 0 \ \text{on } \varSigma, \qquad
\dn w + \alpha w = 0 \ \text{on } \varGamma,
\end{equation}
the corresponding shape gradient becomes
\begin{equation}\label{eq:shape_gradient_Yk}
\GG^{k}_w
=
\Phi_{\ADMMpara}(u, v^{k}, \lambda^{k})
+
\Psi_{\varGamma}(u, w).
\end{equation}

\end{remark}
\begin{remark}\label{prop:shape_gradient_Yk_crossed}
Alternatively, the adjoint system can be split into two components associated with the real and imaginary parts of the state.
The corresponding shape gradient is then given by
\begin{equation}\label{eq:shape_gradient_Yk_split}
\GG^{k}_{\Lambda1}
=
\GG_{1}
+
\frac{\ADMMpara}{2}|\ur-v^{k}|^2
+
\lambda^{k}(\ur-v^{k})
-
\Psi^{\dagger}_{\varGamma}(u, \Lambda),
\end{equation}
where $\GG_{1} = \GG_{1}(u, p)$ is given by \eqref{eq:shape_gradient} while $\Lambda=\rlam+i \ilam$ is the unique weak solution to 
\begin{equation}
\label{eq:adjoint_Lambda_crossed}
-\Delta \Lambda
=
\ADMMpara(\ur-v^{k})+\lambda^{k}
\ \text{in } \varOmega,
\qquad
\dn \Lambda - i \Lambda = 0 \ \text{on } \varSigma,
\qquad
\dn \Lambda + \alpha \Lambda = 0 \ \text{on } \varGamma.
\end{equation}
\end{remark}

\begin{remark}
\label{prop:shape_gradient_Yk_separate}
Another expression for the shape gradient that is associated with the results stated in Remark~\ref{rem:Jmaps} is given by
\begin{equation}\label{eq:shape_gradient_Yk_split_2}
\GG^{k}_{\Lambda2} 
= \GG_2 + \GG^{k}_{\Lambda1}
-
\GG_{1},
\end{equation}
where $\GG_{2} = \GG_{2}(u, p)$ is given by \eqref{eq:GG_correct}.
\end{remark}

Inspection of \eqref{eq:shape_gradient_Yk_split} and \eqref{eq:shape_gradient_Yk_split_2} reveals a natural decoupling into two contributions: one term independent of the adjoint variable $\Lambda$, and another that does not depend on $p$. 
This structural feature provides a clear rationale for introducing partial shape gradients, in which only selected components of the full shape gradient are retained to define the deformation field in the mesh update.

More precisely, one may restrict attention to the following expressions:
\begin{align}
\GG^{k}_{\sharp1}
&=
\GG_{1}
+
\frac{\ADMMpara}{2}|\ur-v^{k}|^2
+
\lambda^{k}(\ur-v^{k})\label{eq:another_set_of_choices_for_G1},\\
\GG^{k}_{\sharp2}
&=
\GG_{2}
+
\frac{\ADMMpara}{2}|\ur-v^{k}|^2
+
\lambda^{k}(\ur-v^{k}).\label{eq:another_set_of_choices_for_G2}
\end{align}
A related strategy was employed in \cite{HIKKP2009} for the exterior Bernoulli free boundary problem. However, we emphasize that the present inverse geometry problem is severely ill-posed, in contrast to the Bernoulli case, which is only mildly ill-posed. 
In this context, the use of partial shape gradients may degrade the quality of the descent direction and, consequently, the reconstruction, unless suitable stabilization mechanisms are incorporated.

The idea of selecting only certain components of the shape gradient within gradient-based algorithms has also been explored in \cite{CherratAfraitesRabago2026}; see also \cite{RabagoKimura2025}. 
Accordingly, within the CCBM--ADMM framework, one may consider the following six choices for the normal velocity of the free boundary in the boundary deformation:
\[
	G_{l}^{k}\nn, \qquad l \in \{ q,w,\Lambda1, \Lambda2,\sharp1, \sharp2\},
\]
where the expressions $G_{l}^{k}$, for $l=q,w,\Lambda1, \Lambda2,\sharp1, \sharp2$, are given respectively by \eqref{eq:Gq_prop}, \eqref{eq:shape_gradient_Yk}, \eqref{eq:shape_gradient_Yk_split}, \eqref{eq:shape_gradient_Yk_split_2}, \eqref{eq:another_set_of_choices_for_G1}, and \eqref{eq:another_set_of_choices_for_G2}.

%
\subsection{Extension and regularization of the deformation field}
The numerical implementation of the normal velocity $G_{l}^{k}\nn$, $l \in \{ q,w,\Lambda1, \Lambda2,\sharp1, \sharp2\}$, in mesh deformation requires computing its $H^{1}$ Riesz representative to obtain a smooth deformation field over the computational mesh (see Section~\ref{sec:Numerical_Approximation}), thereby defining nodal velocities both on the boundary and in the interior. 
This construction naturally leads to the SGBD algorithm in Algorithm~\ref{algo:SGBD_algorithm} for solving \eqref{eq:controle}.
\begin{algorithm}[h!]
\caption{SGBD algorithm for the $\omega$-subproblem \eqref{eq:controle}}\label{algo:SGBD_algorithm}
\begin{enumerate}

\item \textit{Input:} Fix ${\beta}$, $\mu$, $\ADMMpara$, $a$, $b$, and $\varepsilon$. Set $\lambda^{k}$ and initialize $\varOmega^{k}_0=\varOmega^{k}$.
\item \textit{Iteration:} For $m=0,1,2,\ldots$,
\begin{enumerate}\itemsep0.1em
\item[2.1] Solve the state system and the associated adjoint system(s) on $\varOmega_m^k$;
\item[2.2] Compute the descent direction $\VV_m^k$ from \eqref{eq:Sobolev_gradient_computation}, with $G$ replaced by one of $G_q^k$, $G_w^k$, $G_{\Lambda1}^k$, $G_{\Lambda2}^k$, $G_{\sharp1}^k$, or $G_{\sharp2}^k$, depending on the chosen shape gradient;
\item[2.3] Set $t^k=\mu J^{k}(\varOmega_m^k)/\|\VV_m^k\|_{H^{1}(\varOmega_m^k)^d}$ and update
\[
\varOmega_{m+1}^k=\{x+t^k \VV_m^k(x)\mid x\in \varOmega_m^k\}.
\]
\end{enumerate}

\item \textit{Stop test:} Repeat \textit{Iteration} until
\[
\|dY^{k}(\varOmega_m^k)[\VV_m^k]\|<\varepsilon.
\]

\item \textit{Output:} $\varOmega^{k+1}=\varOmega_{m+1}^k$.

\end{enumerate}
\end{algorithm}

%
%
%
%
%
\subsection{Solution of the $v$-subproblem \eqref{eq:etat}}
We next solve the $v$-subproblem \eqref{eq:etat} by minimizing ${L}_{\ADMMpara}$ given by \eqref{eq:augmented_lagrangian} with respect to $v$, that is,
\begin{align*}
v^{k+1} 
&= \argmin_{v \in L^{2}({\varOmega})}
\Big\{ J(\varOmega^{k+1})+U_{\mathcal{K}}(v)
+ \frac{\ADMMpara}{2}  \intO{\vert {\ur^{k+1}}-v \vert^2} 
+ \intO{\lambda^{k}  ( {\ur^{k+1}}-v )} 
\Big\}.
%
\end{align*}

Applying the projection method, we obtain
\begin{equation}\label{eq:projection_formula}
	v^{k+1} = P_{\mathcal{K}}\left({\ur^{k+1}} + \dfrac{\lambda^{k}}{\ADMMpara} \right),
\end{equation}
where $P_{\mathcal{K}}(\varphi) \coloneqq \max(a, \min(b, \varphi))$ for all $\varphi \in L^{2}({\varOmega})$ is the projection operator onto the admissible set $\mathcal{K}$.
\subsection{ADMM-SGBD algorithm}
Finally, building upon the preceding discussion, we propose a modification of Algorithm~\ref{algo:ADMM_algorithm} for the numerical solution of the constrained shape optimization problem \eqref{eq:control}, incorporating an inequality constraint governed by \eqref{eq:real_state}. 
More specifically, Algorithm~\ref{algo:ADMM_algorithm} is refined into a nested ADMM--SGBD scheme tailored to the problem \eqref{eq:control_Uad}, as described in Algorithm~\ref{algo:ADMM-SGBD}.
\begin{algorithm}[!htp] 
\caption{ADMM--SGBD}\label{algo:ADMM-SGBD}
\begin{enumerate} \itemsep0.1em 
\item \textit{Initialization:} Given $(f,g)$, fix $N \in \mathbb{N}$ and parameters $\ADMMpara$, $a$, $b$, $\mu$, $\beta$, $\varepsilon$. Choose initial guesses $\omega^{0}$, $v^{0}$, and $\lambda^{0}$.

\item \textit{Iteration:} For $k=0,\ldots,N$,
\begin{enumerate}\itemsep0.1em
\item[2.1] Update $\omega^{k+1}$ by solving the $\omega$-subproblem with data $(v^k,\lambda^k)$ using Algorithm~\ref{algo:SGBD_algorithm};
\item[2.2] Compute $u^{k+1}$ by solving \eqref{eq:state_weak_form} on $\varOmega^{k+1}$, where $\varOmega^{k+1} = D \setminus \overline{\omega}^{k+1}$, and set $\ur^{k+1}=\Re\{u(\varOmega^{k+1})\}$;
\item[2.3] Update $v^{k+1}$ by projection:
\[
v^{k+1}=\max \big( a, \min (\ur^{k+1}+\lambda^{k}/\ADMMpara, b ) \big);
\]
\item[2.4] Update the multiplier:
\[
\lambda^{k+1}=\lambda^{k}+\ADMMpara (\ur^{k+1}-v^{k+1});
\]
\end{enumerate}

\item \textit{Stop test:} Repeat \textit{Iteration} until convergence.
\end{enumerate}
\end{algorithm}
%
%
%


\subsection{Numerical tests and discussion for ADMM}
\label{subsec:numerical_examples_ADMM} 
We test the proposed numerical method given by Algorithm~\ref{algo:ADMM-SGBD}, focusing on the effect of the weight parameter $\rho > 0$ in the CCBM formulation of the state \eqref{eq:ccbm_state}. 
The parameter $\rho$ appears only in the weak form of the complex state system \eqref{eq:ccbm_state} and does not enter explicitly in the expressions of the shape gradients. 
The test problem is essentially the same as in Subsection~\ref{subsec:numerical_examples}, except that we set $\alpha = 1$ and examine the influence of the Dirichlet data $f = 1$ and $f = f_{1} \coloneqq  \cos{\arctan(x_{2}/x_{1})}$ for $(x_{1}, x_{2}) \in \partial D$, where $D = C(0,0,1)$ denotes the unit circle.
In all tests, we set $\beta = 0.9$ in \eqref{eq:Sobolev_gradient_computation}, and $\gamma = 0.001$, $\lambda^{0} = 0.001$, and $v^{0} = 1$ in Algorithm~\ref{algo:ADMM-SGBD}.
For simplicity, we take $a = \min_{x \in \varOmega^{\star}} u^{\star}(x)$ and $b = \max_{x \in \varOmega^{\star}} u^{\star}(x)$ in the inequality constraint defining the admissible set $\mathcal{K}$ in \eqref{eq:admissible_set_K}.
Lastly, the iteration stops after $N = 1000$ iterations.

We first examine the effect of using different deformation fields, as discussed in subsection~\ref{subsec:omega_subproblem}. 
In particular, we compare \eqref{eq:Gq_prop}, which involves a single adjoint problem (denoted ADMM-joint or ADMMj), with \eqref{eq:shape_gradient_Yk_split}, which requires two adjoint problems (denoted ADMM-split or ADMMs). 
The reconstructions for $f=1$ are displayed in Figure~\ref{ADMMfig:gradient_test_1}, while the corresponding histories of the cost functional, gradient norm, and Hausdorff distance are shown in Figure~\ref{ADMMfig:gradient_test_1_cost_norm_hd}. 
Both approaches yield nearly identical reconstructions and exhibit very similar convergence behavior across all metrics, suggesting that the additional adjoint solve in the split formulation does not provide a significant advantage for the present test case.

We next investigate the influence of the parameter $\rho$. 
For small values such as $\rho = 1$, the method exhibits limited sensitivity to the nonconvex part of the geometry and does not clearly recover the cavity concavity. 
As $\rho$ increases, the reconstructions improve qualitatively, and for $\rho \geqslant 5$, a slight indentation becomes visible in the corresponding region. 
This indicates that larger values of $\rho$ enhance the sensitivity of the reconstruction to finer geometric features. 
Nevertheless, accurately recovering the depth and sharpness of the concavity remains difficult even for larger $\rho$, reflecting the inherent ill-posedness of the inverse problem. 
Overall, the numerical results suggest that the CCBM formulation may provide improved qualitative reconstruction properties relative to the classical approaches numerically investigated in \cite[Figs.~1--2]{AfraitesRabago2025}.

To further assess the observed limitation, we repeat the experiment with the alternative Dirichlet data $f = \cos{\arctan(x_{2}/x_{1})}$. 
The resulting reconstructions are shown in Figure~\ref{ADMMfig:gradient_test_2}, with the corresponding histories in Figure~\ref{ADMMfig:gradient_test_2_cost_norm_hd}. 
In this case, the reconstructions are noticeably improved, particularly in capturing the concave features. 
This improvement is consistent with the observed decrease in both the cost and the Hausdorff distance. 
Nevertheless, the difference between the joint and split approaches remains negligible, as both methods exhibit almost identical behavior.

The numerical results indicate that the choice of prescribed Dirichlet data has a significant impact on reconstruction quality, with more informative inputs leading to improved recovery of geometric features, particularly concavity. The weight parameter $\rho$ also plays a crucial role: small values yield limited sensitivity, while larger values enhance the reconstruction, although accurate recovery of concavity depth remains challenging. In contrast, the use of a single adjoint formulation or its split counterpart produces nearly identical reconstructions and convergence histories, suggesting that the additional computational cost of the split approach does not provide a clear benefit in this setting.

\begin{figure}[h!]
    \centering 
    \resizebox{0.24\textwidth}{!}{\includegraphics{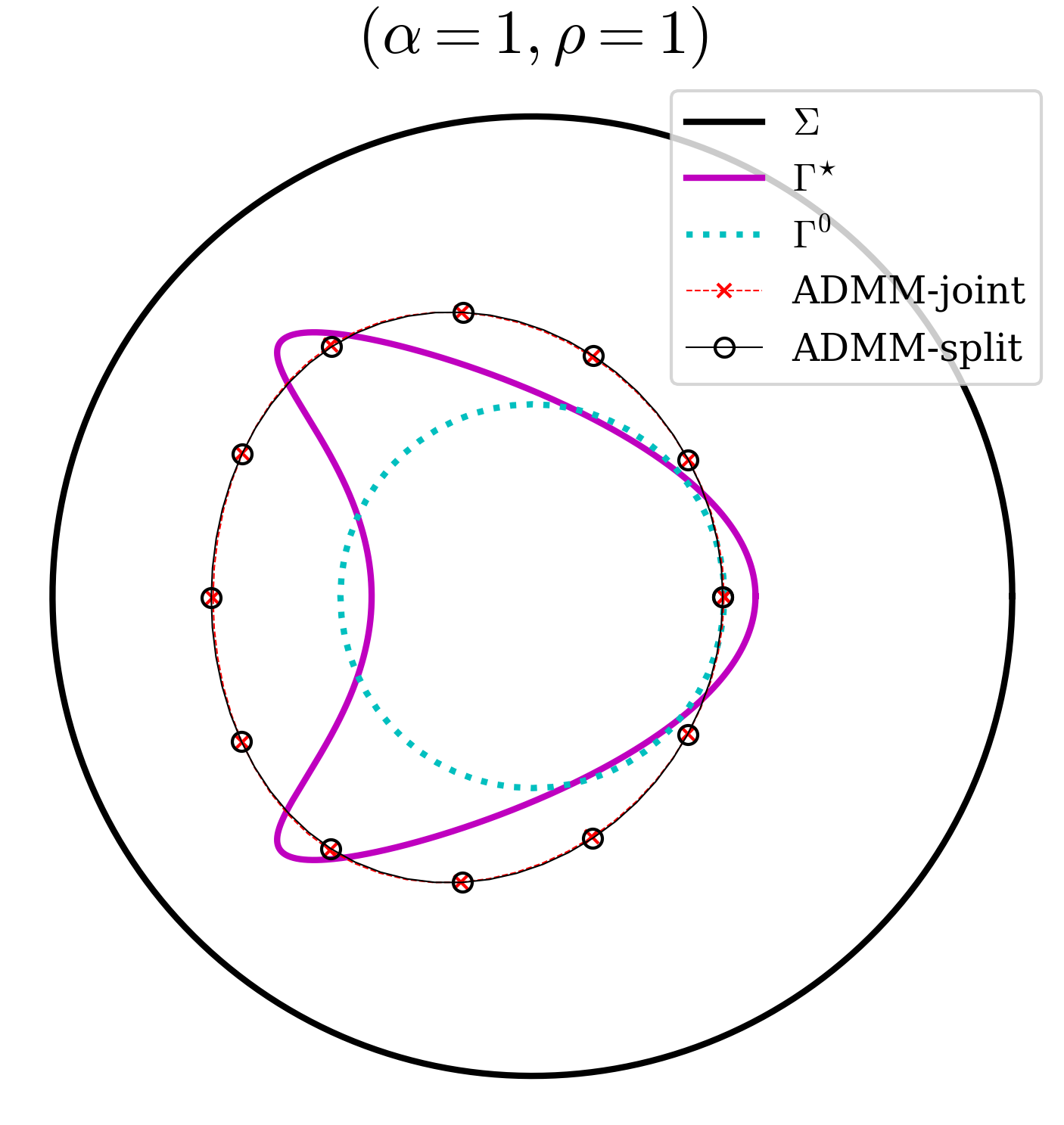}}
    \resizebox{0.24\textwidth}{!}{\includegraphics{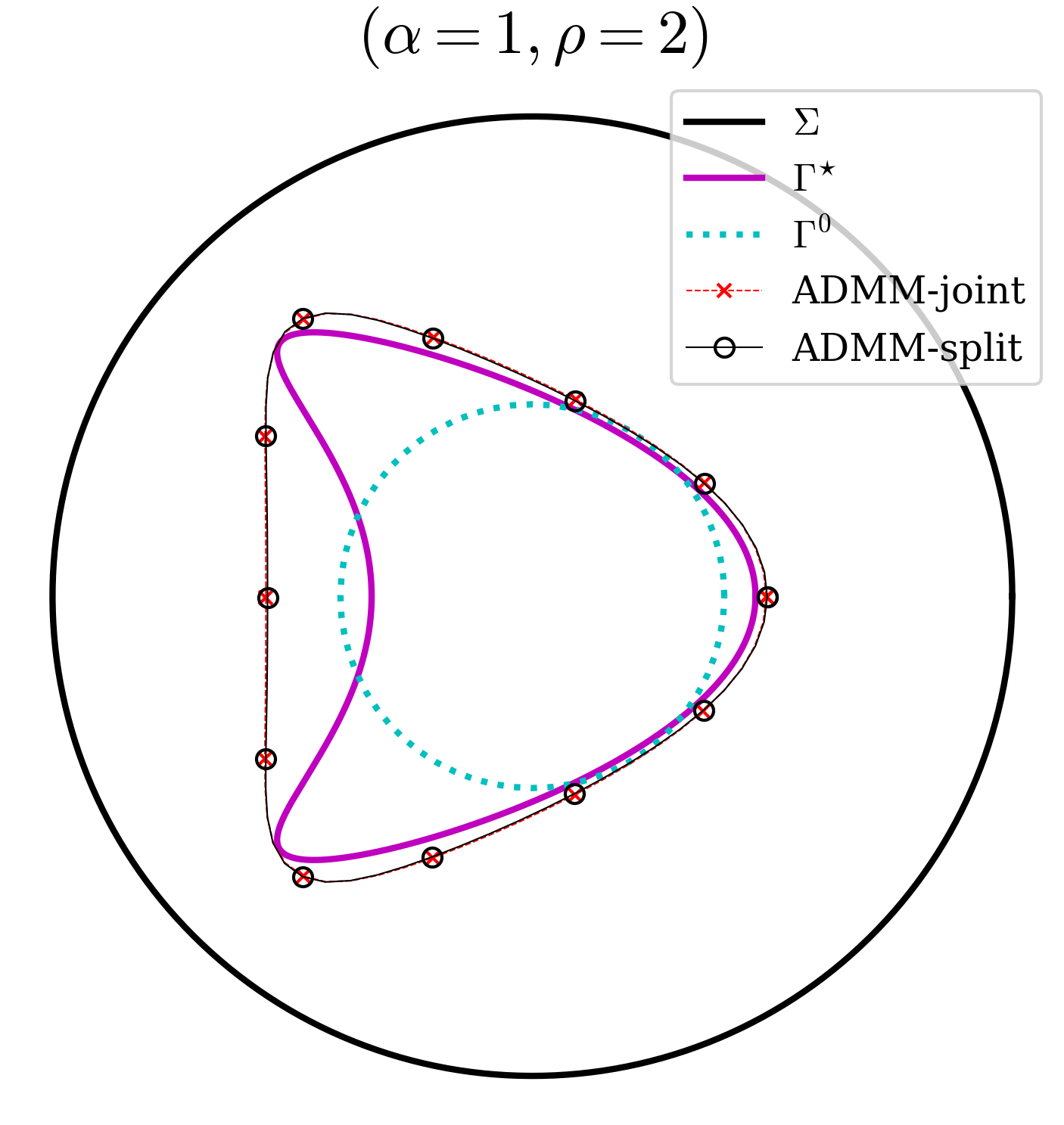}}
    \resizebox{0.24\textwidth}{!}{\includegraphics{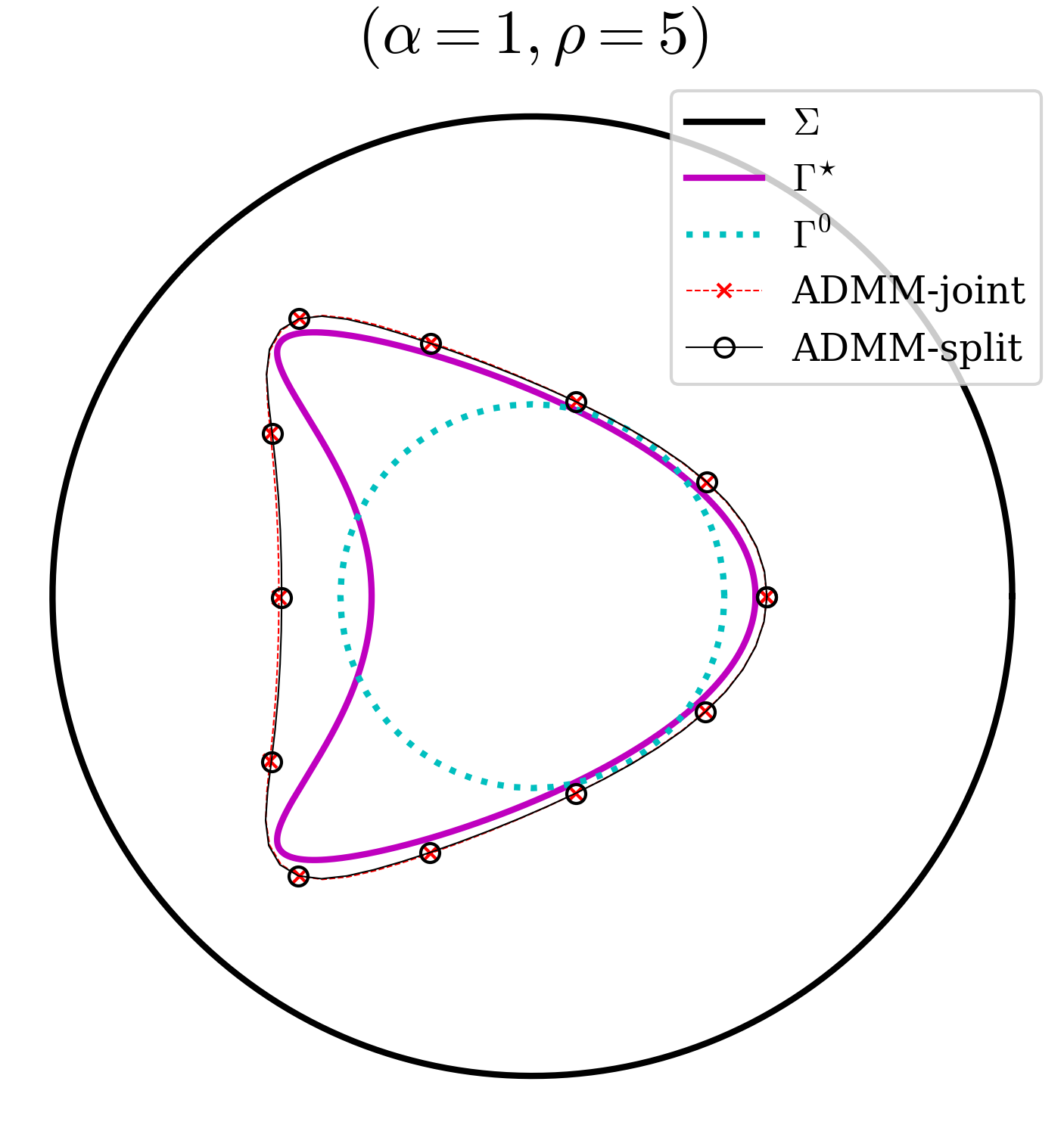}}
    \resizebox{0.24\textwidth}{!}{\includegraphics{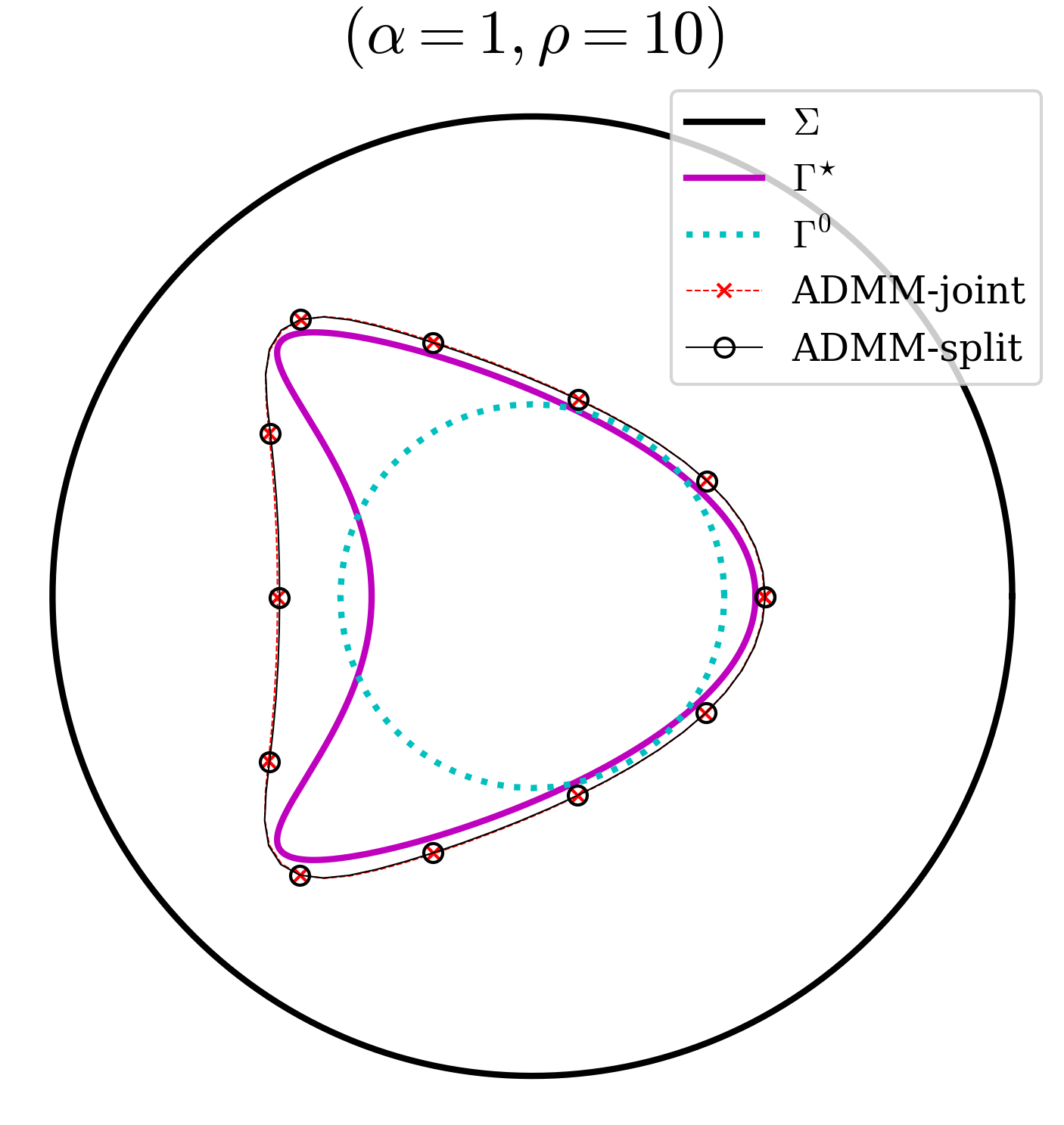}} 
\caption{
Comparison of the ADMM-based reconstruction using \eqref{eq:Gq_prop} and \eqref{eq:shape_gradient_Yk_split}, tested with different values of $\rho = 1,2,5,10$ under exact measurements with $f=1$.
}
\label{ADMMfig:gradient_test_1}
\end{figure}

\begin{figure}[h!]
    \centering 
    \resizebox{0.31\textwidth}{!}{\includegraphics{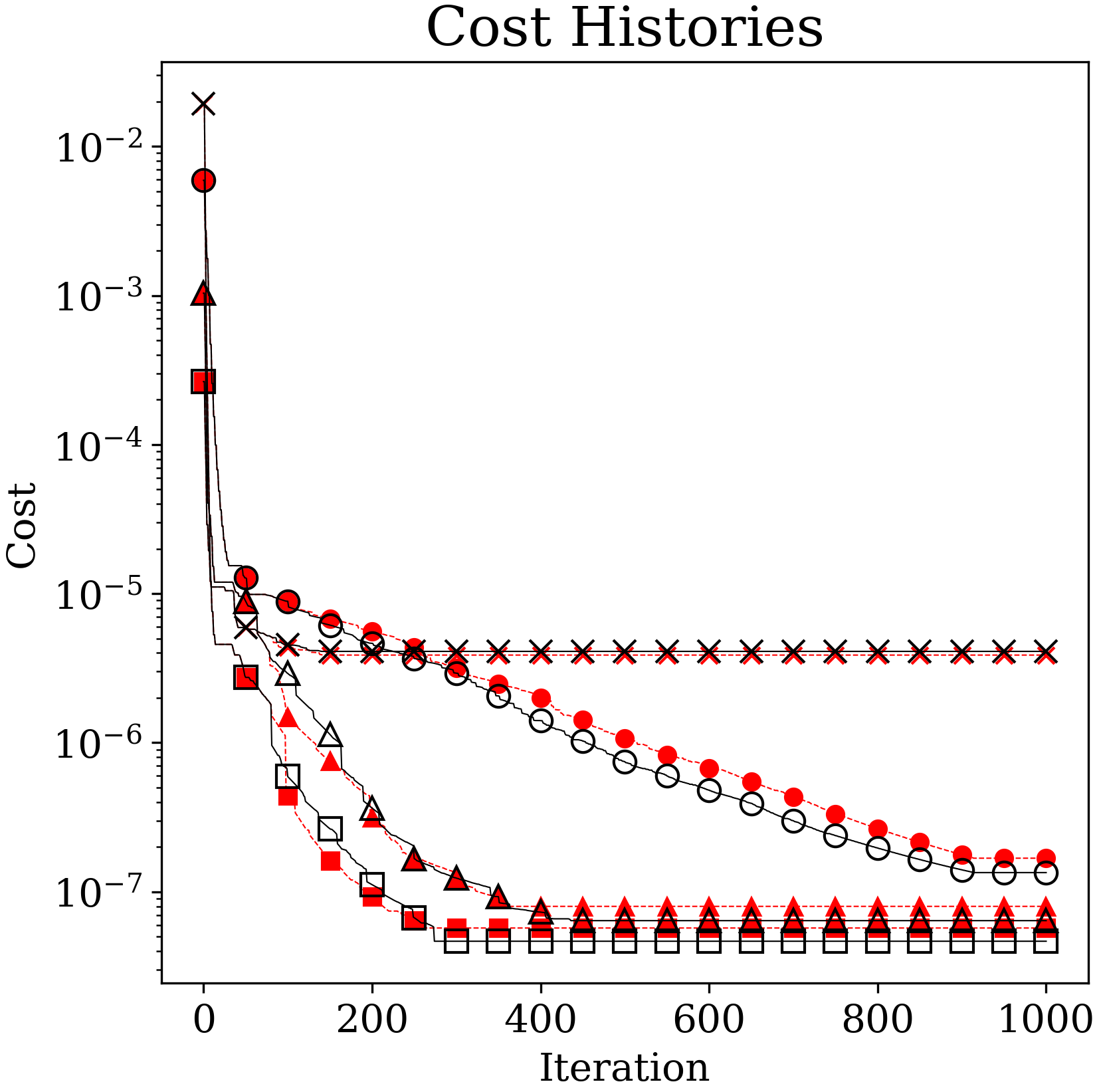}} \hfill
    \resizebox{0.31\textwidth}{!}{\includegraphics{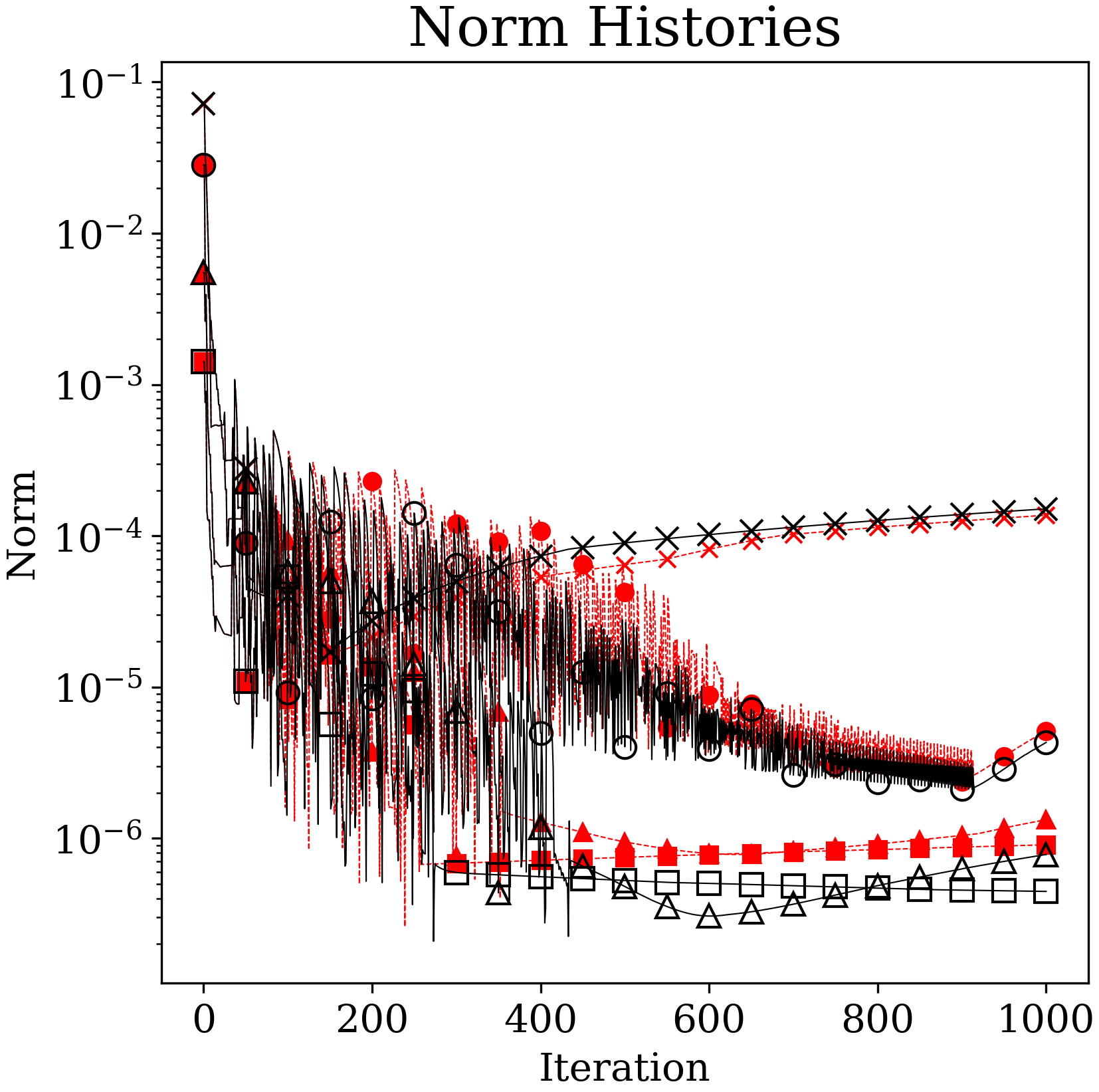}}\hfill
    \resizebox{0.32\textwidth}{!}{\includegraphics{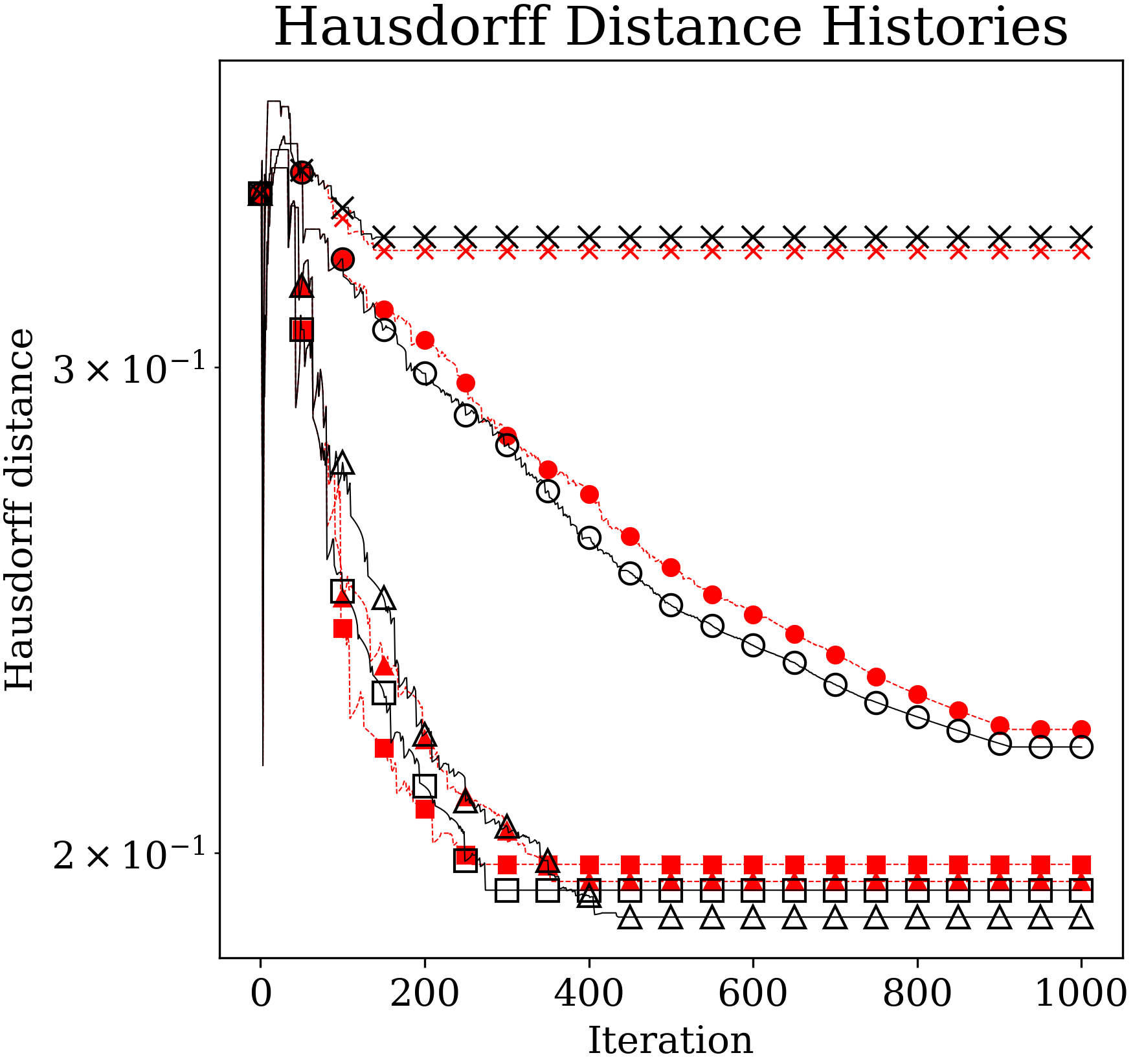}} \\[4pt]
    \resizebox{0.8\textwidth}{!}{\includegraphics{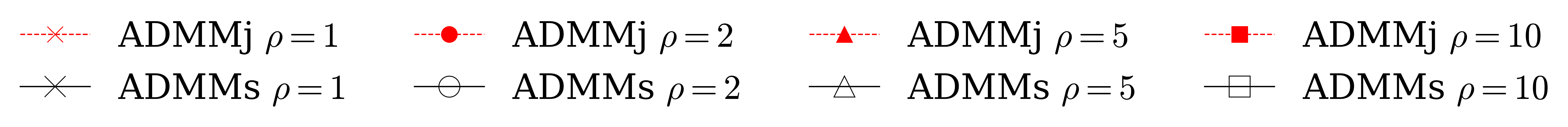}}    
\caption{
Corresponding histories of values for the cost, gradient norm, and Hausdorff distances for Figure~\ref{ADMMfig:gradient_test_1}.
}
\label{ADMMfig:gradient_test_1_cost_norm_hd}
\end{figure}

\begin{figure}[h!]
    \centering 
    \resizebox{0.24\textwidth}{!}{\includegraphics{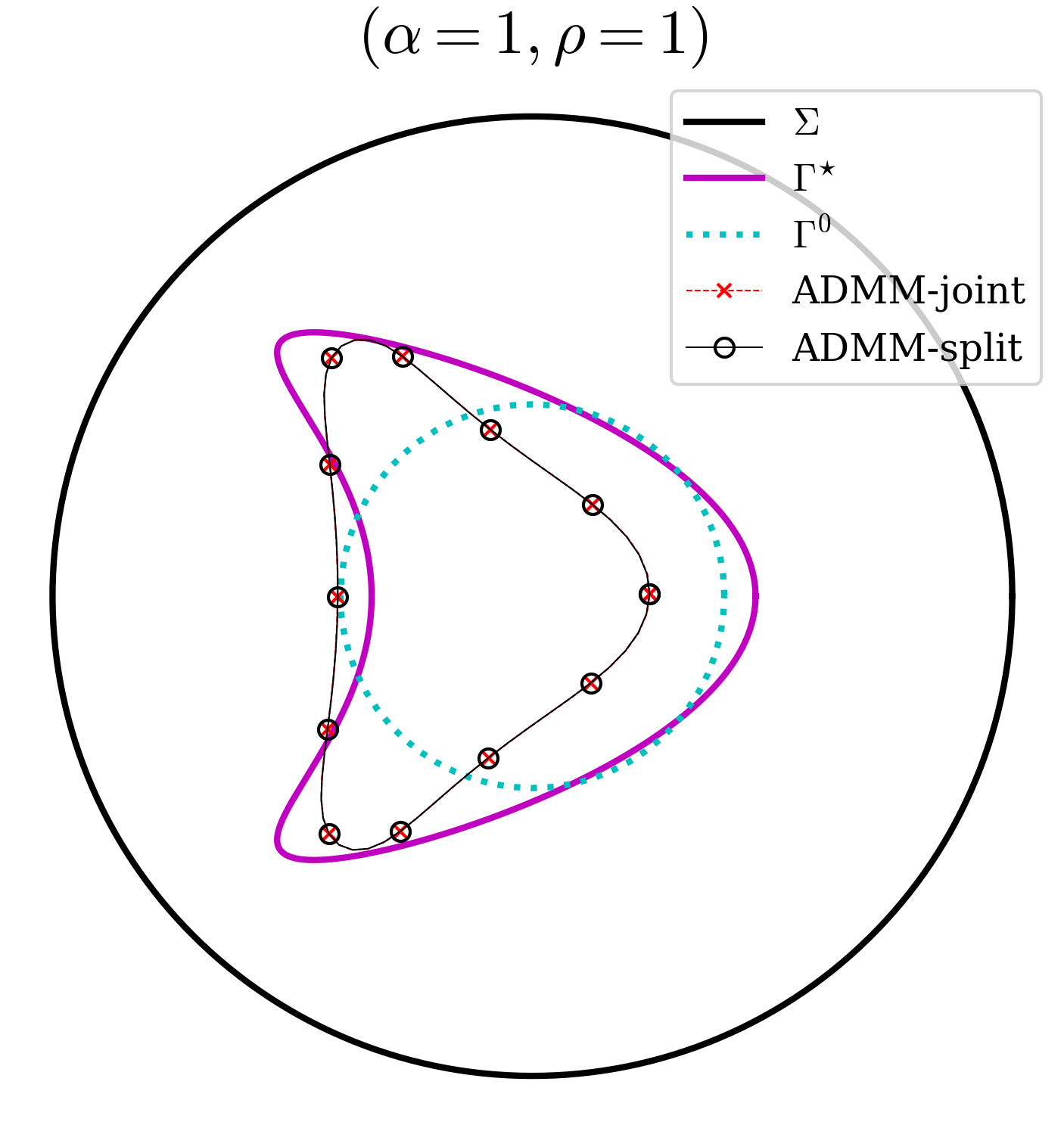}}
    \resizebox{0.24\textwidth}{!}{\includegraphics{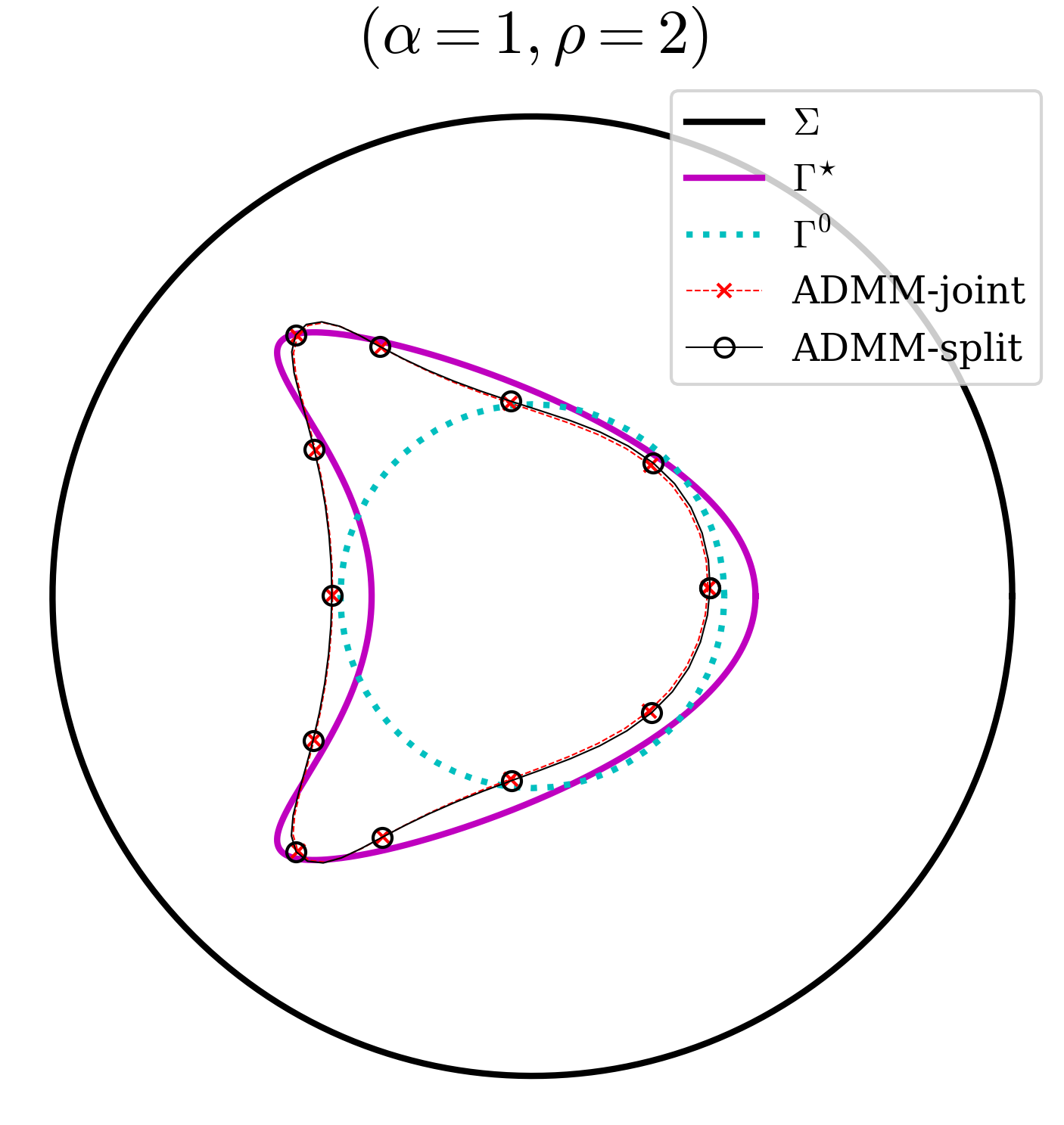}}
    \resizebox{0.24\textwidth}{!}{\includegraphics{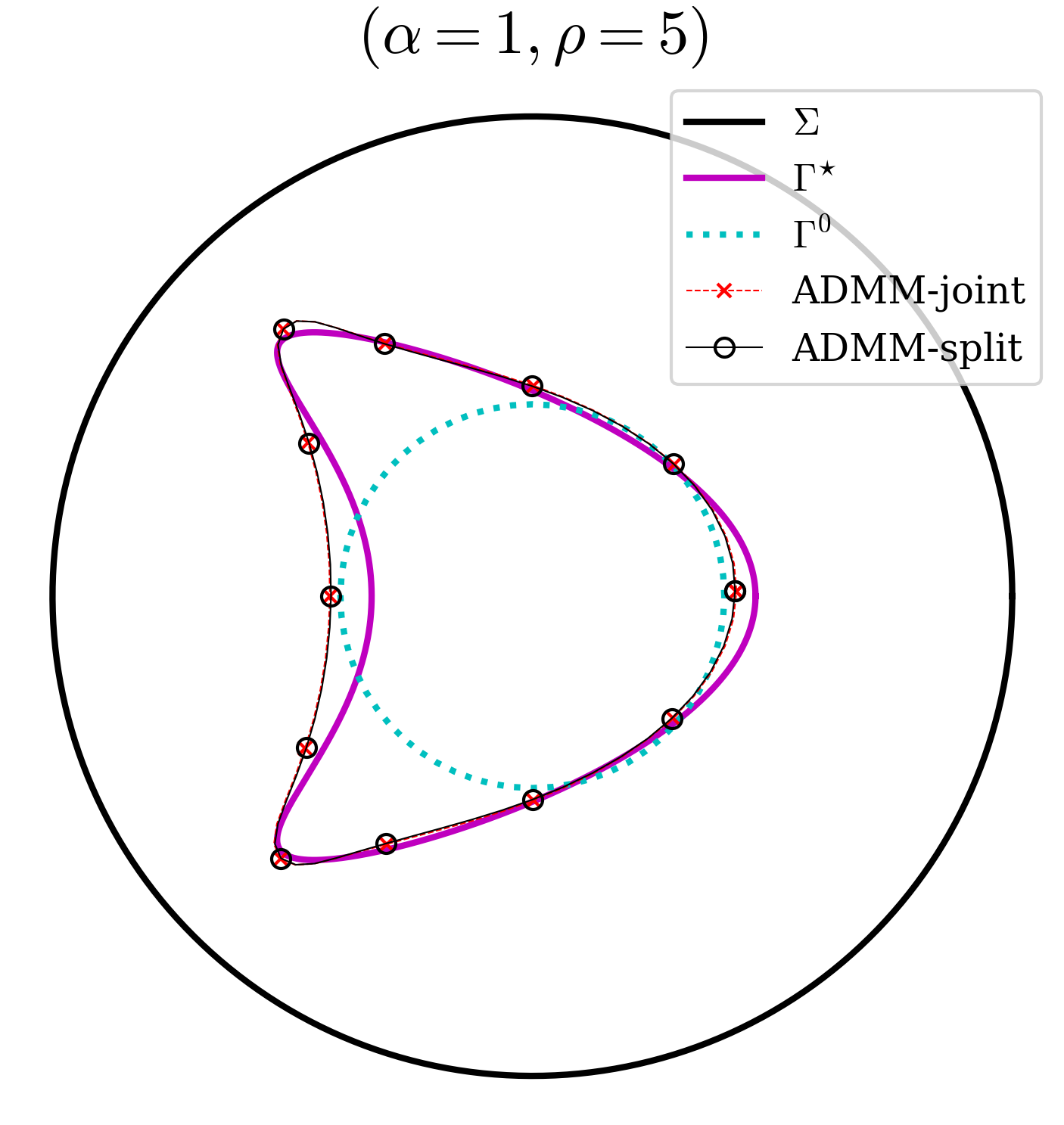}}
    \resizebox{0.24\textwidth}{!}{\includegraphics{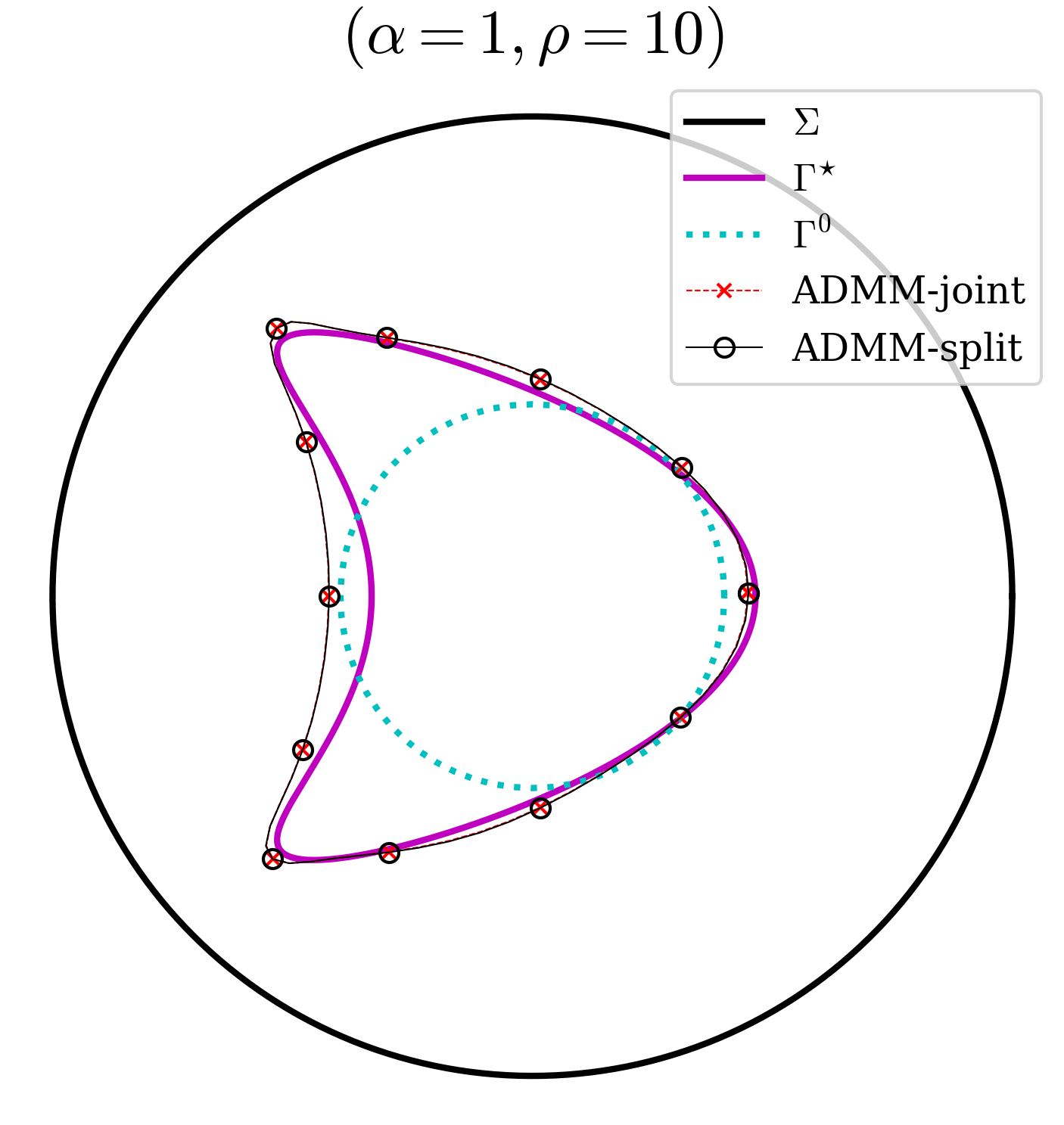}} 
\caption{
Comparison of the ADMM-based reconstruction using \eqref{eq:Gq_prop} and \eqref{eq:shape_gradient_Yk_split}, tested with different values of $\rho = 1,2,5,10$ under exact measurements with $f=\cos{\arctan(x_{2}/x_{1})}$.
}
\label{ADMMfig:gradient_test_2}
\end{figure}

\begin{figure}[h!]
    \centering 
    \resizebox{0.31\textwidth}{!}{\includegraphics{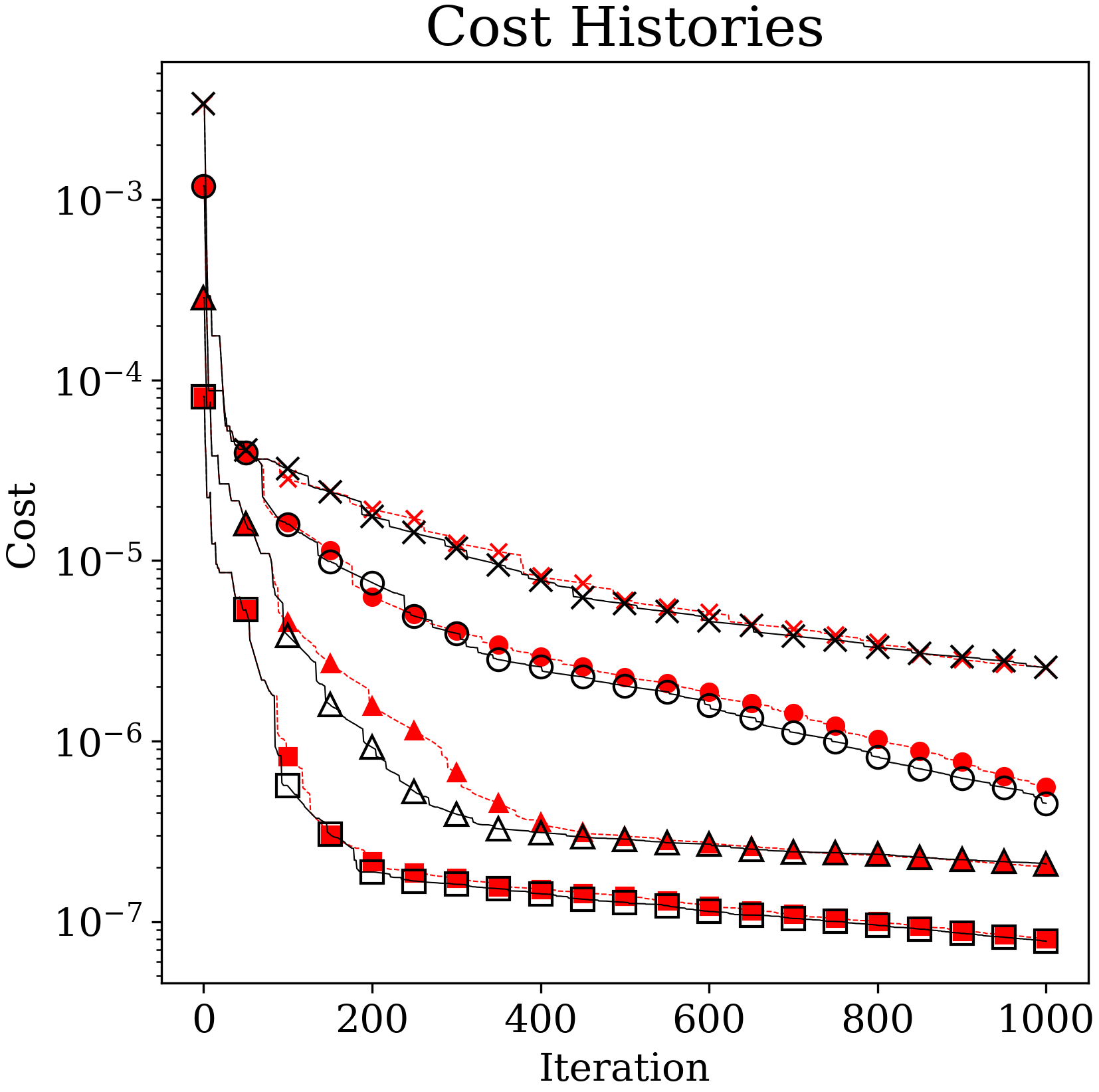}} \hfill
    \resizebox{0.31\textwidth}{!}{\includegraphics{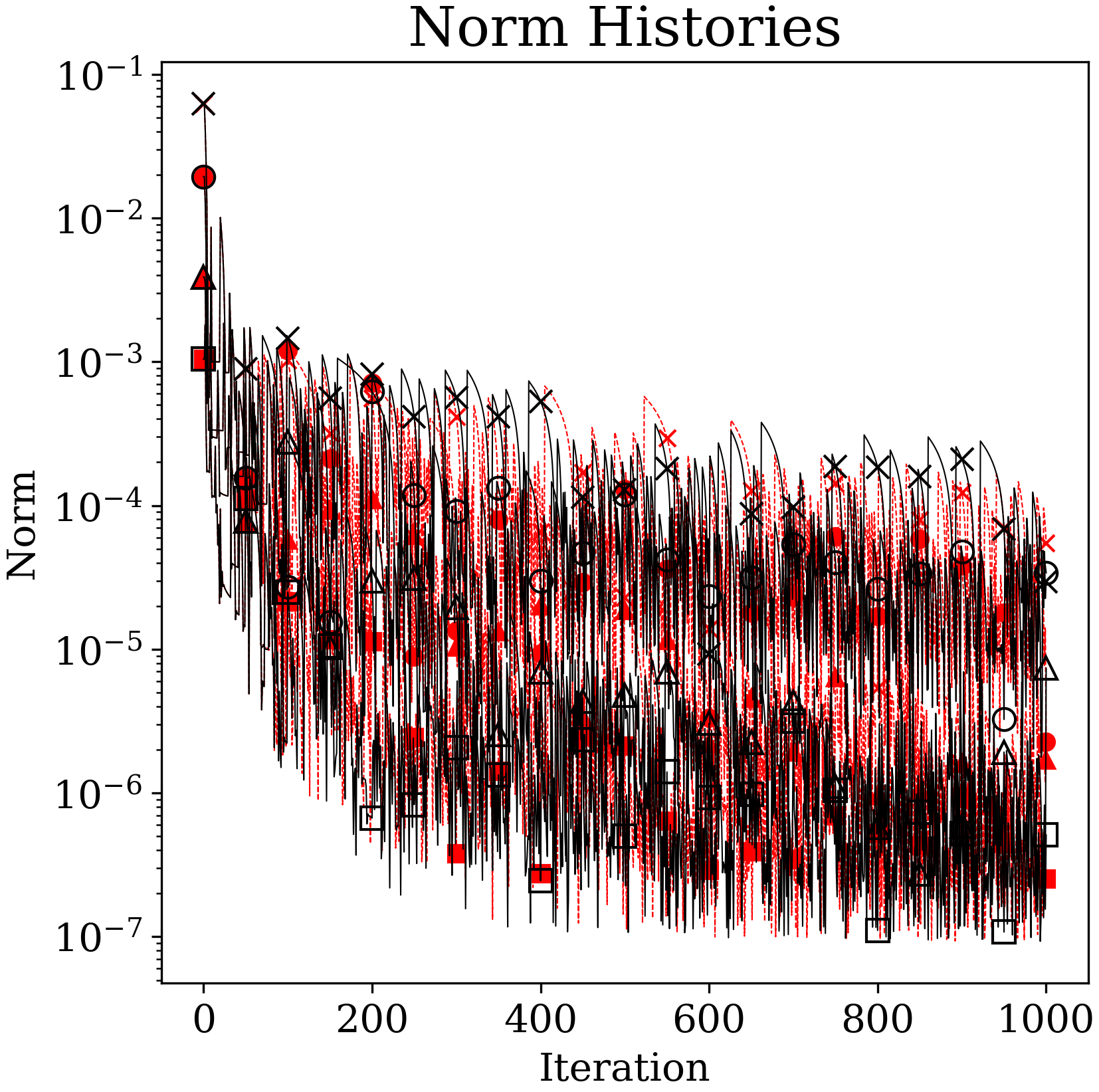}}\hfill
    \resizebox{0.33\textwidth}{!}{\includegraphics{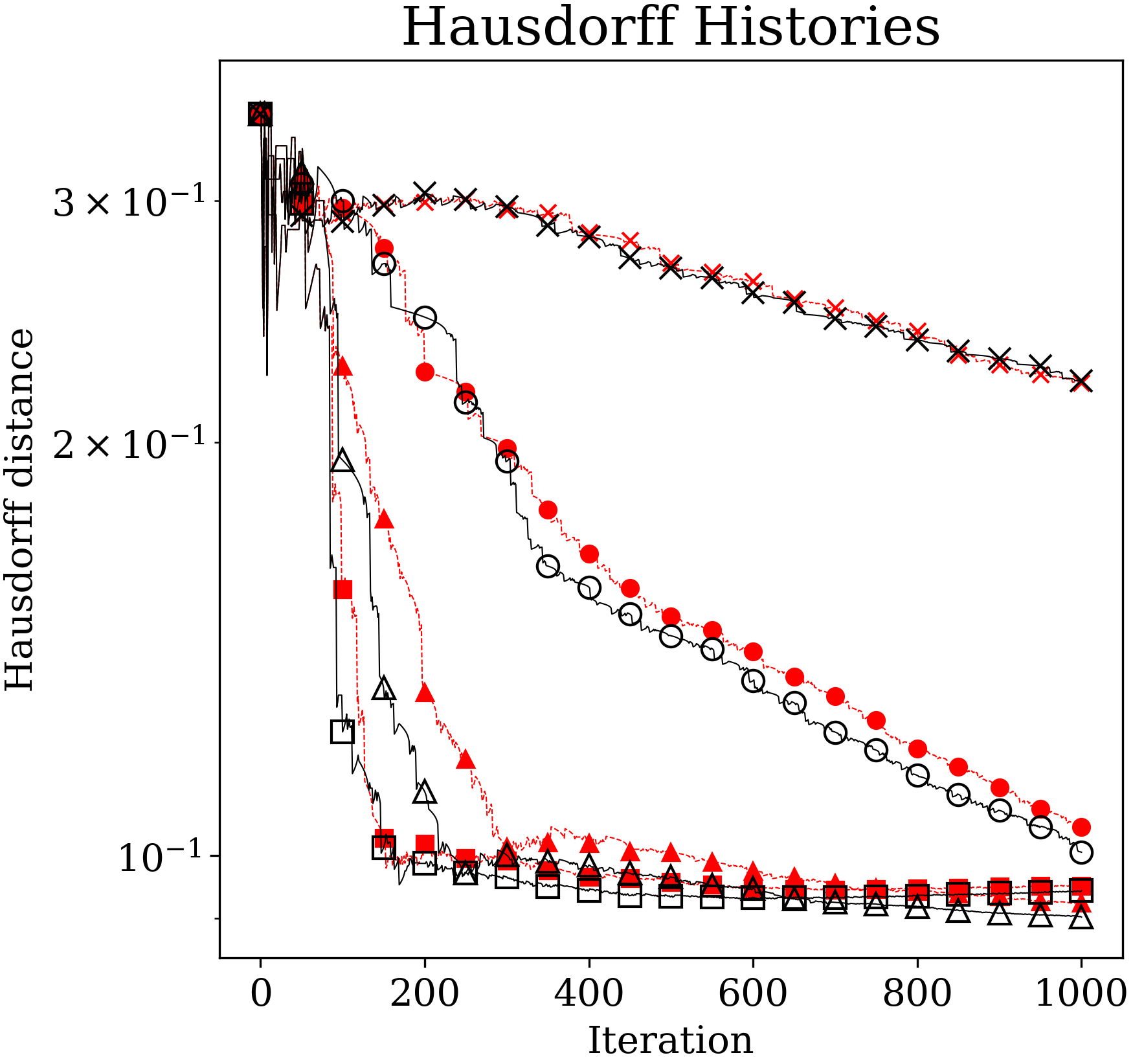}} \\[4pt]
    \resizebox{0.8\textwidth}{!}{\includegraphics{K_joint_vs_split_unit_input_legend_CCBM_ADMM_2x6.png}}    
\caption{
Corresponding histories of values for the cost, gradient norm, and Hausdorff distances for Figure~\ref{ADMMfig:gradient_test_2}.
}
\label{ADMMfig:gradient_test_2_cost_norm_hd}
\end{figure}

We now compare the proposed ADMM--CCBM method with split adjoint to the standard CCBM approach, adopting the prescribed Dirichlet data $f = f_{1} = \cos{\arctan(x_{2}/x_{1})}$ throughout, unless otherwise stated.

Figures~\ref{ADMMfig:test_initial_guess_1} and \ref{ADMMfig:test_initial_guess_2} illustrate the influence of the initial guess and the penalty parameter $\rho$ on the reconstruction obtained with CCBM and ADMM-CCBM (hereafter referred to as ADMM).

For $\rho = 10$, Figure~\ref{ADMMfig:test_initial_guess_1} suggests that ADMM exhibits weaker sensitivity to the choice of initial guess compared to CCBM, indicating improved robustness with respect to initialization. 
In particular, ADMM consistently recovers shapes that are qualitatively closer to the target across the different initializations considered. 
In contrast, while CCBM can yield good reconstructions for favorable initial guesses (e.g., $\varGamma^{0} = C(0,0,0.4)$), its performance deteriorates when the initial geometry is either poorly located or mismatched in size.

Figure~\ref{ADMMfig:test_initial_guess_2} further highlights the role of $\rho$.
Within the tested range $\rho = 2,5,10$, increasing $\rho$ appears to improve the stability of the reconstructions produced by ADMM with respect to the initial configuration. 
For smaller values of $\rho$, however, this effect is less pronounced, and the reconstruction quality may degrade. 
We emphasize that this trend is empirical and restricted to the values of $\rho$ considered; larger values of $\rho$ do not necessarily lead to further improvement; this behavior was also observed in additional numerical experiments not reported here.

Across both figures, CCBM appears comparatively less responsive to variations in $\rho$ and the initialization. 
In several instances, the iterates stagnate away from the target configuration, which is consistent with convergence to suboptimal stationary points. 
This behavior persists despite the use of a backtracking strategy, suggesting that the underlying optimization landscape remains challenging for the present adjoint formulation.

Overall, ADMM appears less sensitive to the choice of initial guess than CCBM, yielding stable reconstructions even for unfavorable initial guesses under exact measurements.

\begin{figure}[h!]
    \centering 
    \resizebox{0.32\textwidth}{!}{\includegraphics{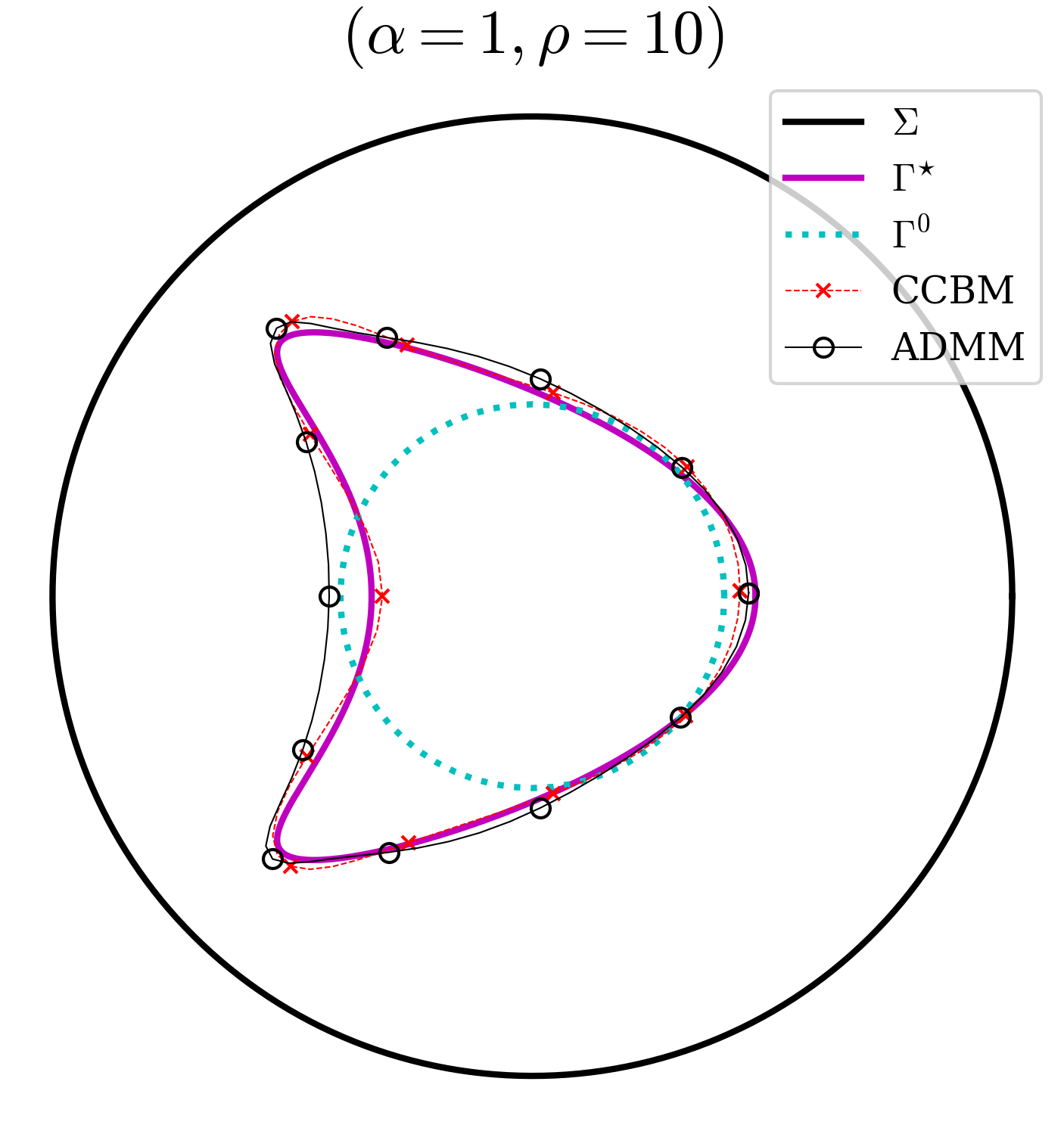}} \hfill
    \resizebox{0.32\textwidth}{!}{\includegraphics{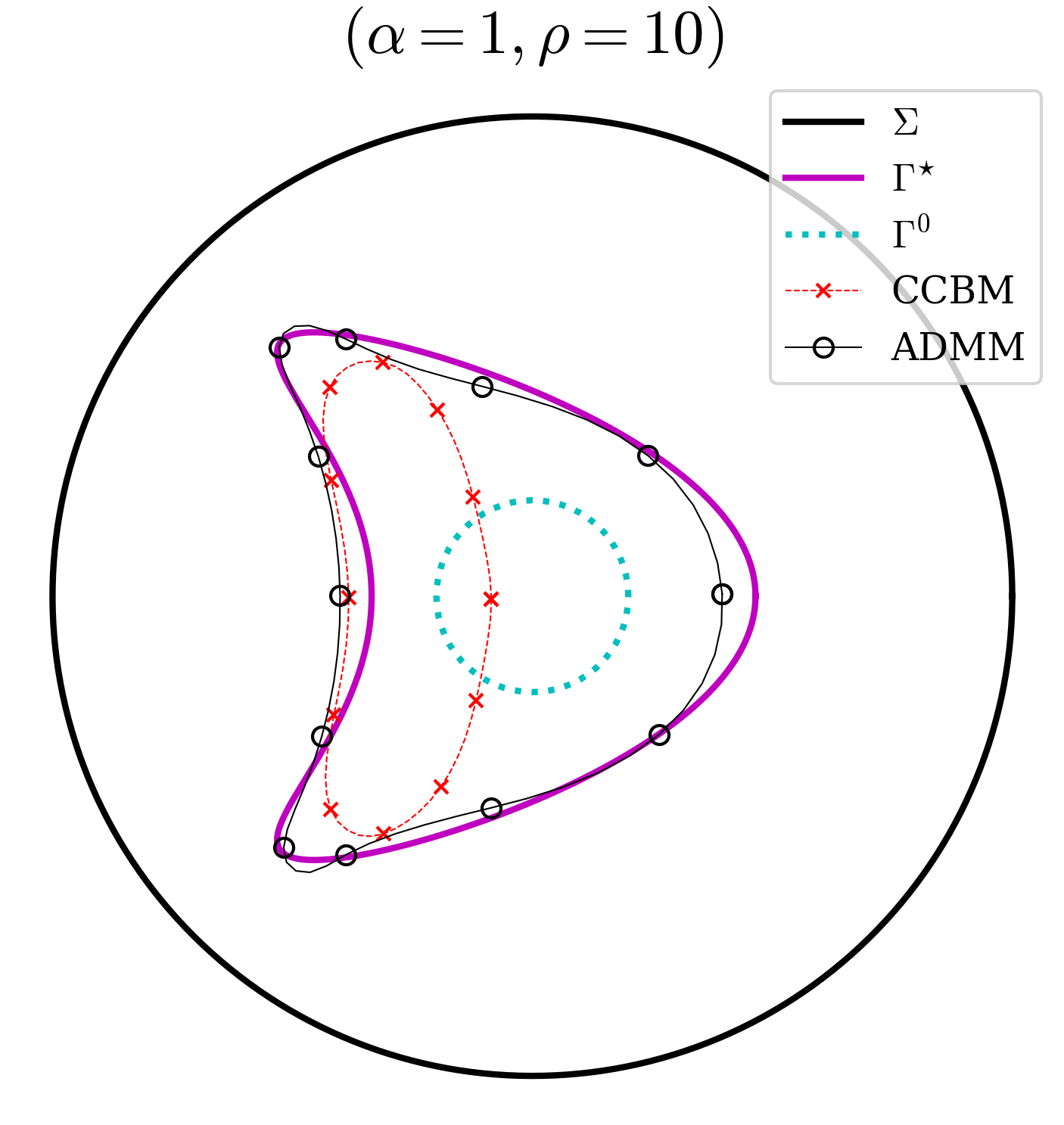}} \hfill
    \resizebox{0.32\textwidth}{!}{\includegraphics{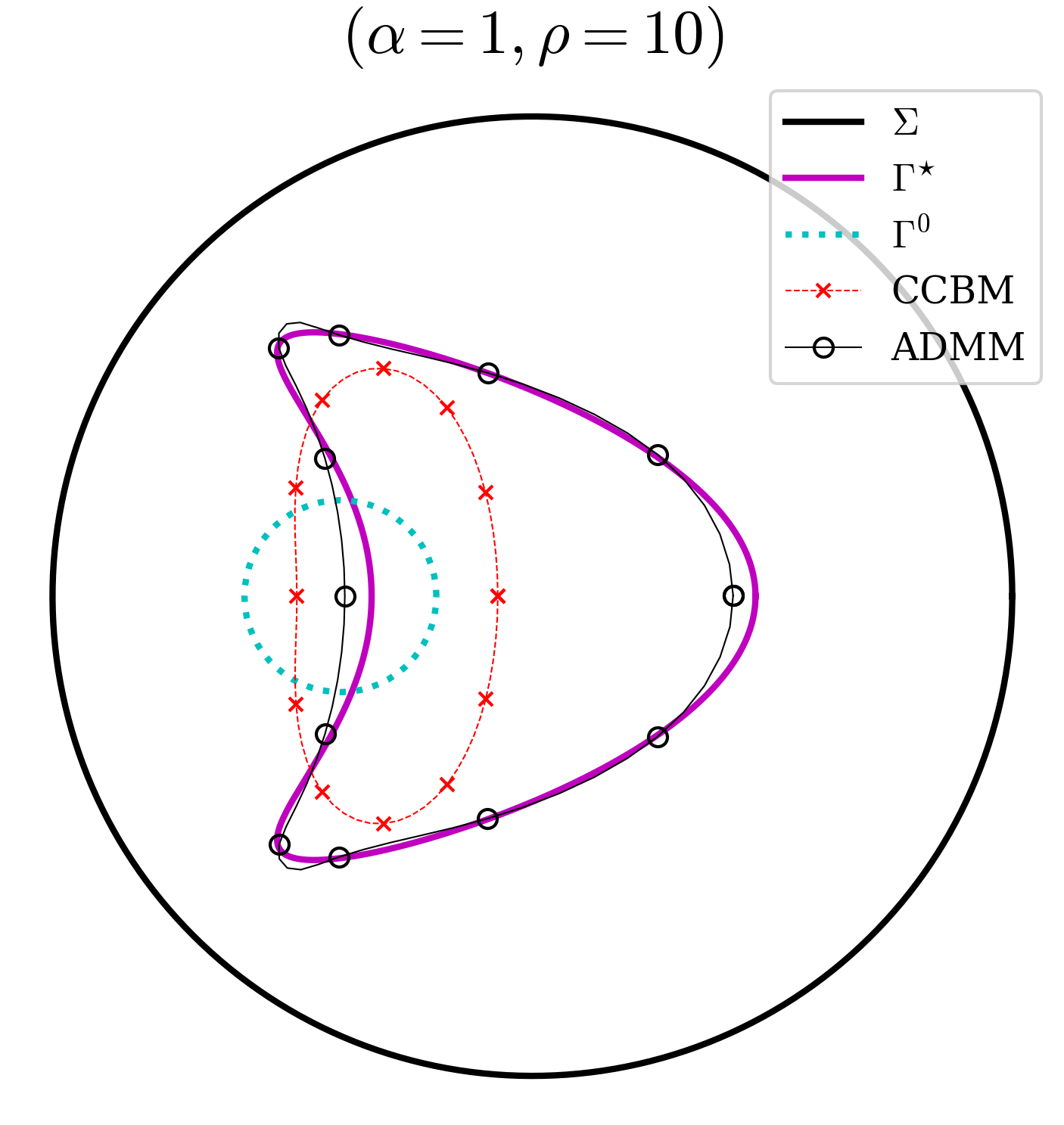}} 
\caption{
Effect of size and location of the initial guess with $\rho = 10$ under exact measurements.
}
\label{ADMMfig:test_initial_guess_1}
\end{figure}

\begin{figure}[h!]
    \centering 
    \resizebox{0.32\textwidth}{!}{\includegraphics{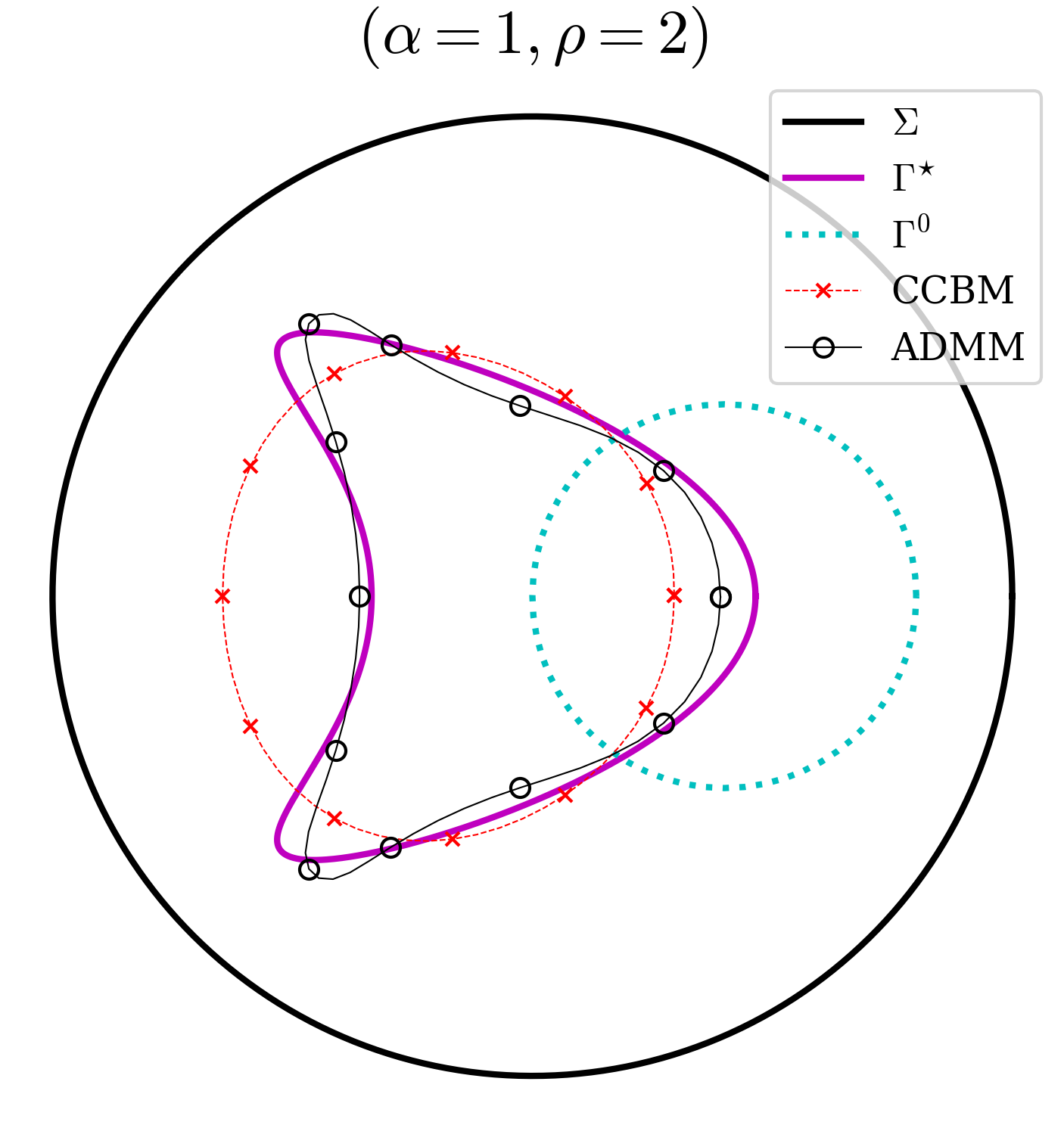}} \hfill
    \resizebox{0.32\textwidth}{!}{\includegraphics{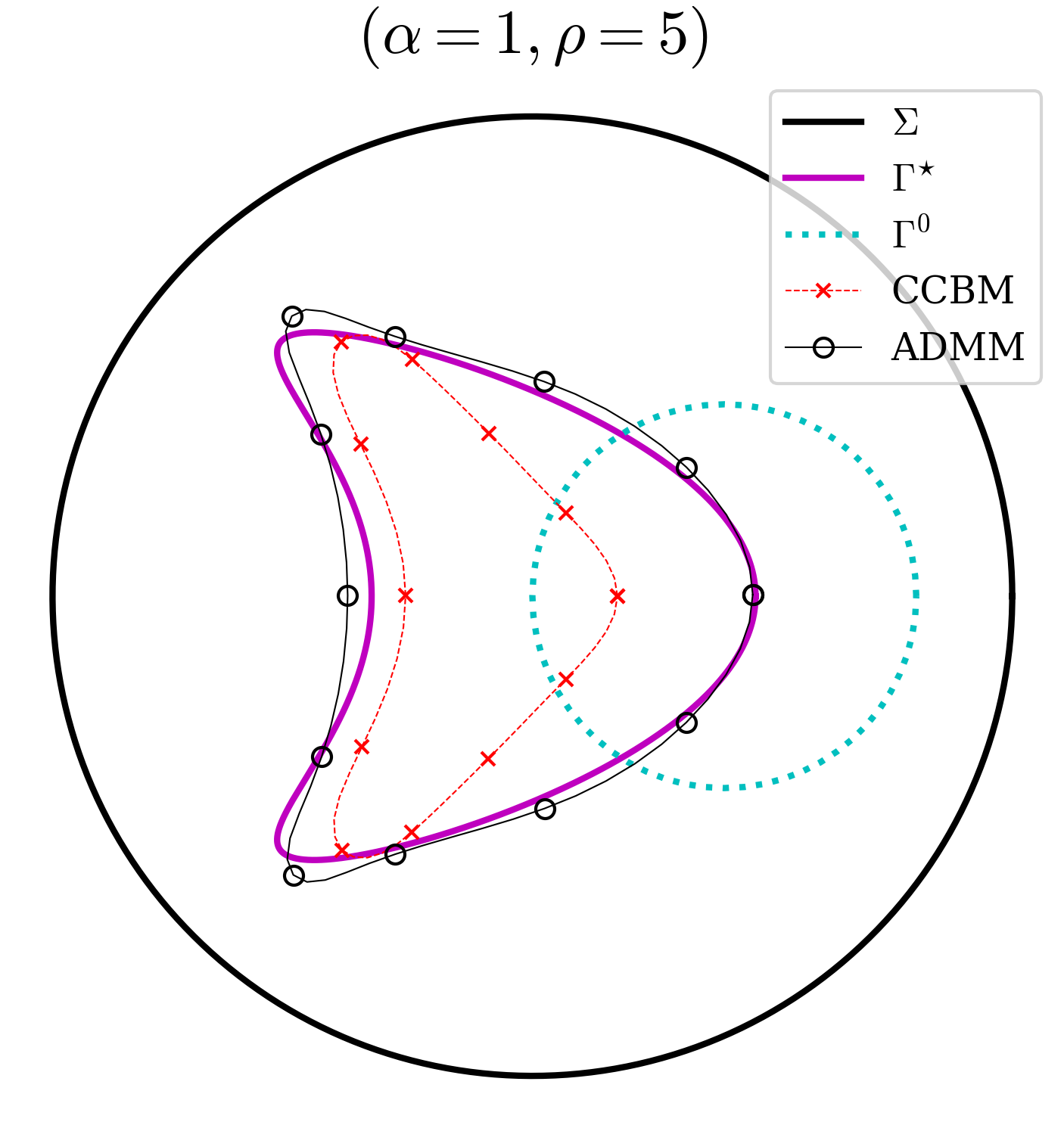}} \hfill
    \resizebox{0.32\textwidth}{!}{\includegraphics{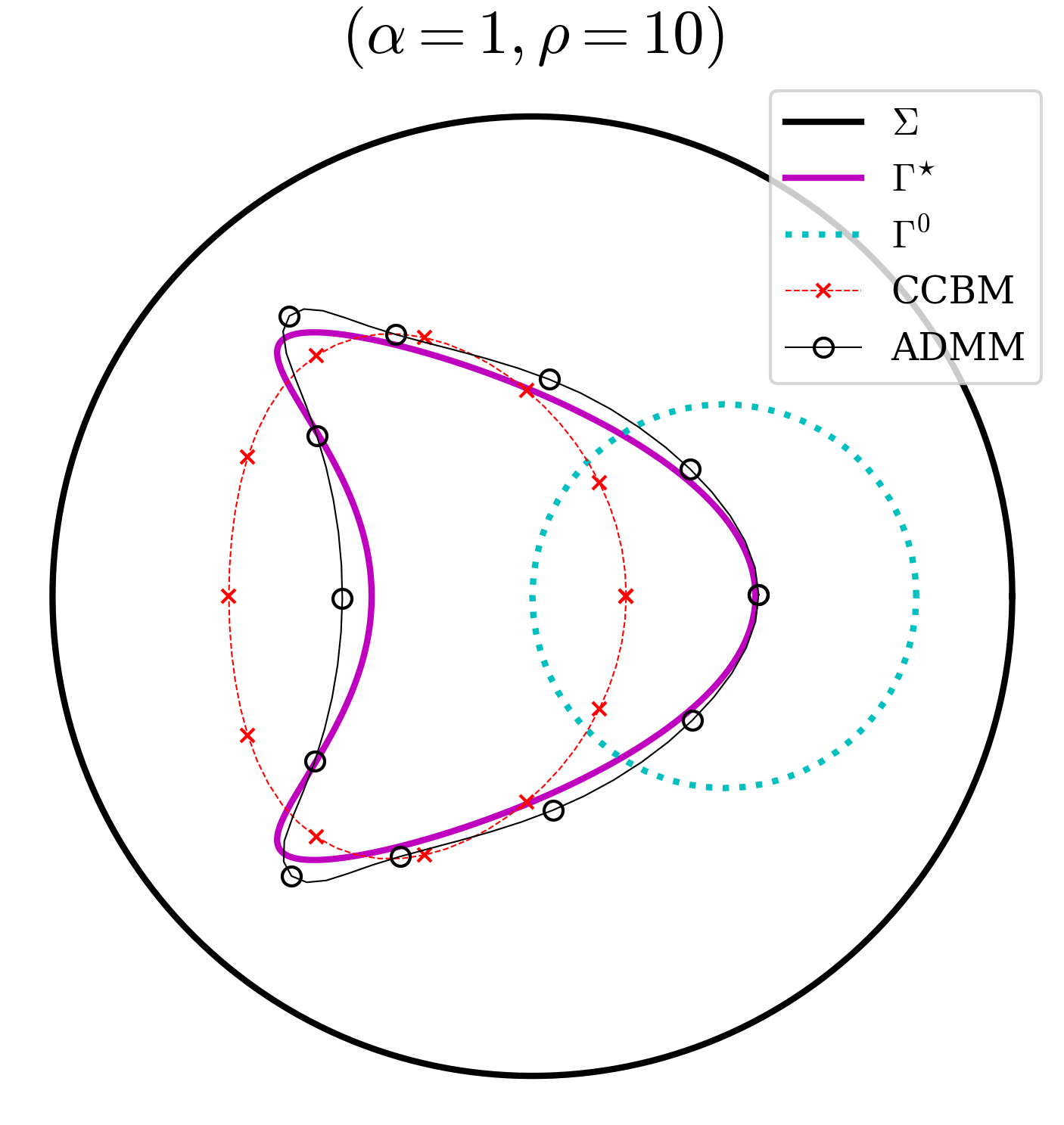}}
\caption{
Reconstructions with varying $\rho = 2,5,10$ under exact measurements.
}
\label{ADMMfig:test_initial_guess_2}
\end{figure}


We now turn to the more realistic setting of noisy data to assess the robustness of the proposed {{ADMM}} framework. 
The measurements are perturbed with relative noise levels $\delta = 0.005$ and $\delta = 0.01$.

Figure~\ref{ADMMfig:noisy_test_kite} reports reconstructions corresponding to three choices of Dirichlet input $f$.
The numerical results indicate a pronounced dependence of the reconstruction quality on the selected excitation, particularly in the recovery of concave portions of the cavity boundary, a behavior already observed in the multiple-measurement framework studied in \cite{AfraitesRabago2025}.
For $f = f_{1}$, the method captures the overall geometry of the cavity, although concave regions remain smoothed and the reconstructed boundary exhibits a slight displacement that becomes more noticeable as $\delta$ increases.
For the oscillatory input
$f = f_{2} \coloneqq \sin(0.1\pi x_{1}) \sin(0.1\pi x_{2})$,
the reconstruction improves locally, suggesting enhanced sensitivity to higher-frequency geometric features, despite the presence of minor distortions.
The combined excitation $f = f_{1} + f_{2}$ yields the most accurate reconstructions for both noise levels considered.
These observations are consistent with the behavior obtained in the noise-free regime and in the {{ADMM}} comparison tests, where richer boundary excitations lead to improved identifiability properties.
At the same time, the progressive deterioration as $\delta$ increases further reflects the mildly ill-posed character of the reconstruction problem.

\begin{figure}[h!]
    \centering 
    \resizebox{0.32\textwidth}{!}{\includegraphics{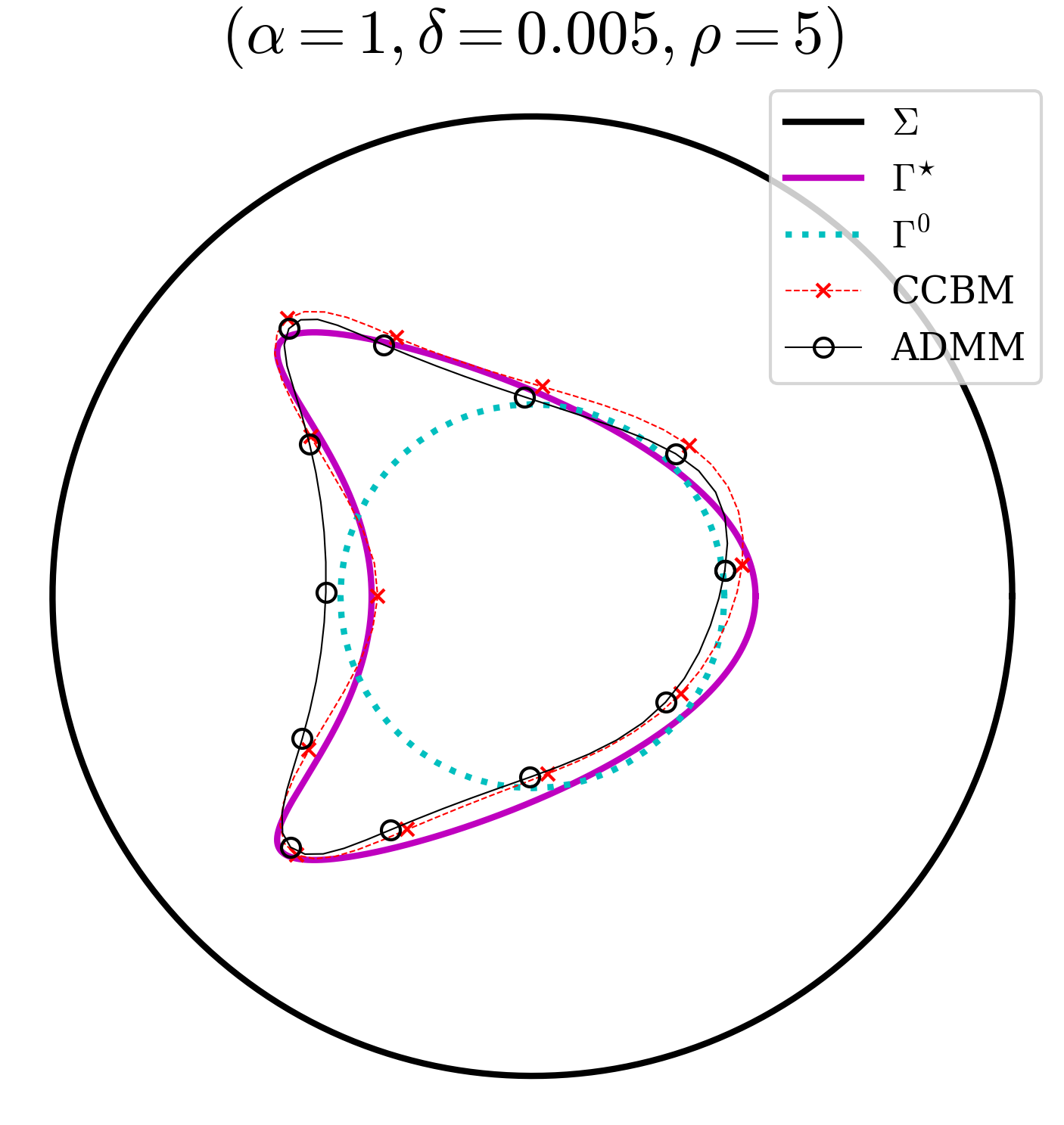}} \hfill
    \resizebox{0.32\textwidth}{!}{\includegraphics{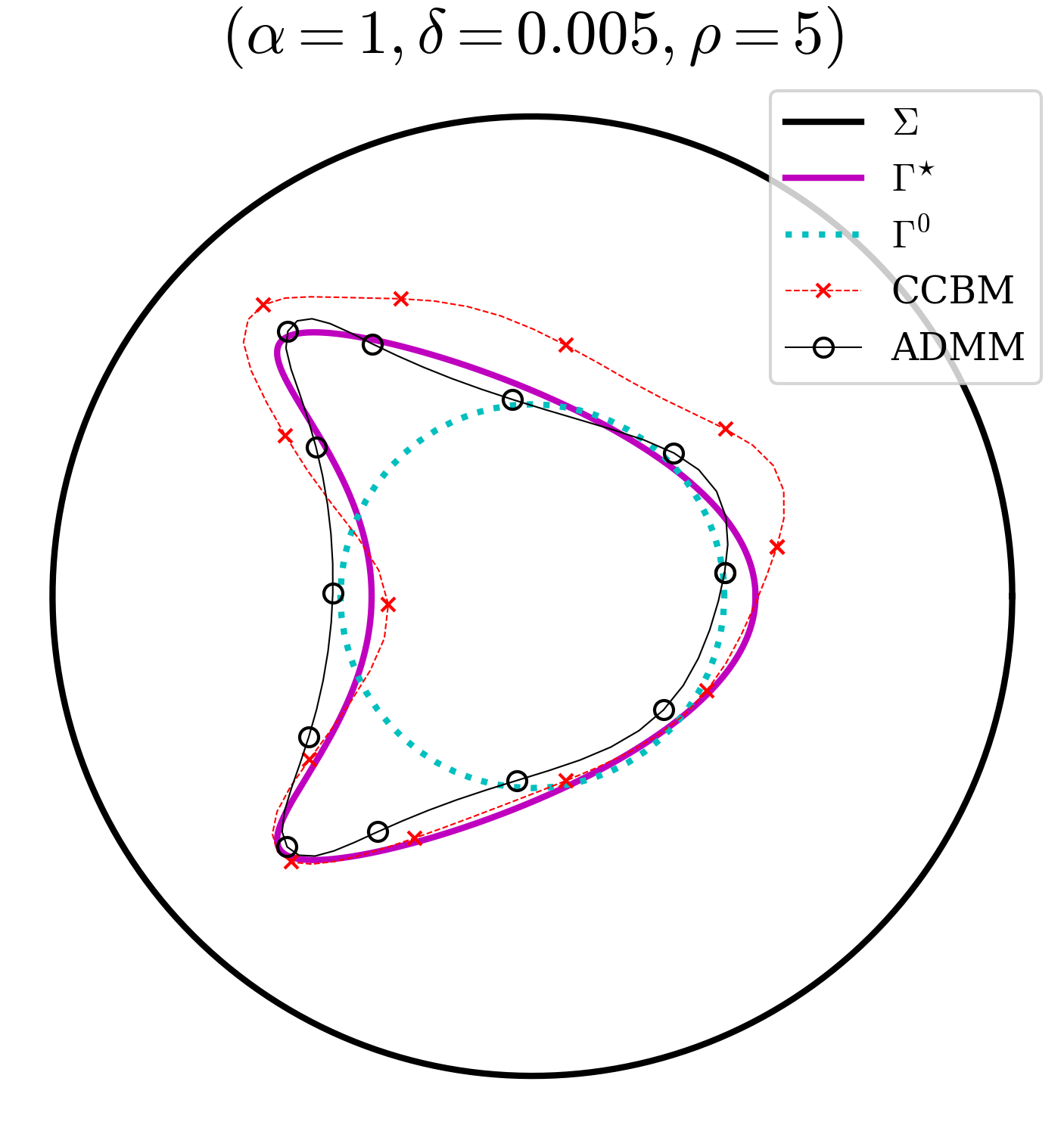}} \hfill
    \resizebox{0.32\textwidth}{!}{\includegraphics{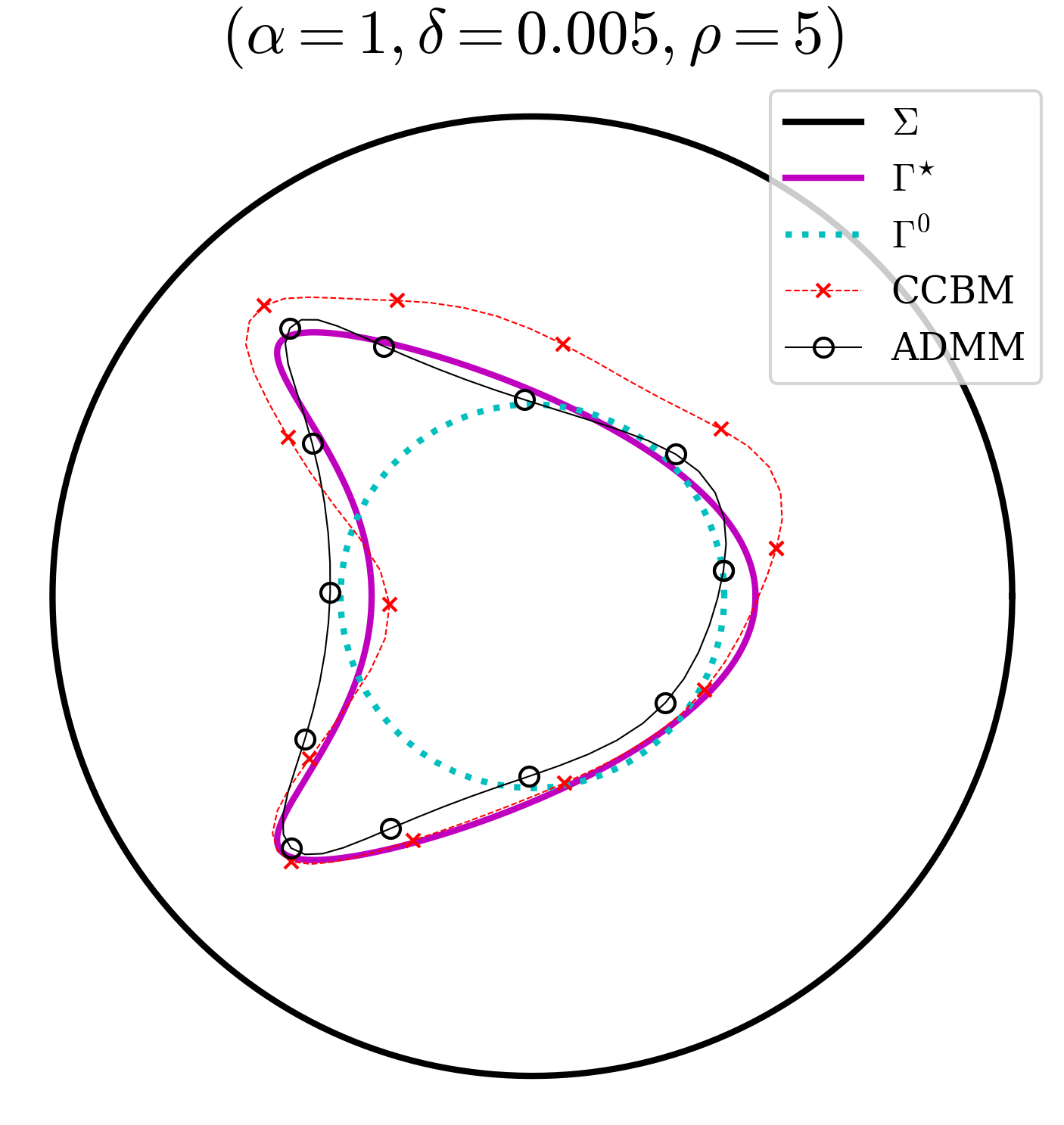}}\\[4pt]
    \resizebox{0.32\textwidth}{!}{\includegraphics{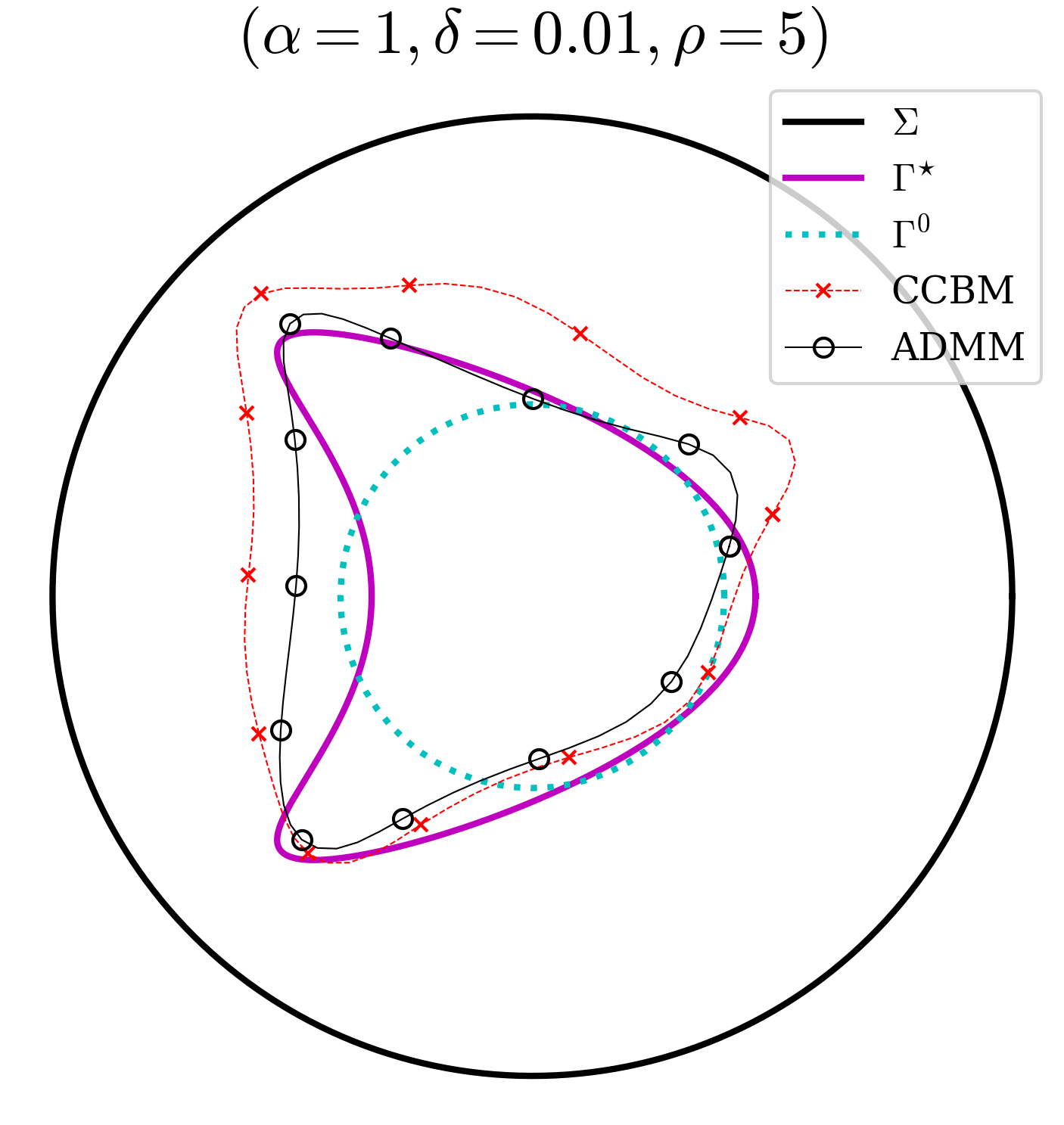}} \hfill
    \resizebox{0.32\textwidth}{!}{\includegraphics{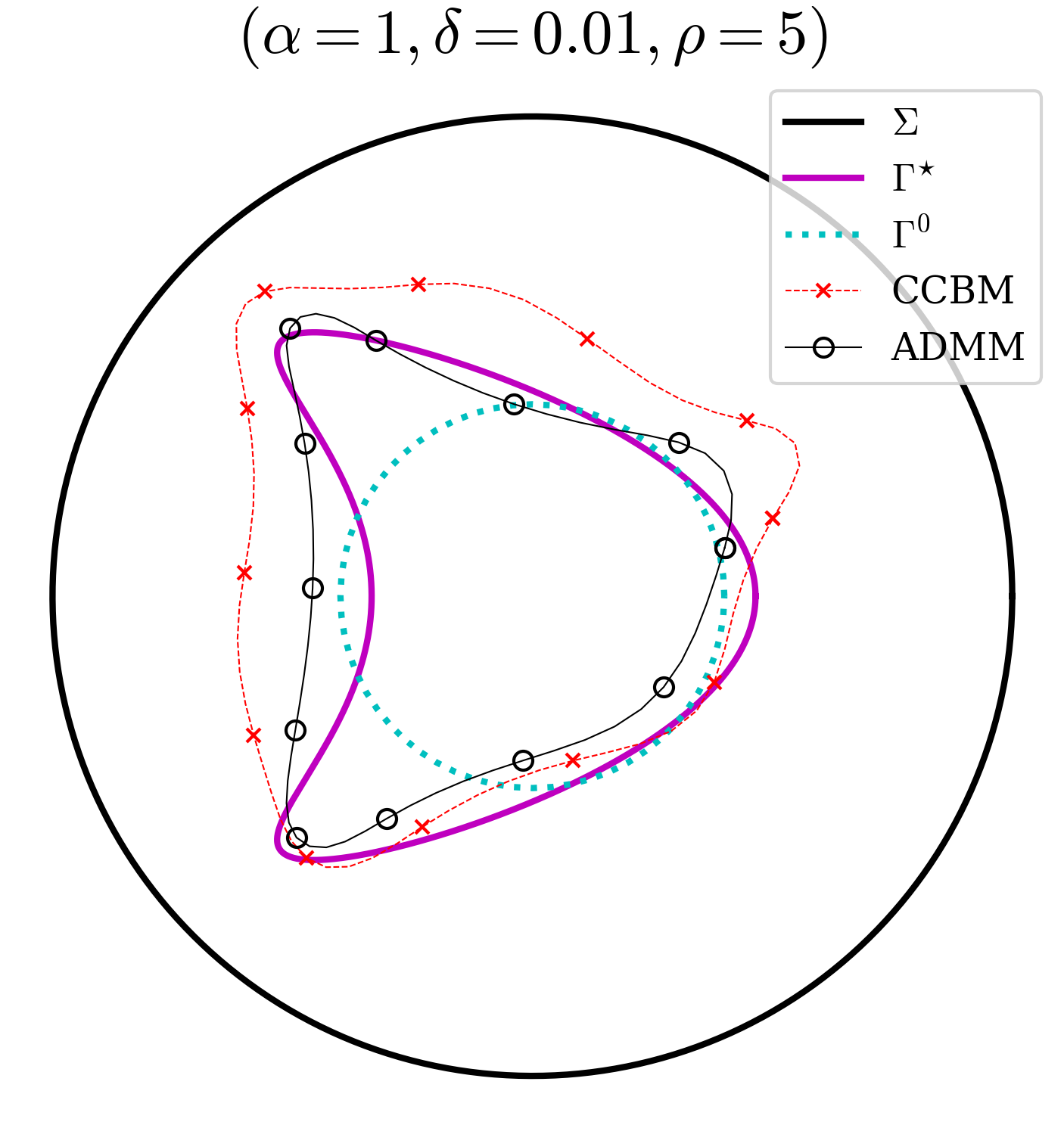}} \hfill
    \resizebox{0.32\textwidth}{!}{\includegraphics{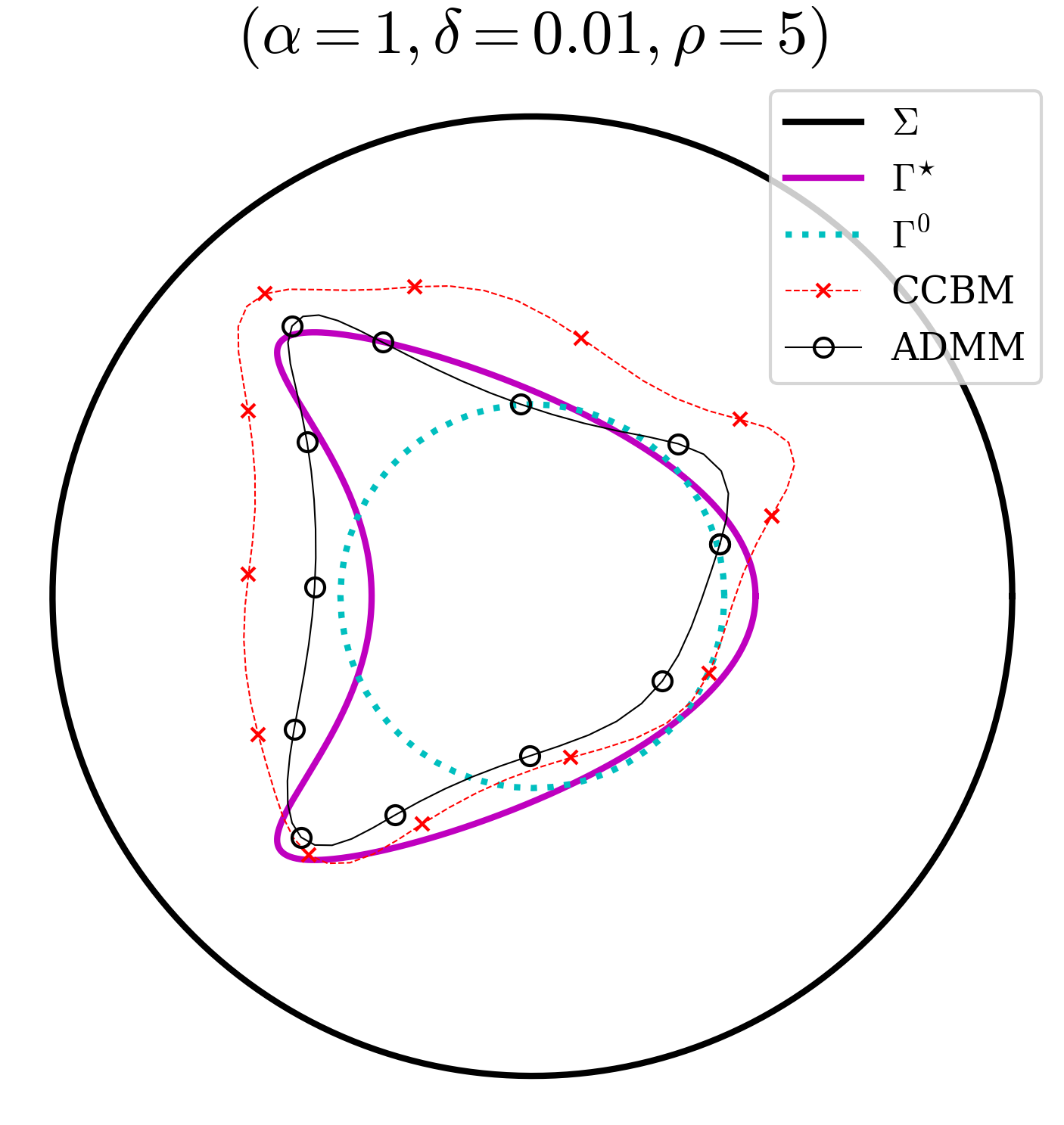}}
\caption{
Reconstructions under noisy measurements with low noise levels $\delta = 0.005$ (top row) and $\delta = 0.01$ (bottom row). The prescribed Dirichlet data are $f = \cos{\arctan(x_{2}/x_{1})}$ (left column), $f = \sin{(0.1\pi x_{1})} \sin{(0.1\pi x_{2})}$ (middle column), and $f = \cos{\arctan(x_{2}/x_{1})} + \sin{(0.1\pi x_{1})} \sin{(0.1\pi x_{2})}$ (right column).
}
\label{ADMMfig:noisy_test_kite}
\end{figure}

In Figure~\ref{ADMMfig:noisy_test_L}, we examine the effect of $\rho$ for the \texttt{L}-shaped cavity. 
For $\rho = 5$, the reconstructions are overly smooth and exhibit a visible bias near the reentrant corner. 
Increasing to $\rho = 10$ improves geometric fidelity, particularly along flat edges and corners. 
For $\rho = 20$, the improvement is marginal, and slight boundary oscillations appear in some cases. 
This behavior reflects the role of $\rho$ in the ADMM splitting: small values lead to weak enforcement of the auxiliary constraint, while large values increase sensitivity to noise. 
An intermediate choice of $\rho$ provides a reasonable compromise, consistent with the trends observed in the {{ADMM}} comparison.

\begin{figure}[h!]
    \centering 
    \resizebox{0.32\textwidth}{!}{\includegraphics{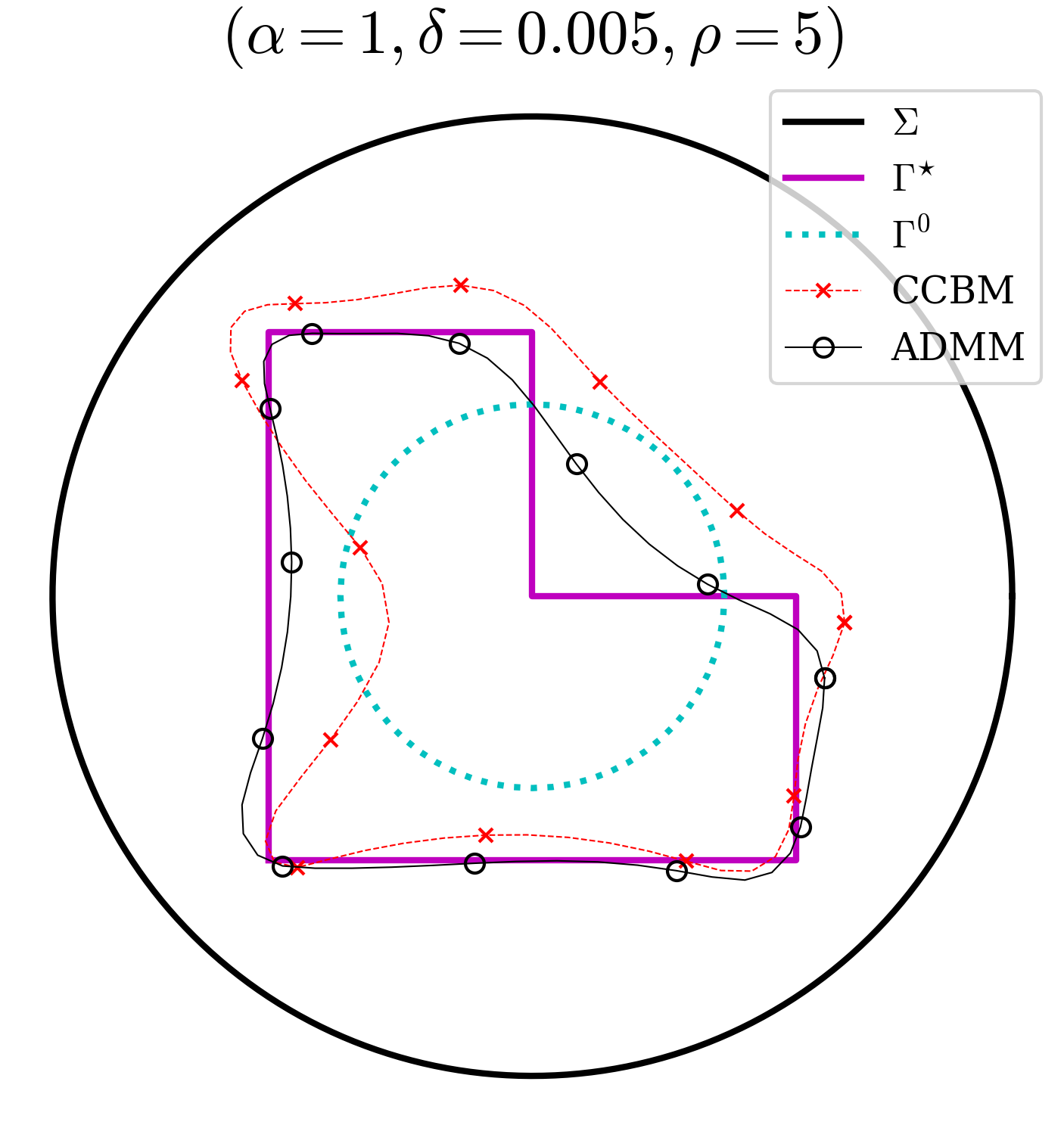}} \hfill
    \resizebox{0.32\textwidth}{!}{\includegraphics{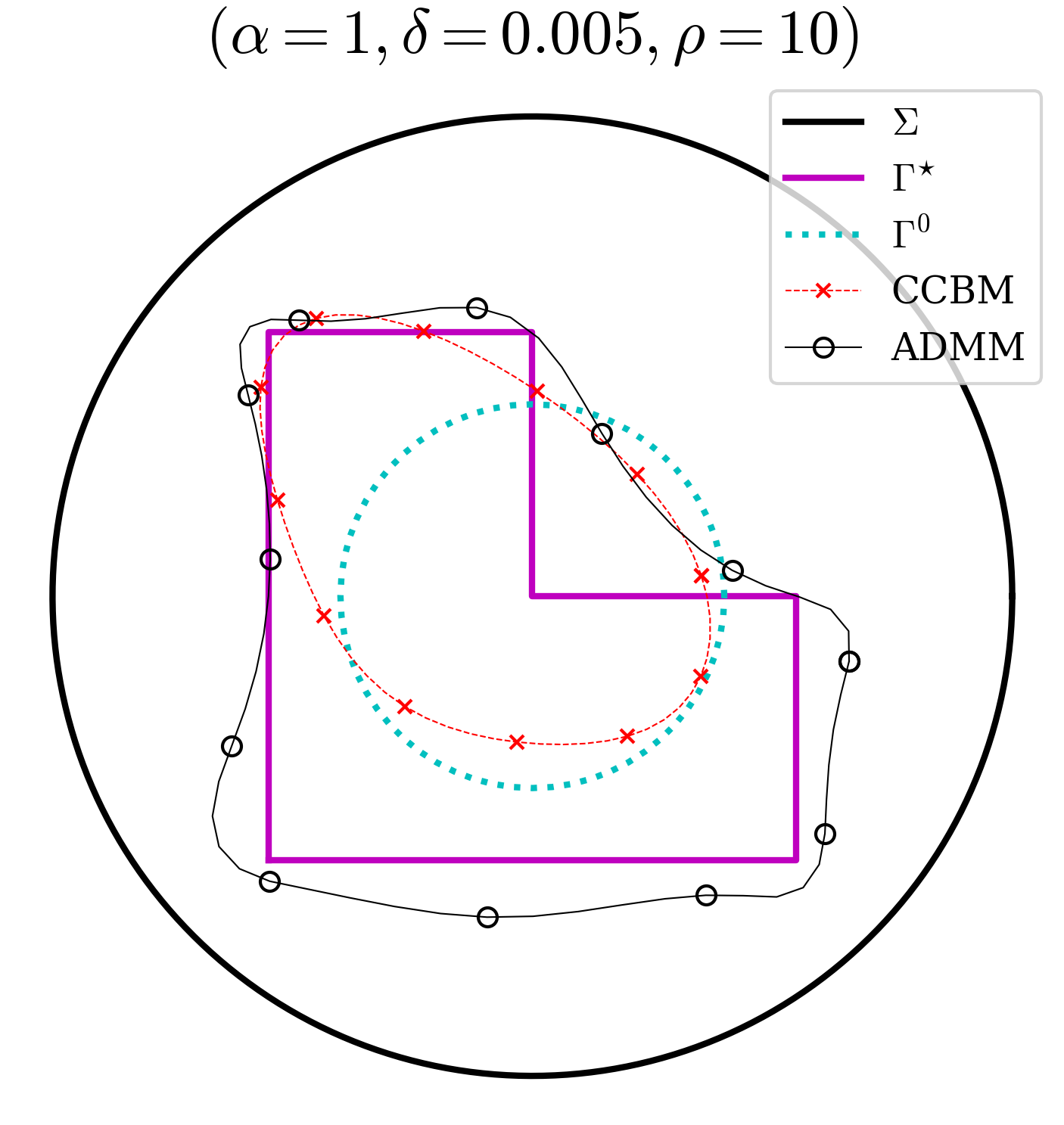}} \hfill
    \resizebox{0.32\textwidth}{!}{\includegraphics{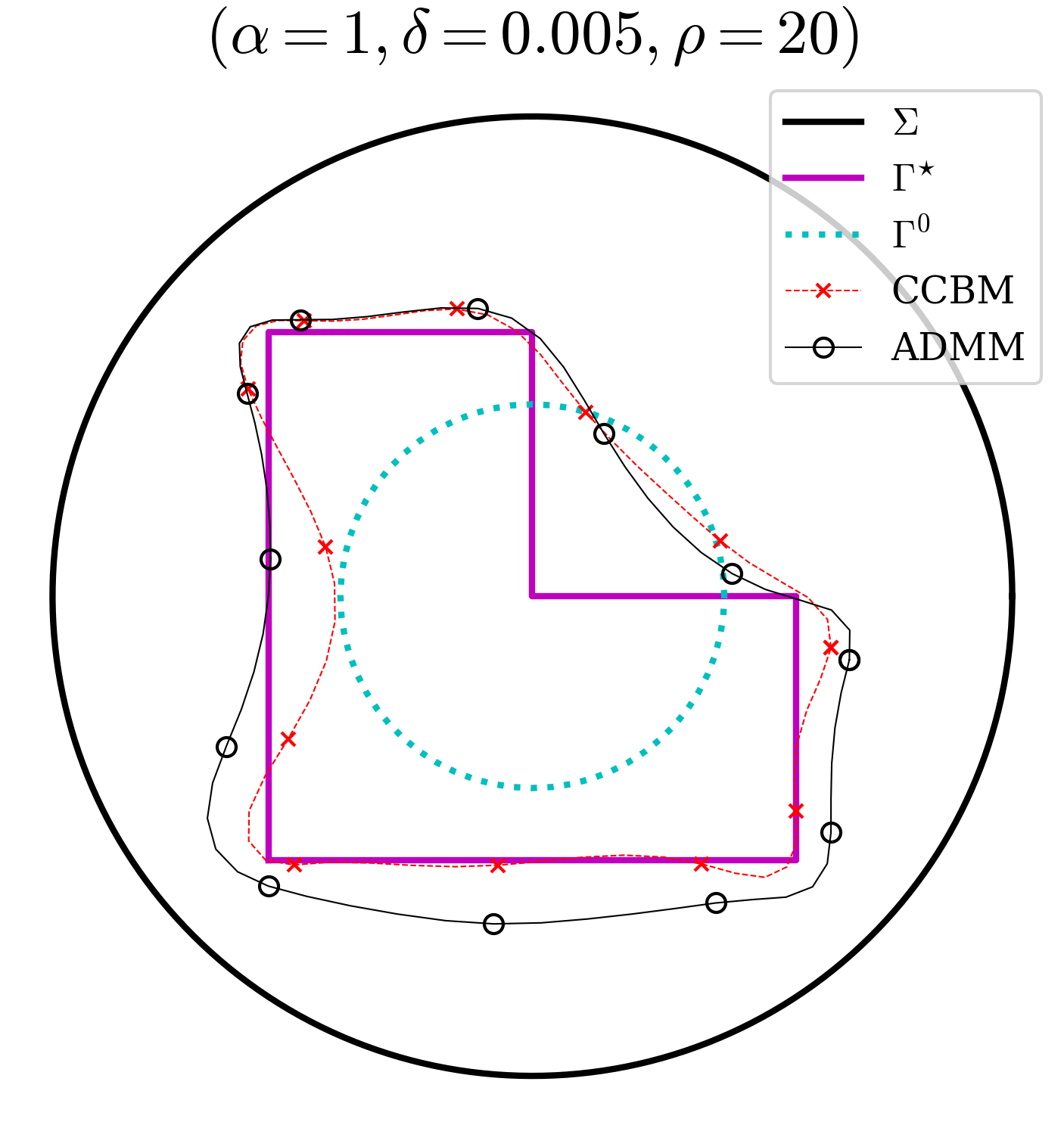}}\\[4pt]
    \resizebox{0.32\textwidth}{!}{\includegraphics{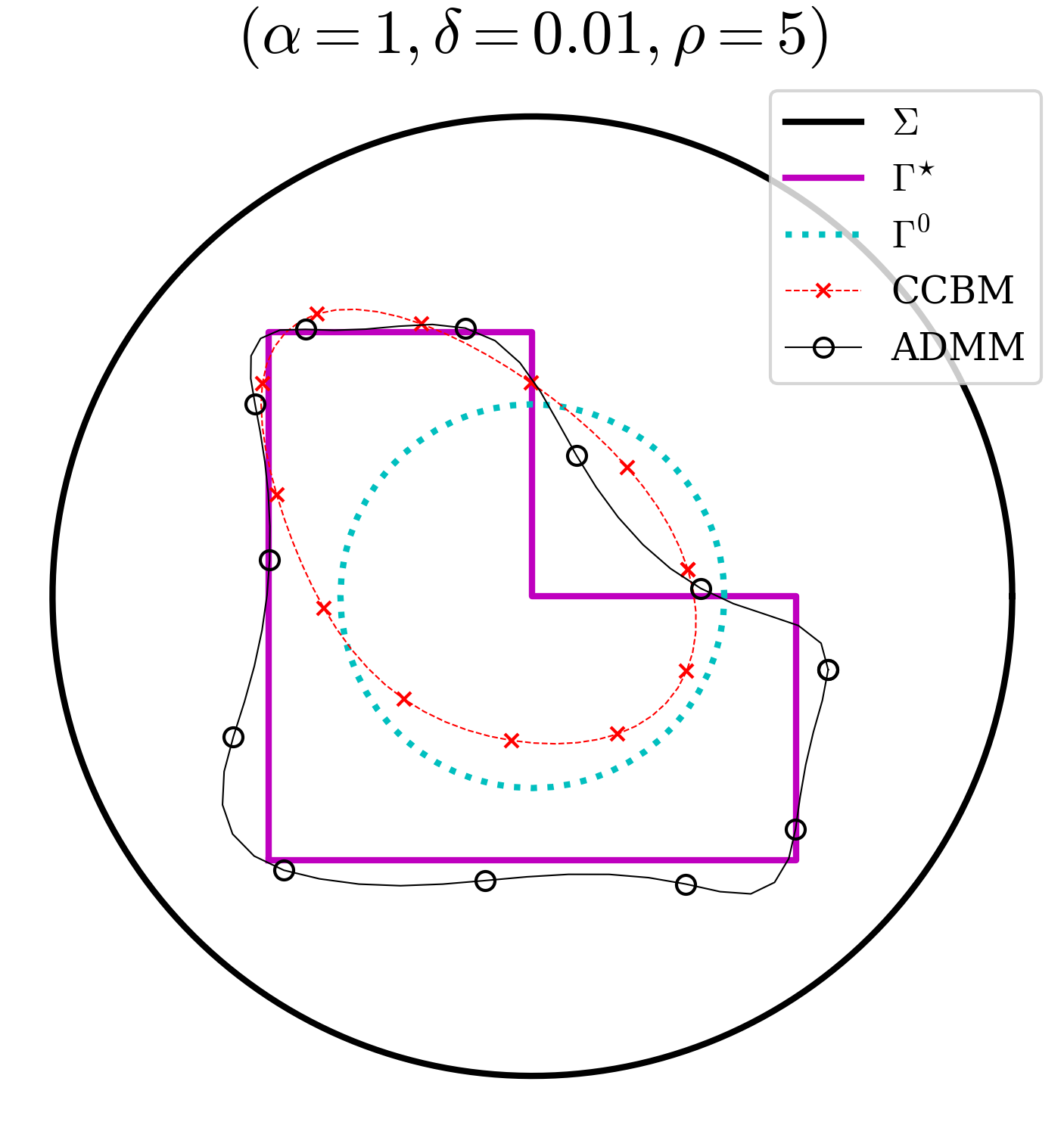}} \hfill
    \resizebox{0.32\textwidth}{!}{\includegraphics{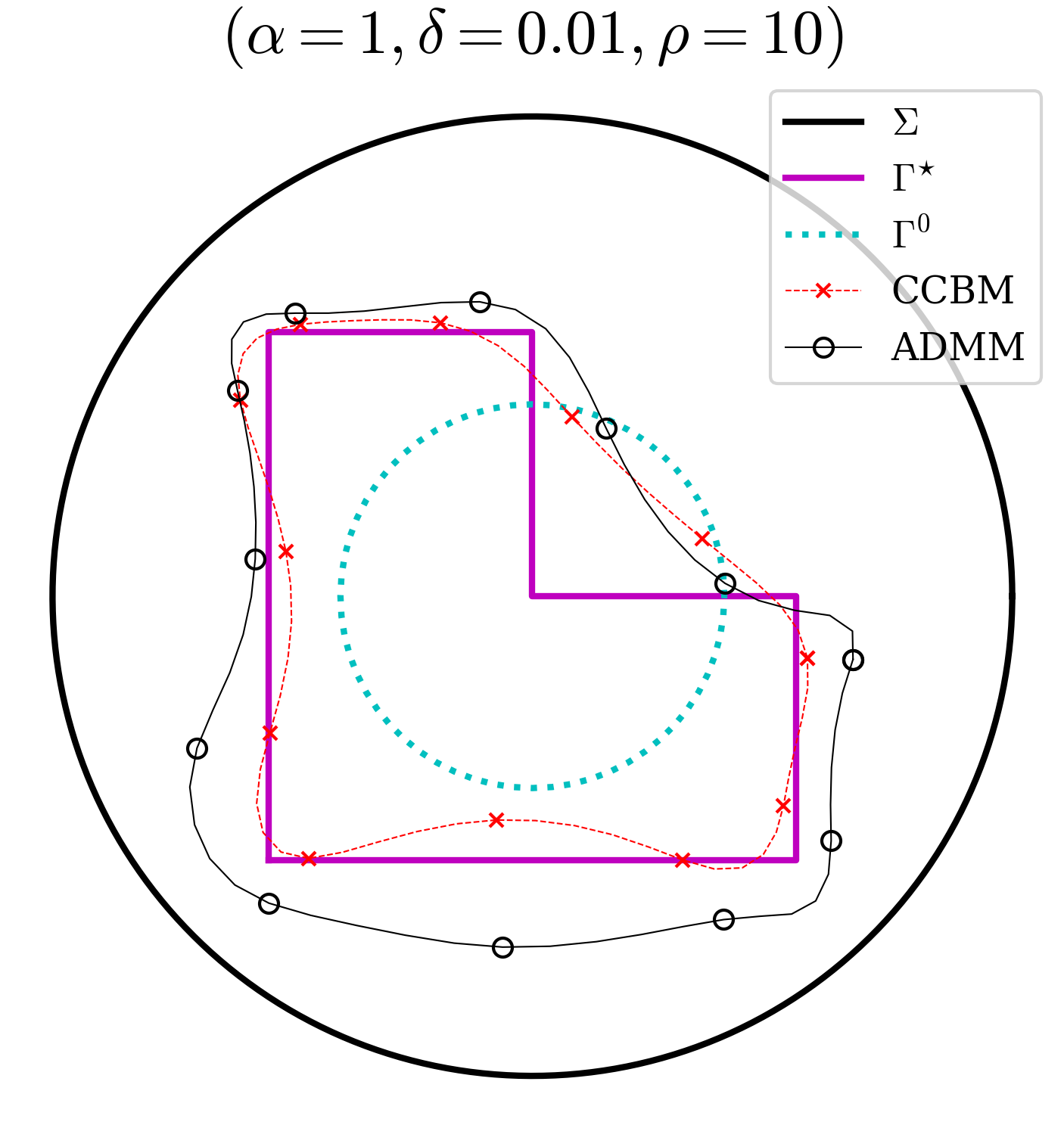}} \hfill
    \resizebox{0.32\textwidth}{!}{\includegraphics{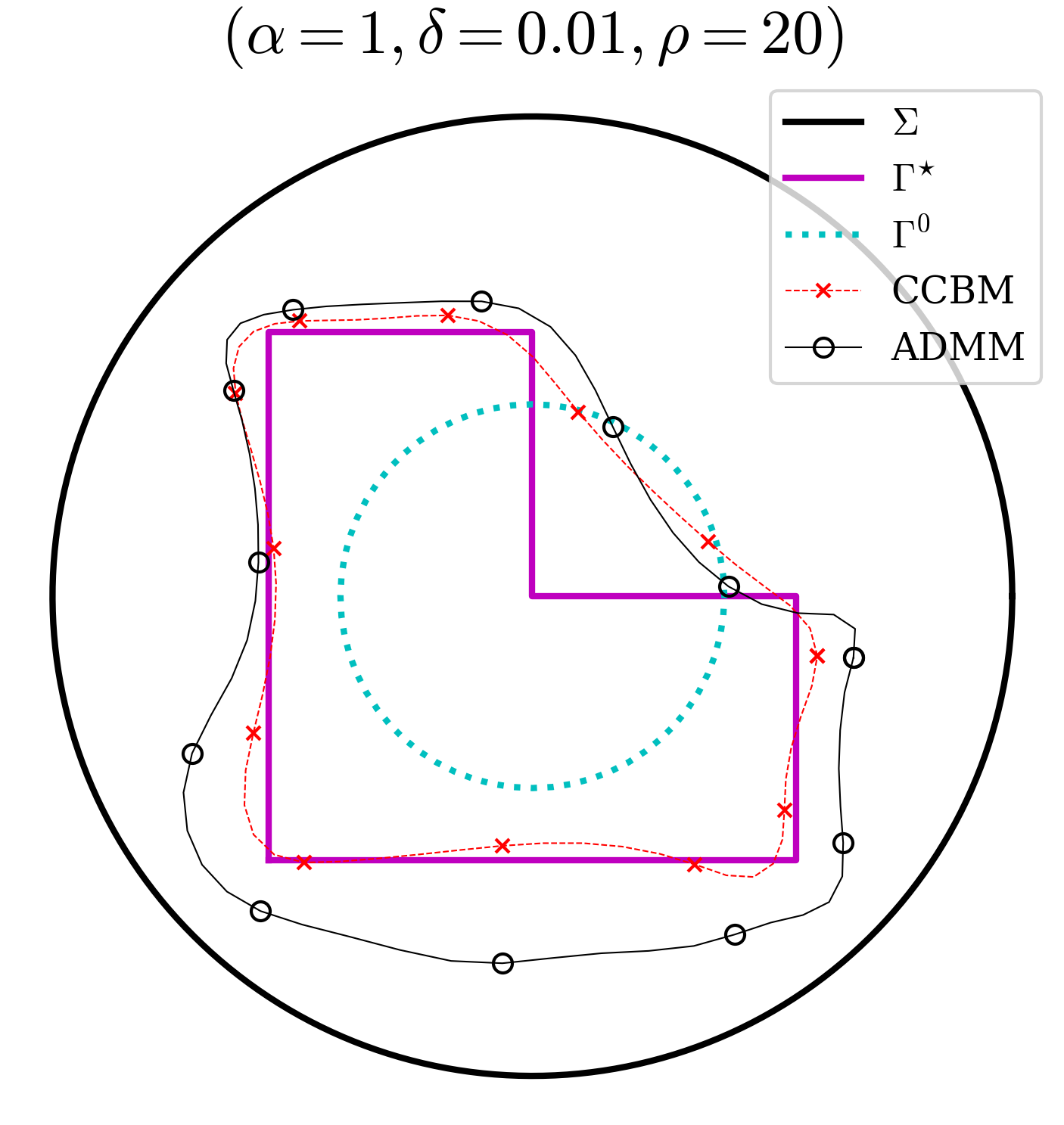}}
\caption{
Reconstructions under noisy measurements with low noise levels $\delta = 0.005$ (top row) and $\delta = 0.01$ (bottom row) for an \texttt{L}-shape cavity under the prescribed Dirichlet data $f = \cos{\arctan(x_{2}/x_{1})} + \sin{(0.1\pi x_{1})} \sin{(0.1\pi x_{2})}$.
}
\label{ADMMfig:noisy_test_L}
\end{figure}

Figure~\ref{ADMMfig:noisy_test_L_fminmax_ab} illustrates the effect of imposing the inequality constraint through the admissible set \eqref{eq:admissible_set_K} with 
$a = \min_{\xi \in \varOmega^{\star}} f(\xi)$ and $b = \max_{\xi \in \varOmega^{\star}} f(\xi)$.
The constraint stabilizes the reconstruction across different initializations and suppresses spurious oscillations induced by noise. 
In particular, the recovered shapes remain geometrically consistent, even in the presence of perturbations. 
However, the constraint introduces a mild bias, with reduced resolution near the reentrant corner. 
This reflects a trade-off between stability and accuracy, which is typical for constrained inverse problems. 
Compared with the unconstrained CCBM formulation, the constrained ADMM approach yields more stable reconstructions under noise, at the cost of slightly reduced sharpness.

\begin{figure}[h!]
    \centering 
    \resizebox{0.24\textwidth}{!}{\includegraphics{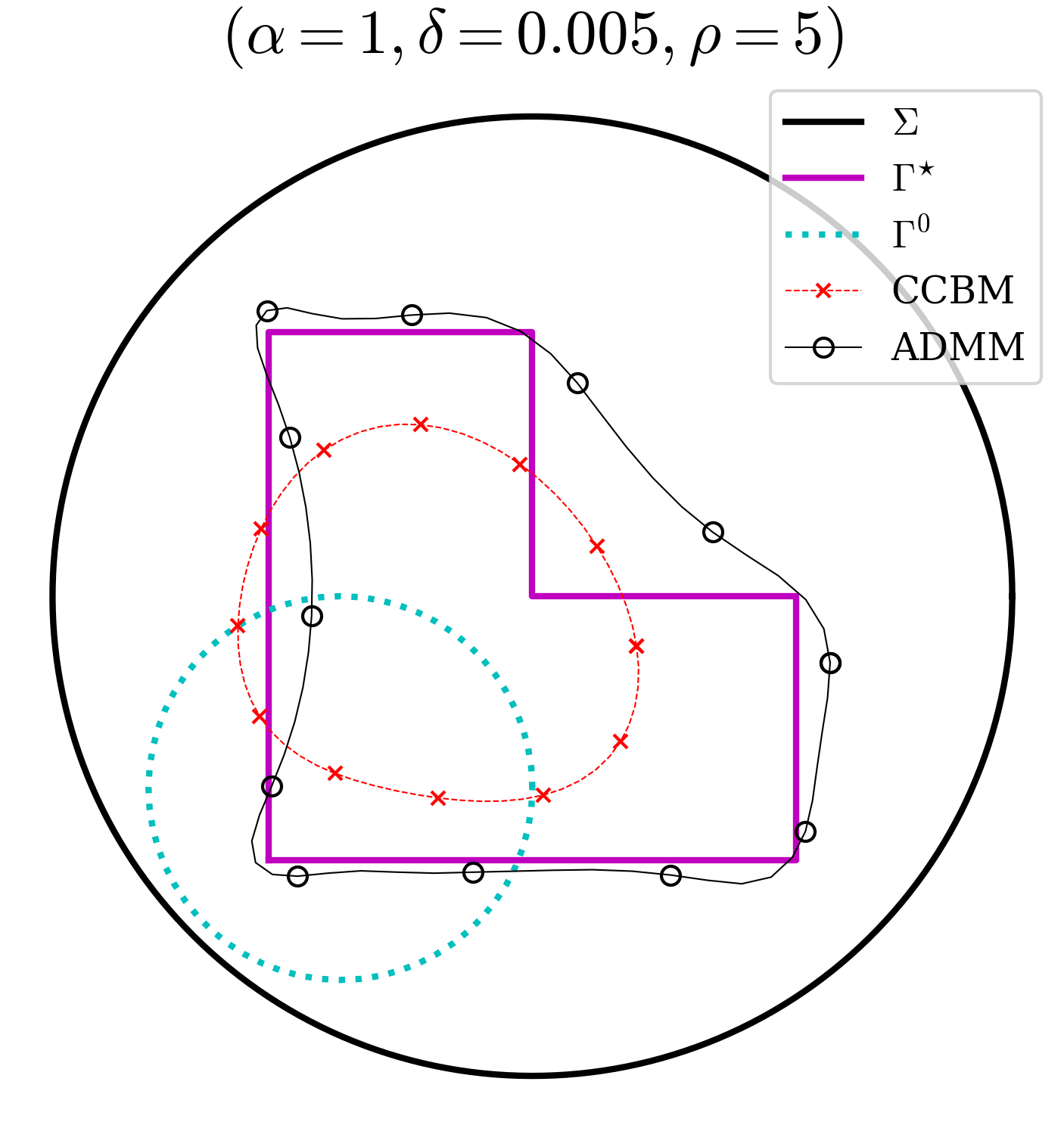}} 
    \resizebox{0.24\textwidth}{!}{\includegraphics{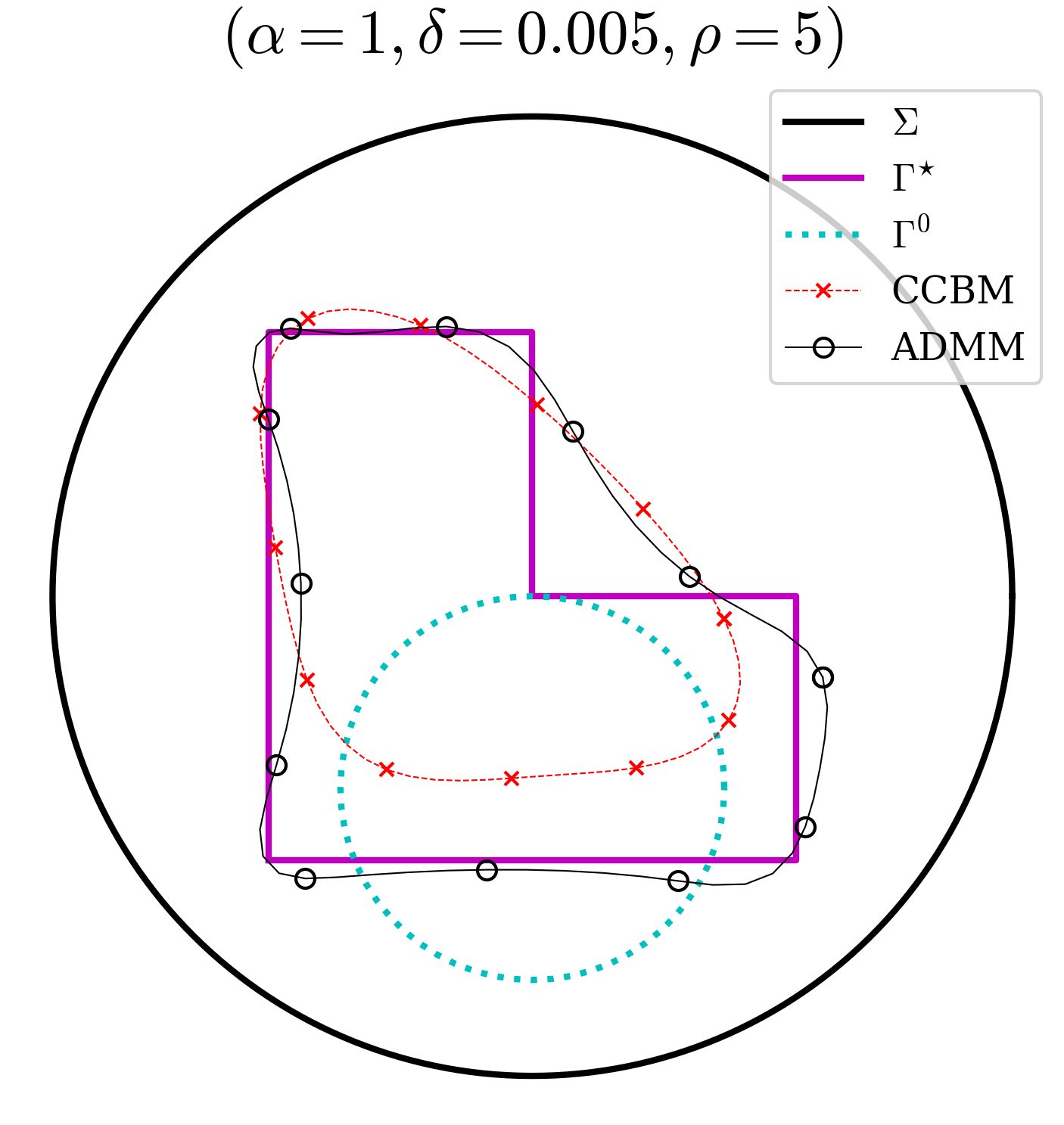}} 
    \resizebox{0.24\textwidth}{!}{\includegraphics{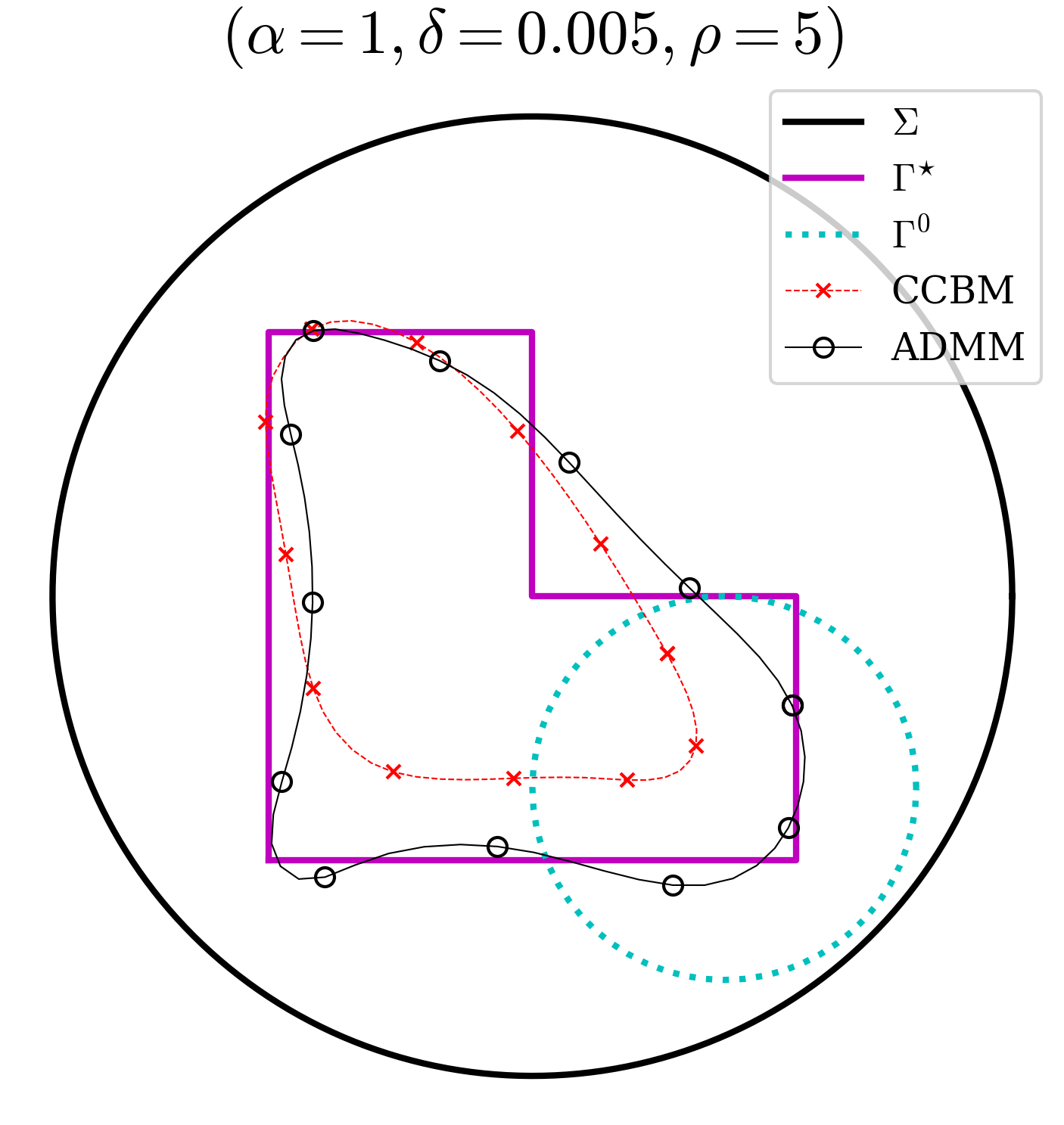}}
    \resizebox{0.24\textwidth}{!}{\includegraphics{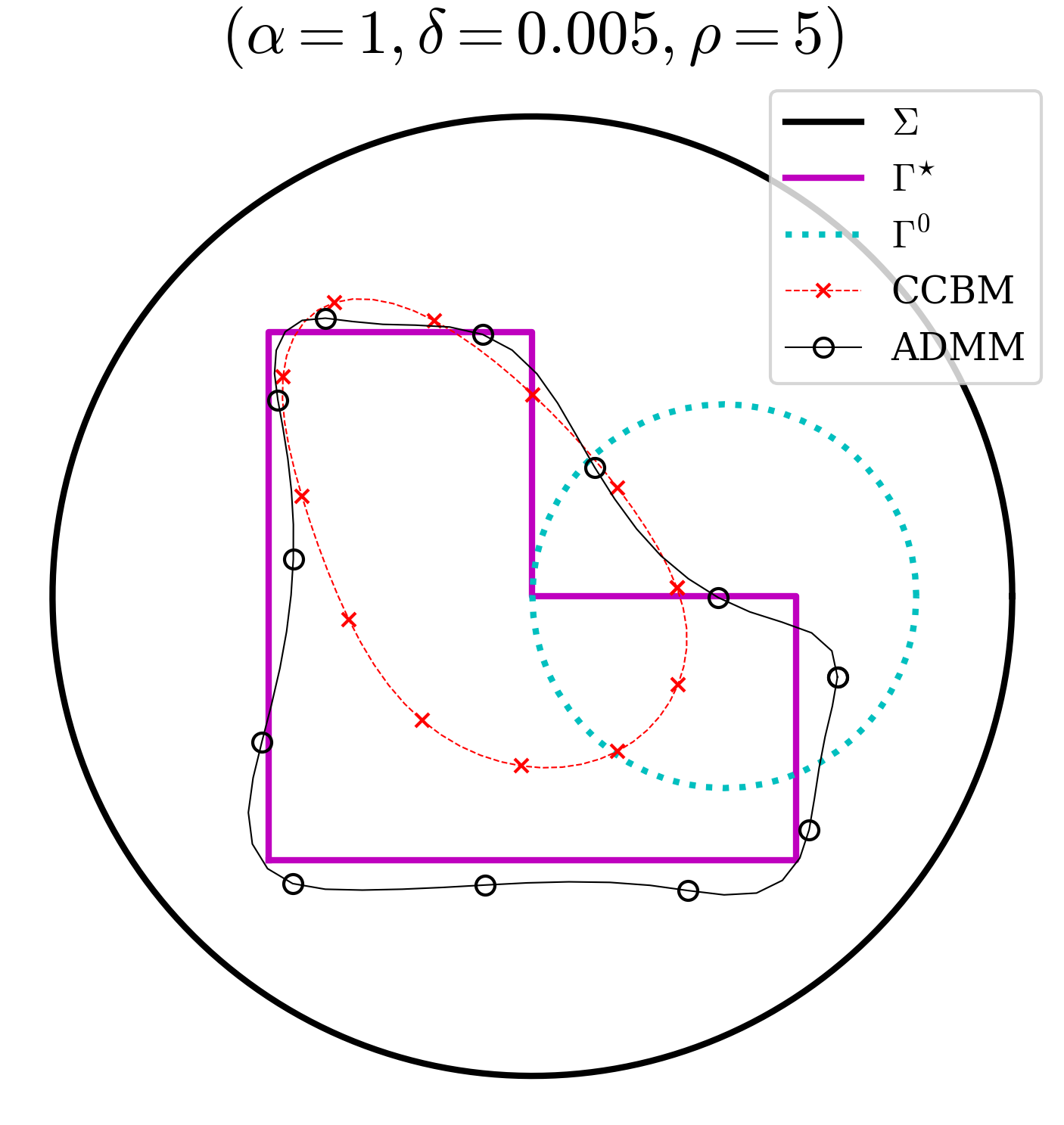}} \\[4pt]
    \resizebox{0.24\textwidth}{!}{\includegraphics{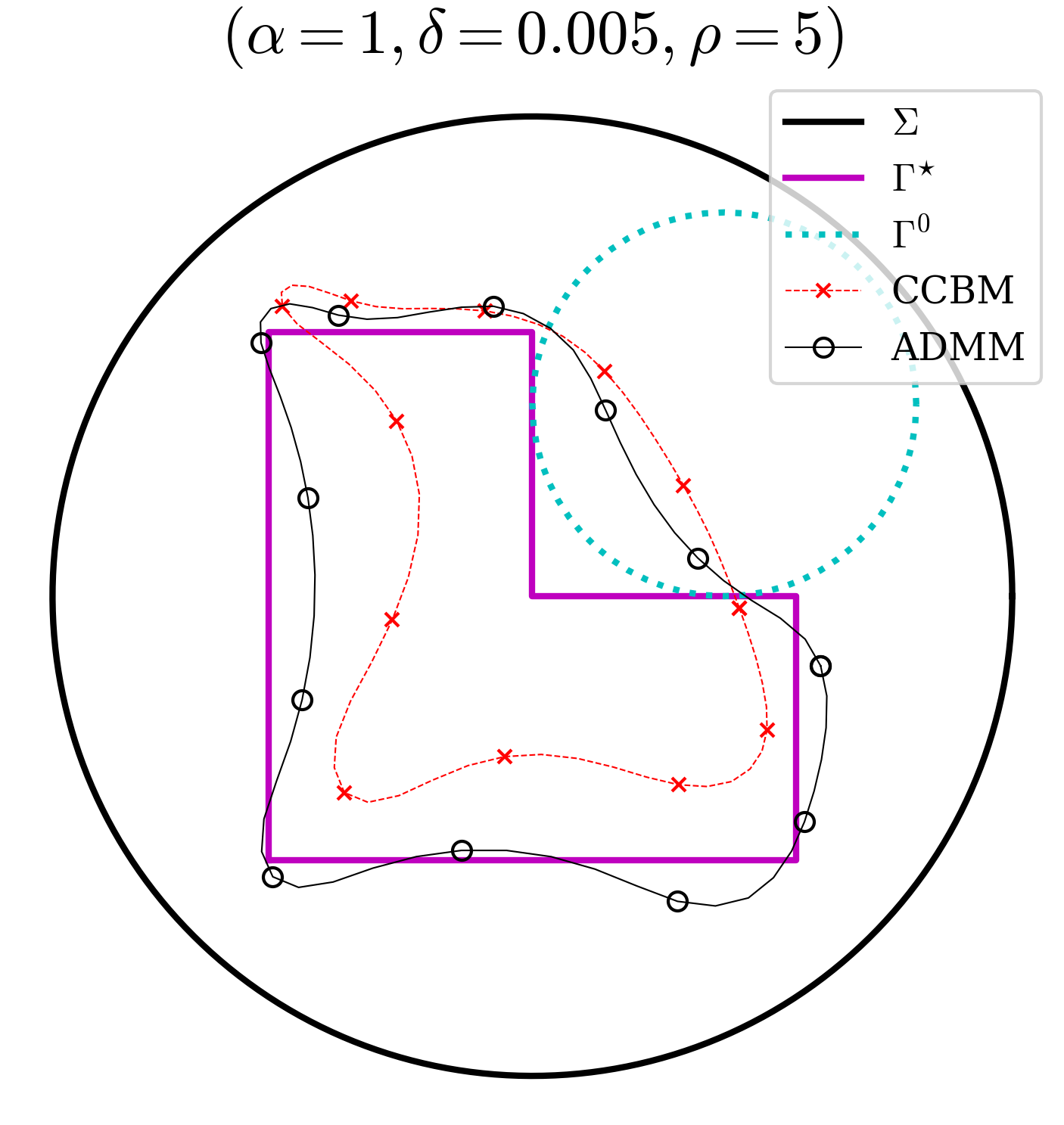}} 
    \resizebox{0.24\textwidth}{!}{\includegraphics{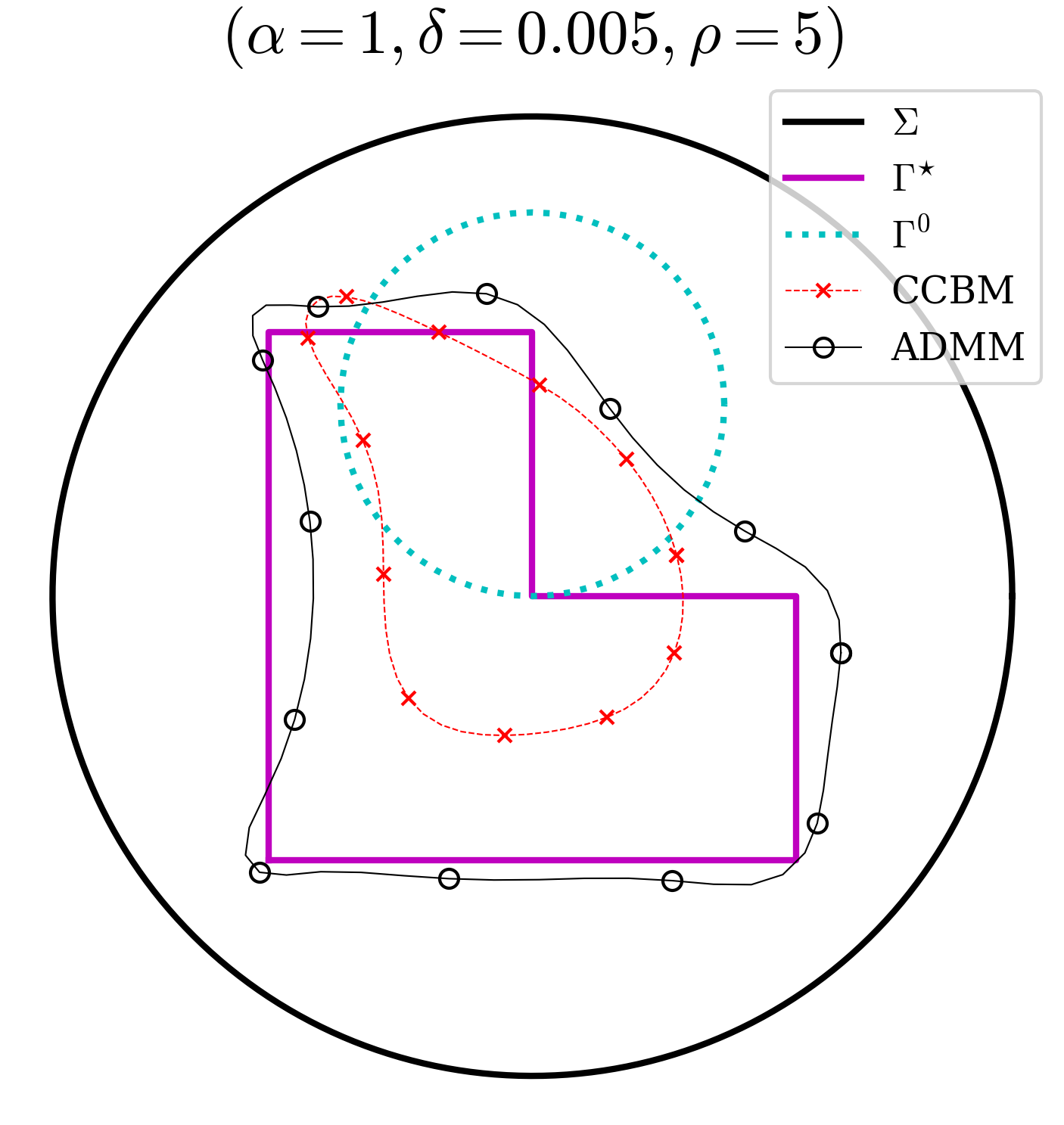}} 
    \resizebox{0.24\textwidth}{!}{\includegraphics{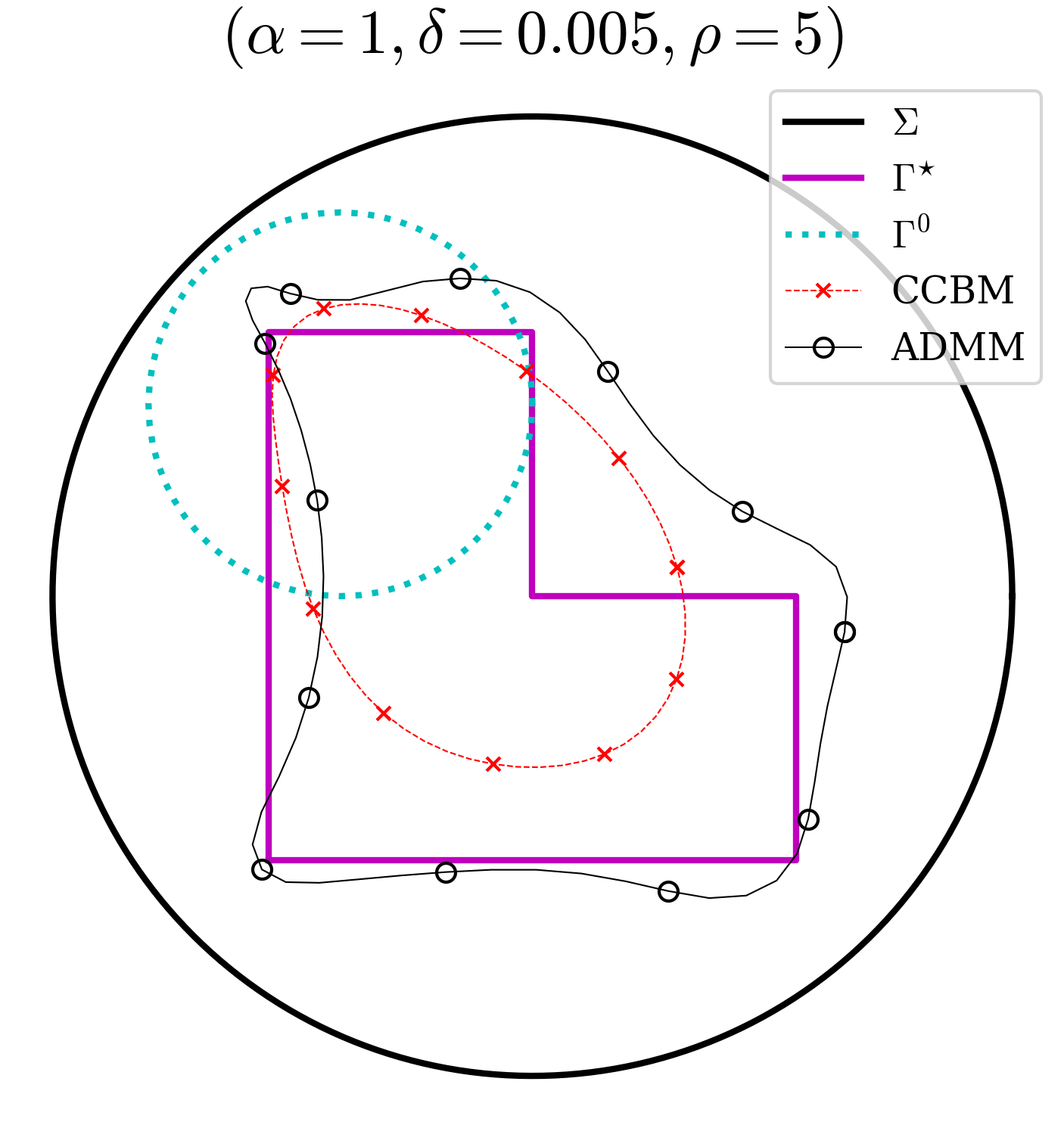}} 
    \resizebox{0.24\textwidth}{!}{\includegraphics{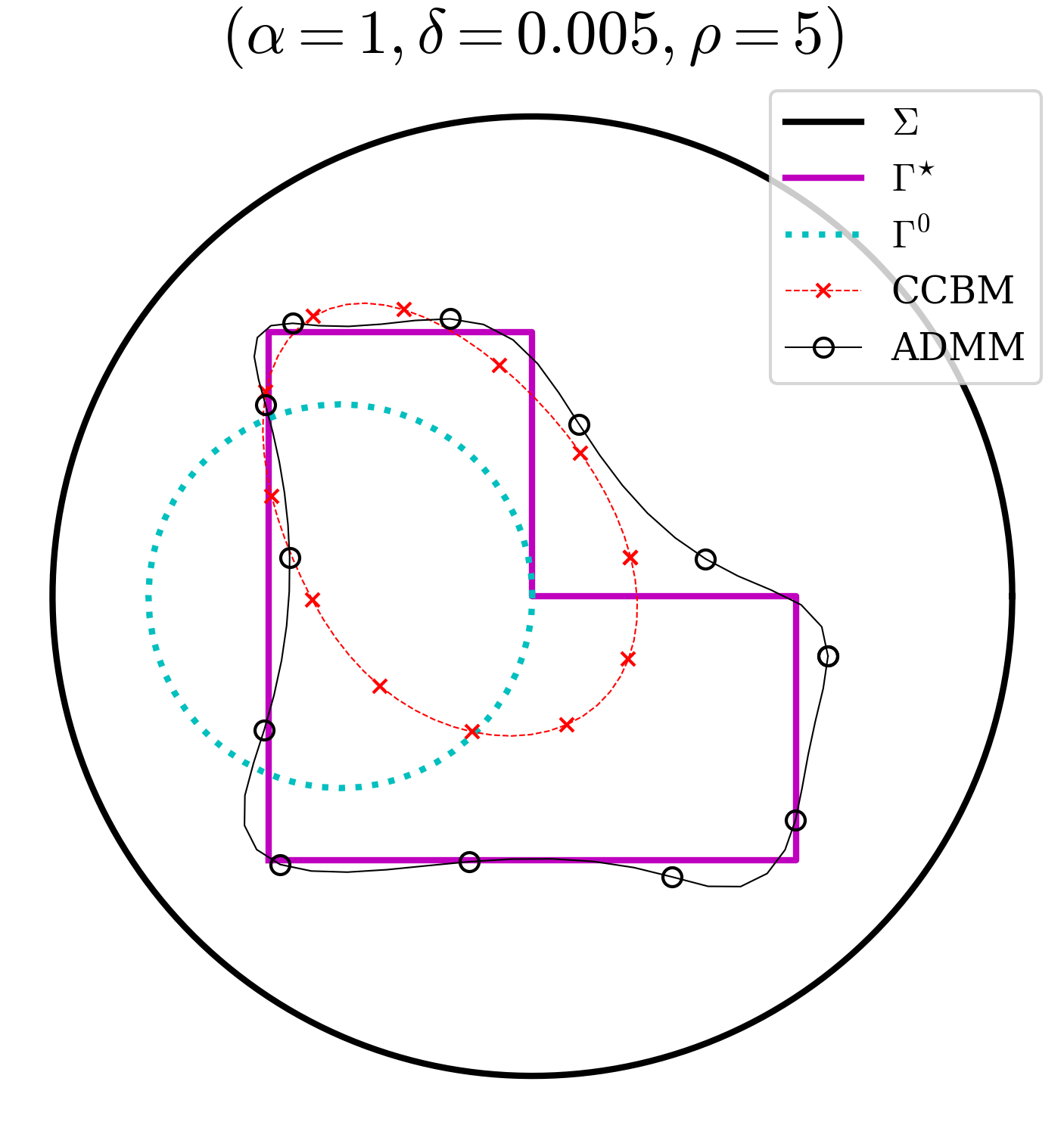}}
\caption{
Reconstructions under noisy measurements with noise level $\delta = 0.005$ for an \texttt{L}-shape cavity under the prescribed Dirichlet data $f = \cos{\arctan(x_{2}/x_{1})} + \sin{(0.1\pi x_{1})} \sin{(0.1\pi x_{2})}$ with $a=\min_{\xi \in \varOmega^{\star}} f(\xi)$ and $b = \max_{\xi \in \varOmega^{\star}} f(\xi)$.
}
\label{ADMMfig:noisy_test_L_fminmax_ab}
\end{figure}

The evolution of the iterates shown in Figure~\ref{ADMMfig:noisy_test_L_fminmax_ab_shape_evolution} indicates a stable and progressive deformation toward the reconstructed shape.
The algorithm first recovers the main geometric features and subsequently refines finer boundary details.
No pronounced oscillatory behavior or numerical instabilities are observed, despite the presence of noise in the data.
This behavior is consistent with the splitting structure of ADMM and further illustrates the stabilizing role of the admissible set $\mathcal{K}$ together with the extension-regularization procedure \eqref{eq:Sobolev_gradient_computation}.

Finally, to complete our numerical investigations, we consider the reconstruction of a cavity containing multiple concave regions, in contrast to the previous examples which involved only a single concavity. 
In this experiment, we choose the Dirichlet input $f = 1.0 + \cos\!\bigl(\arctan(x_{2}/x_{1})\bigr)$.
The corresponding reconstructions are displayed in Figure~\ref{ADMMfig:last_example}. The results indicate that the proposed ADMM-based formulation provides improved qualitative recovery compared with the conventional CCBM approach, particularly in the reconstruction of the concave portions of the boundary.

Moreover, the inclusion of the spatially varying term $\cos(\arctan(x_{2}/x_{1}))$ in the input data highlights the importance of the choice of excitation. In particular, compared with the related numerical experiment reported in \cite{AfraitesRabago2025}, where the constant input $f=1$ was employed, the present choice leads to a noticeably clearer identification of the concave regions. This observation suggests that appropriately chosen nonconstant boundary excitations may improve reconstruction quality, even in the single-measurement setting.

Overall, the numerical results indicate that the proposed ADMM framework yields stable reconstructions under low noise levels. The reconstruction quality depends significantly on the choice of Dirichlet input, with spatially varying excitations generally producing improved results. The penalty parameter $\rho$ plays an important role in balancing constraint enforcement and sensitivity to noise. Furthermore, the imposed inequality constraint suppresses oscillatory behavior and enhances stability, although at the expense of a slight reduction in boundary resolution. Taken together, these observations support the effectiveness of the proposed approach for noisy inverse shape reconstruction problems involving homogeneous Robin boundary conditions.

\begin{figure}[h!]
    \centering 
    \resizebox{0.35\textwidth}{!}{\includegraphics{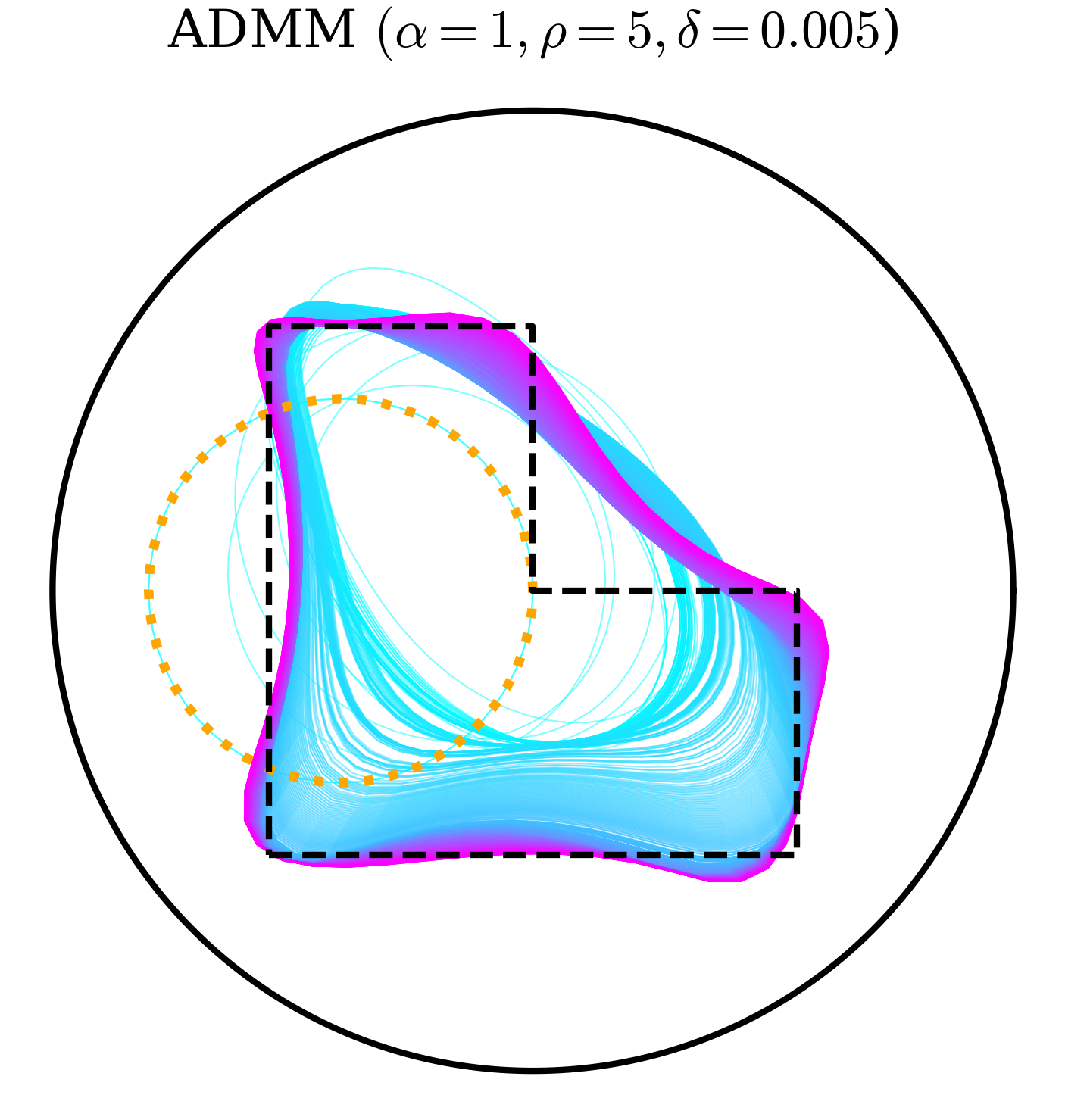}} \quad
    \resizebox{0.35\textwidth}{!}{\includegraphics{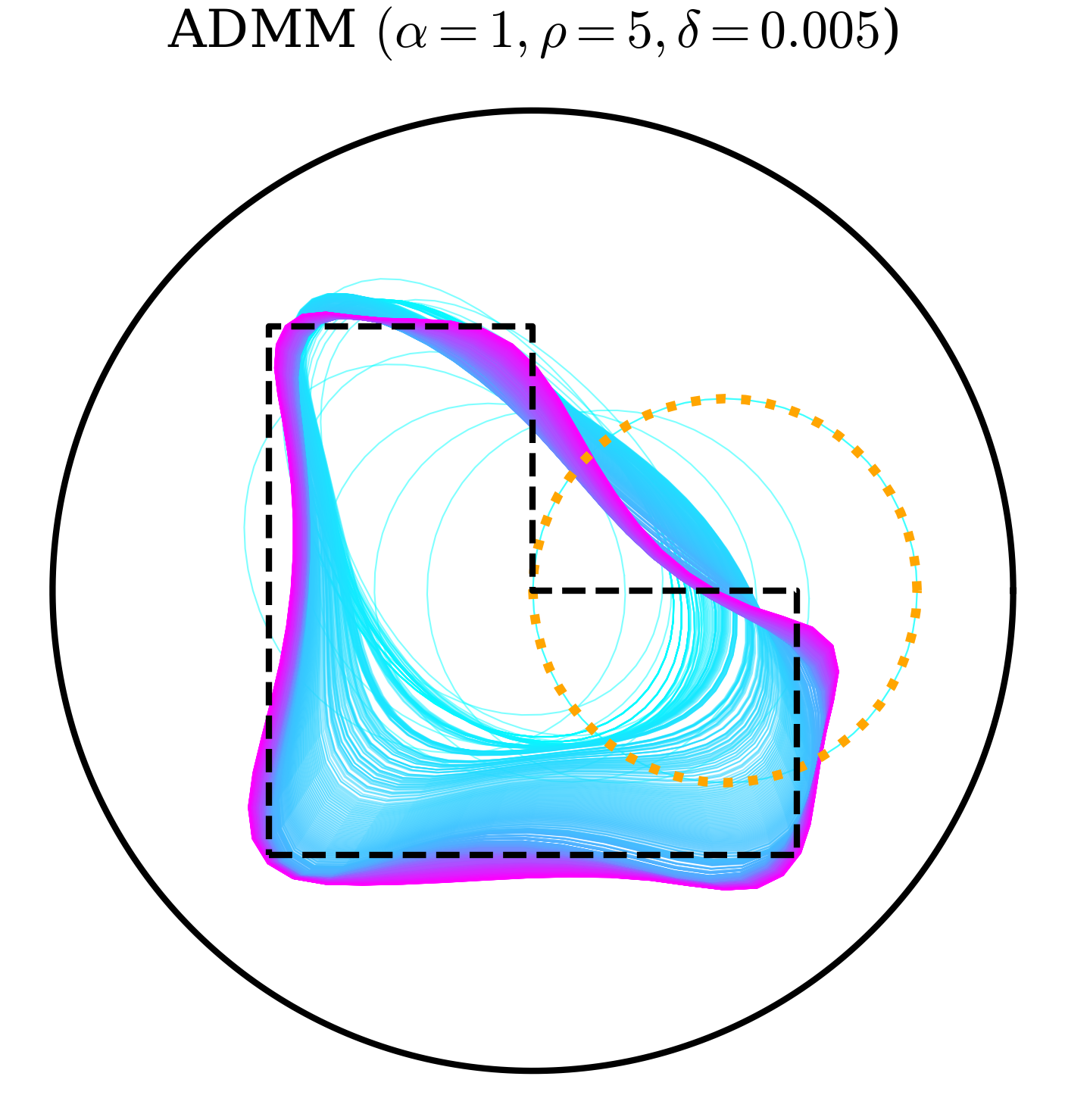}}  
\caption{
Shape evolution of two of the reconstructed shapes in Figure~\ref{ADMMfig:noisy_test_L_fminmax_ab} under ADMM.
}
\label{ADMMfig:noisy_test_L_fminmax_ab_shape_evolution}
\end{figure}

\begin{figure}[h!]
    \centering 
    \resizebox{0.32\textwidth}{!}{\includegraphics{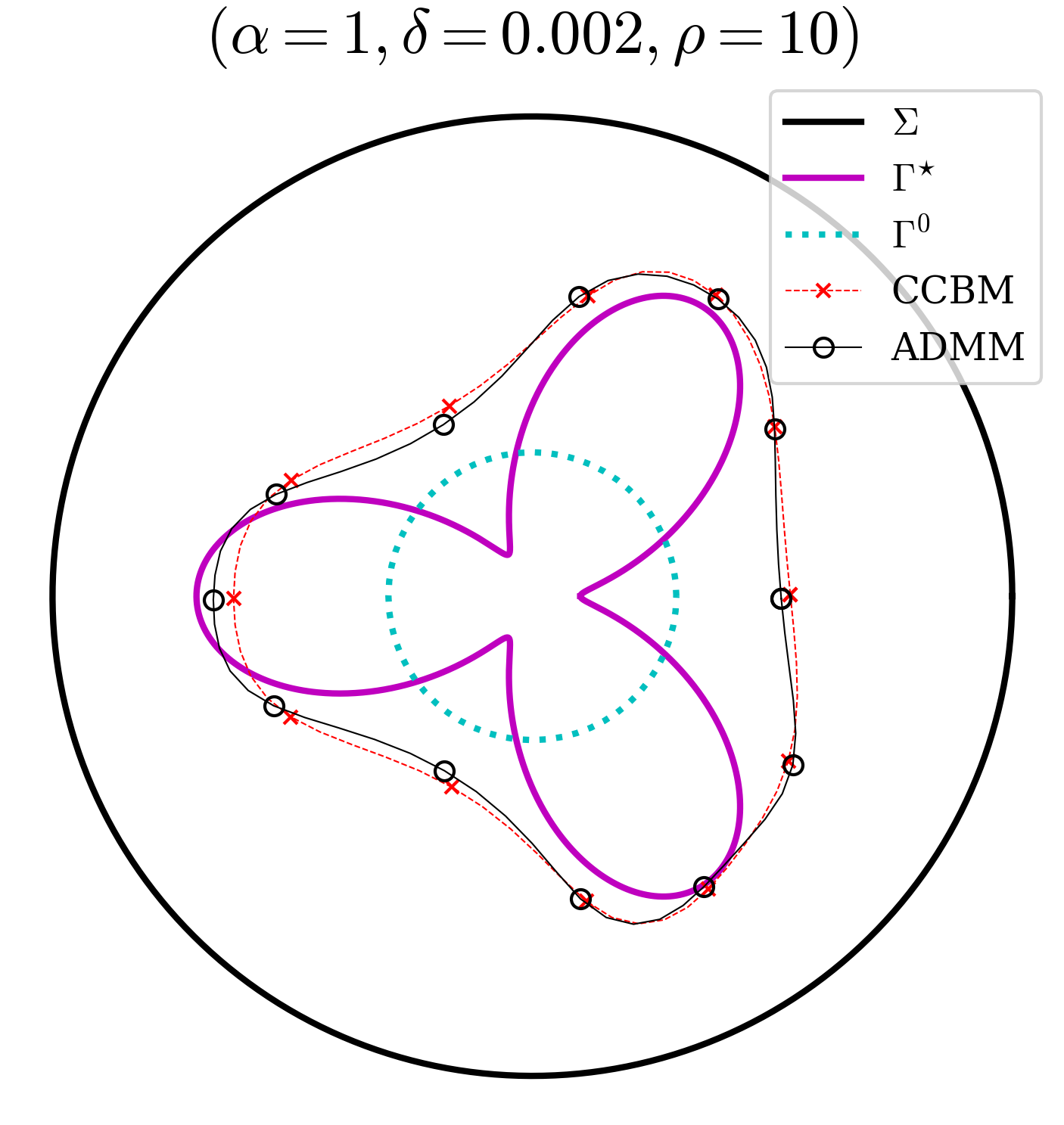}}\hfill
    \resizebox{0.32\textwidth}{!}{\includegraphics{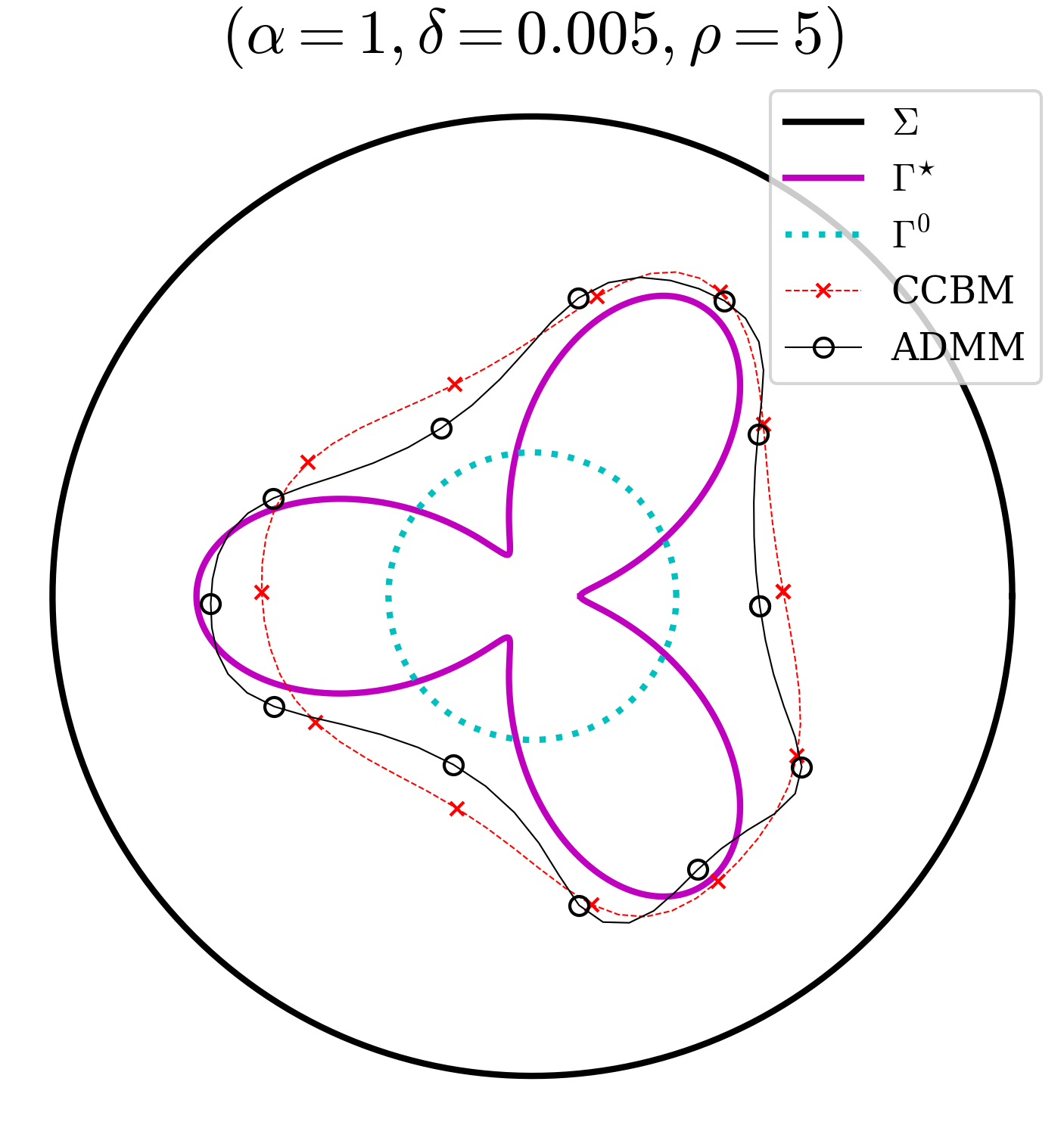}} \hfill
    \resizebox{0.32\textwidth}{!}{\includegraphics{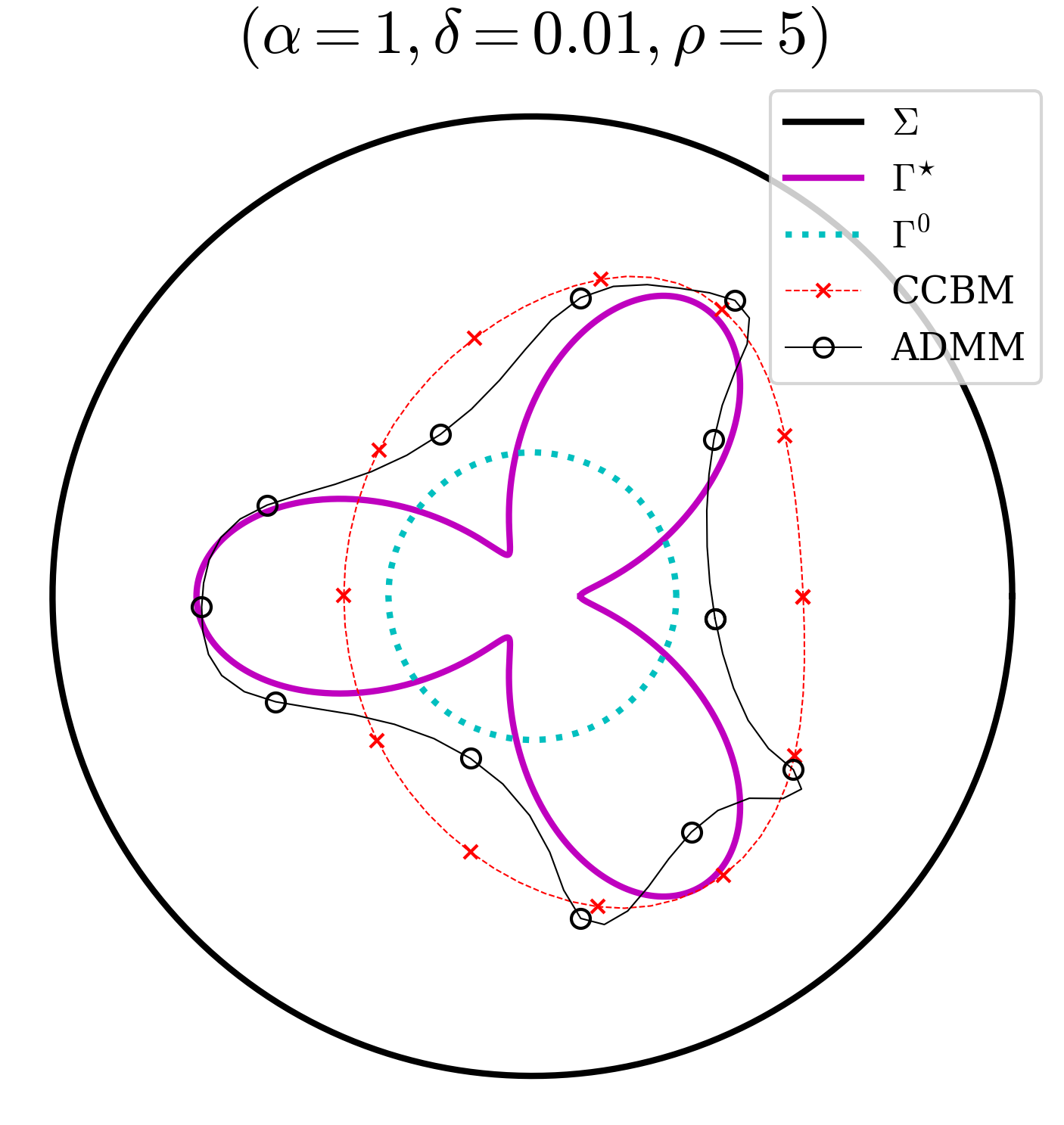}}
\caption{
Reconstructions with multiple concave regions under different noise levels $\delta = 0.002, 0.005, 0.01$.
}
\label{ADMMfig:last_example}
\end{figure}

\color{black}
\section{Summary and concluding remarks}\label{sec:conclusion}

We studied the inverse problem of reconstructing an unknown boundary portion subject to a homogeneous Robin condition from a single pair of Cauchy data prescribed on the accessible boundary. The reconstruction was formulated as a shape optimization problem based on a complex-valued state equation, with descent directions computed from the associated shape derivative and Sobolev gradient.

The problem remains severely ill-posed, particularly in the single-measurement setting. The numerical experiments show that the proposed approach recovers the main geometric features of the inclusion, while fine-scale details deteriorate under noise and depend strongly on the choice of boundary data.

To improve stability, we introduced an ADMM-based formulation with an auxiliary variable subject to an inequality constraint. Compared with the standard CCBM approach, the ADMM scheme exhibits improved robustness with respect to noise and initialization. In particular, the splitting strategy and admissibility constraint reduce spurious oscillations and tend to produce more coherent reconstructions, especially in the noisy regime.

The numerical results also indicate that the penalty parameter and constraint set must be chosen carefully. Weak enforcement may lead to diffused reconstructions, whereas excessive penalization can amplify noise or reduce resolution. In addition, the reconstruction quality depends significantly on the prescribed Dirichlet data, with combined excitations yielding improved results.

Overall, the results indicate that the proposed ADMM--CCBM framework provides a comparatively more stable and robust approach than the standard CCBM formulation for noisy inverse shape reconstruction problems with homogeneous Robin conditions. These findings are consistent with previous works \cite{RabagoHadriAfraitesHendyZaky2024,CherratAfraitesRabago2025b,CherratAfraitesRabago2026} supporting the use of inequality constraints in ADMM-based shape reconstruction methods. Future work includes the incorporation of multiple measurements, adaptive parameter selection strategies, and additional regularization techniques to further improve stability and reconstruction accuracy.
\section*{Acknowledgements}
The work of LA is supported by Toubkal/25/205-Campus France:51628SC. 
The work of JFTR is supported by the JSPS Postdoctoral Fellowships for Research in Japan (Grant Number JP24KF0221), and partially by the JSPS Grant-in-Aid for Early-Career Scientists (Grant Number JP23K13012) and the JST CREST (Grant Number JPMJCR2014).

%

\section*{Conflict of interest} The authors declare that they have no conflict of interest.

\section*{Funding} Not applicable.

\section*{Author Contributions} All listed authors contributed equally to this work.

\section*{Availability of data and materials} The authors confirm that the data supporting the findings of this study
are available within the article and supplementary sources.


\bibliographystyle{alpha} 
\bibliography{main}   

@article{HriziPrakashNovotny2025,
	author = {M. Hrizi and R. Prakash and A. A. Novotny},
	journal = {Comput. Math. Appl.},
	pages = {117--144},
	title = {Approximation of unknown sources in a time fractional {PDE} by the optimal ones and their reconstruction},
	volume = {200},
	year = {2025}}

@article{WuGongGongZhangZhu2026,
	author = {Q. Wu and R. Gong and W. Gong and Z. Zhang and S. Zhu},
	journal = {Inverse Problems},
	number = {3},
	pages = {035008},
	title = {Bioluminescence tomography via a shape optimization method based on a complex-valued model},
	volume = {42},
	year = {2026}}

@unpublished{AfraitesHadriHriziRabago2026,
	author = {L. Afraites and A. Hadri and M. Hrizi and J. F. T. Rabago},
	month = {March},
	note = {arXiv preprint arXiv:2603.04216},
	title = {Statistical topological gradient and shape optimization for robust metal--semiconductor contact reconstruction},
	year = {2026}}

@article{RabagoKimura2025,
	author = {J. F. T. Rabago and M. Kimura},
	journal = {ESAIM Control Optim. Calc. Var.},
	pages = {Art. 64},
	title = {On the well-posedness of a {H}ele-{S}haw-like system resulting from an inverse geometry problem formulated through a shape optimization setting},
	volume = {31},
	year = {2025}}

@article{Chicco1997,
	author = {M. Chicco},
	journal = {Boll. Un. Mat. Ital. B (7)},
	pages = {531--538},
	title = {A maximum principle for mixed boundary value problems for elliptic equations in non-divergence form},
	volume = {11},
	year = {1997}}

@article{CherratAfraitesRabago2026,
	author = {E. Cherrat and L. Afraites and J. F. T. Rabago},
	journal = {J. Optim. Theory Appl.},
	pages = {Art. 44, 28 pp.},
	title = {Enhanced shape recovery in advection--diffusion problems via a novel {A}{D}{M}{M}-based {C}{C}{B}{M} optimization},
	volume = {209},
	year = {2026}}

@article{RabagoHadriAfraitesHendyZaky2024,
	author = {J. F. T. Rabago and A. Hadri and L. Afraites and A. S. Hendy and M. A. Zaky},
	journal = {Comput. Math. Appl.},
	month = {September},
	pages = {19-32},
	title = {A robust alternating direction numerical scheme in a shape optimization setting for solving geometric inverse problems},
	volume = {175},
	year = {2024}}

@article{RabagoAfraitesNotsu2025,
	author = {J. F. T. Rabago and L. Afraites and H. Notsu},
	journal = {SIAM J. Control \& Optim.},
	number = {2},
	pages = {822--851},
	title = {Detecting immersed obstacle in {S}tokes fluid flow using the coupled complex boundary method},
	volume = {63},
	year = {2025}}

@article{ZhengChengGong2020,
	author = {X. Zheng and X. Cheng and R. Gong},
	journal = {Int. J. Comput. Math.},
	number = {5},
	pages = {998--1015},
	title = {A coupled complex boundary method for parameter identification in elliptic problems},
	volume = {97},
	year = {2020}}

@article{Rabago2023b,
	author = {J. F. T. Rabago},
	journal = {Math. Control Relat. Fields},
	number = {4},
	pages = {1362--1398},
	title = {On the new coupled complex boundary method in shape optimization framework for solving stationary free boundary problems},
	volume = {13},
	year = {2023}}

@article{CherratAfraitesRabago2025,
	author = {E. Cherrat and L. Afraites and J. F. T. Rabago},
	journal = {Discrete Contin. Dyn. Syst. Ser. S},
	number = {1},
	pages = {296--320},
	title = {Shape reconstruction for advection-diffusion problems by shape optimization techniques: The case of constant velocity},
	volume = {18},
	year = {2025}}

@article{CherratAfraitesRabago2025b,
	author = {E. Cherrat and L. Afraites and J. F. T. Rabago},
	journal = {Appl. Math. Optim.},
	pages = {Art. 13},
	title = {Numerical solution by shape optimization method to an inverse shape problem in multi-dimensional advection-diffusion problem with space dependent coefficients},
	volume = {92},
	year = {2025}}

@book{Azegami2020,
	address = {Singapore},
	author = {H. Azegami},
	publisher = {Springer},
	series = {Springer Optimization and Its Applications},
	title = {Shape Optimization Problems},
	volume = {164},
	year = {2020}}

@article{Rabago2025,
	author = {J. F. T. Rabago},
	journal = {Appl. Numer. Math.},
	pages = {135--171},
	title = {Localization of tumor through a non-conventional numerical shape optimization technique},
	volume = {217},
	year = {2025}}

@article{RabagoNotsu2024,
	author = {J. F. T. Rabago and H. Notsu},
	journal = {Appl. Math. Optim.},
	pages = {Art. 2},
	title = {Numerical solution to a free boundary problem for the {S}tokes equation using the coupled complex boundary method in shape optimization setting},
	volume = {89},
	year = {2024}}

@article{AfraitesRabago2024,
	author = {L. Afraites and J. F. T. Rabago},
	journal = {Comput. Appl. Math.},
	number = {5},
	pages = {Art. 270},
	title = {Boundary shape reconstruction with {R}obin condition: existence result, stability analysis, and inversion via multiple measurements},
	volume = {43},
	year = {2024}}

@techreport{Loh1987,
	author = {W. H. Loh},
	institution = {Stanford Electronics Labs.},
	number = {830-1},
	title = {Modeling and Measurement of Contact Resistances},
	type = {Tech. Rep. No. G},
	year = {1987}}

@article{FangLu2004,
	author = {W. Fang and M. Lu},
	journal = {Internat. J. Numer. Methods Engrg.},
	pages = {1563--1585},
	title = {A fast collocation method for an inverse boundary value problem,},
	volume = {59},
	year = {2004}}

@article{CakoniKressSchuft2010a,
	author = {F. Cakoni and R. Kress and C. Schuft},
	journal = {Inverse Problems},
	pages = {Art. 095012 24pp.},
	title = {Integral equations for inverse problems in corrosion detection from partial {C}auchy data},
	volume = {26},
	year = {2010}}

@article{Sincich2010,
	author = {E. Sincich},
	journal = {SIAM J. Math. Anal.},
	pages = {2922--2943},
	title = {Stability for the determination of unknown boundary and impedance with a {R}obin boundary condition},
	volume = {42},
	year = {2010}}

@article{PaganiPeerotti2009,
	author = {C. D. Pagani and D. Peerotti},
	journal = {Inverse Problems},
	pages = {Art. 055007 12pp},
	title = {Identifiability problems of defects with {R}obin condition},
	volume = {25},
	year = {2009}}

@article{Fang2022,
	author = {W. Fang},
	journal = {J. Comput. Appl. Math.},
	pages = {Art. 114376 13 pp},
	title = {Simultaneous recovery of {R}obin boundary and coefficient for the {L}aplace equation by shape derivative},
	volume = {413},
	year = {2022}}

@article{FangZeng2009,
	author = {W. Fang and S. Zeng},
	journal = {J. Comput. Appl. Math.},
	pages = {573--580},
	title = {Numerical recovery of {R}obin boundary from boundary measurements for the {L}aplace equation},
	volume = {224},
	year = {2009}}

@article{FangLinMa2019,
	author = {W. Fang and F. Lin and Y. Ma},
	journal = {East Asian J. Appl. Math.},
	pages = {485--505},
	title = {Fast algorithms for boundary integral equations on elliptic domains and related inverse problems},
	volume = {9},
	year = {2019}}

@article{AfraitesMasnaouiNachaoui2022,
	author = {L. Afraites and C. Masnaoui and M. Nachaoui},
	journal = {Discrete Contin. Dyn. Syst. Ser. S},
	number = {1},
	pages = {1--21},
	title = {Shape optimization method for an inverse geometric source problem and stability at critical shape},
	volume = {15},
	year = {2022}}

@article{CaubetDambrineKateb2013,
	author = {F. Caubet and M. Dambrine and D. Kateb},
	journal = {Inverse Problems},
	pages = {Art. 115011 (26pp)},
	title = {Shape optimization methods for the inverse obstacle problem with generalized impedance boundary conditions},
	volume = {29},
	year = {2013}}

@article{KressRundell2005,
	author = {R. Kress and W. Rundell},
	journal = {Inverse Problems},
	pages = {1207--1223},
	title = {Nonlinear integral equations and the iterative solution for an inverse boundary value problem},
	volume = {21},
	year = {2005}}

@article{Bacchelli2009,
	author = {V. Bacchelli},
	journal = {Inverse Problems},
	pages = {Art. 015004 (4pp)},
	title = {Uniqueness for the determination of unknown boundary and impedance with the homogeneous {R}obin condition},
	volume = {25},
	year = {2009}}

@article{KaupSantosaVogelius1996,
	author = {P. G. Kaup and F. Santosa and M. Vogelius},
	journal = {Inverse Problems},
	pages = {279--293},
	title = {Method for imaging corrosion damage in thin plates from electrostatic data},
	volume = {12},
	year = {1996}}

@article{KaupSantosa1995,
	author = {P. G. Kaup and F. Santosa},
	journal = {J. Nondestr. Eval.},
	pages = {127--136},
	title = {Nondestructive evaluation of corrosion damage using electrostatic measurements},
	volume = {14},
	year = {1995}}

@article{Inglese1997,
	author = {G. Inglese},
	journal = {Inverse Problems},
	pages = {977--994},
	title = {An inverse problem in corrosion detection},
	volume = {13},
	year = {1997}}

@article{FasinoInglese2007,
	author = {D. Fasino and G. Inglese},
	journal = {J. Comput. Appl. Math.},
	pages = {460--470},
	title = {Recovering nonlinear terms in an inverse boundary value problem for {L}aplace's equation: a stability estimate},
	volume = {198},
	year = {2007}}

@article{Rundell2008,
	author = {W. Rundell},
	journal = {Inverse Problems},
	pages = {1--22},
	title = {Recovering an obstacle and its impedance from {C}auchy data},
	volume = {24},
	year = {2008}}

@article{IngleseMariani2004,
	author = {G. Inglese and F. Mariani},
	journal = {Inverse Problems},
	pages = {1207--1215},
	title = {Corrosion detection in conducting boundaries},
	volume = {20},
	year = {2004}}

@article{ChaabaneaJaoua1999,
	author = {S. Chaabane and M. Jaoua},
	journal = {Inverse Problems},
	pages = {1425--1438},
	title = {Identification of {R}obin coefficients by means of boundary measurements},
	volume = {15},
	year = {1999}}

@article{AlessandriniSincich2007,
	author = {G. Alessandrini and E. Sincich},
	journal = {J. Comput. Appl. Math.},
	pages = {307--320},
	title = {Solving elliptic {C}auchy problems and identification of non-linear corrosion},
	volume = {198},
	year = {2007}}

@article{AlessandriniDelPieroRondi2003,
	author = {G. Alessandrini and L. Del Piero and L. Rondi},
	journal = {Inverse Problems},
	pages = {973--984},
	title = {Stable determination of corrosion by a single electrostatic boundary measurement},
	volume = {19},
	year = {2003}}

@article{CakoniKress2007,
	author = {F. Cakoni and R. Kress},
	journal = {Inverse Prob. Imaging},
	pages = {229--245},
	title = {Integral equations for inverse problems in corrosion detection from partial Cauchy data},
	volume = {1},
	year = {2007}}

@article{Ouiassaetal2022,
	author = {H. Ouaissa and A. Chakib and A. Nachaoui and M. Nachaoui},
	journal = {Appl. Math. Optim.},
	number = {Art. 3},
	pages = {37 pp.},
	title = {On Numerical Approaches for Solving an Inverse {C}auchy {S}tokes Problem},
	volume = {85},
	year = {2022}}

@book{DautrayLionsv21998,
	author = {R. Dautray and J.-L. Lions},
	publisher = {Springer},
	title = {Mathematical {A}nalysis and {N}umerical {M}ethods for {S}cience and {T}echnology},
	volume = {2},
	year = {1998}}

@article{Afraites2022,
	author = {L. Afraites},
	journal = {Discrete Contin. Dyn. Syst. Ser. S},
	number = {1},
	pages = {23 -- 40},
	title = {A new coupled complex boundary method ({C}{C}{B}{M}) for an inverse obstacle problem},
	volume = {15},
	year = {2022}}

@article{Gongetal2017,
	author = {R. Gong and X. Cheng and W. Han},
	journal = {Appl. Anal.},
	number = {5},
	pages = {869--885},
	title = {A coupled complex boundary method for an inverse conductivity problem with one measurement},
	volume = {96},
	year = {2017}}

@article{Chengetal2014,
	author = {X. L. Cheng and R. F. Gong and W. Han and X. Zheng},
	journal = {Inverse Problems},
	pages = {Article 055002},
	title = {A novel coupled complex boundary method for solving inverse source problems},
	year = {2014}}

@article{Doganetal2007,
	author = {G. Do\v{g}an and P. Morin and R.H. Nochetto and M. Verani},
	journal = {Comput. Methods Appl. Mech. Engrg.},
	pages = {3898--3914},
	title = {Discrete gradient flows for shape optimization and applications},
	volume = {196},
	year = {2007}}

@techreport{MuratSimon1976,
	address = {Paris},
	author = {F. Murat and J. Simon},
	institution = {Univ. Pierre et Marie Curie},
	number = {76015},
	title = {Sur le contr\^{o}le par un domaine g\'{e}om\'{e}trique},
	type = {Research report},
	year = {1976}}

@book{SokolowskiZolesio1992,
	address = {Berlin, Heidelberg},
	author = {J. Soko\l{}owski and J.-P. Zol\'{e}sio},
	publisher = {Springer-Verlag},
	series = {Springer Series in Computational Mathematics},
	title = {{I}ntroduction to {S}hape {O}ptimization: {S}hape {S}ensitivity {A}nalysis},
	year = {1992}}

@book{SauterSchwab2011,
	address = {Berlin, Heidelberg},
	author = {S. A. Sauter and C. Schwab},
	publisher = {Springer-Verlag},
	title = {Boundary Element Methods},
	year = {2011}}

@article{Simon1980,
	author = {J. Simon},
	journal = {Numer. Funct. Anal. Optim.},
	pages = {649-687},
	title = {Differentiation with respect to the domain in boundary value problems},
	volume = {2},
	year = {1980}}

@article{RabagoAzegami2020,
	author = {J. F. T. Rabago and H. Azegami},
	journal = {Comput. Optim. Appl.},
	number = {1},
	pages = {251--305},
	title = {A second-order shape optimization algorithm for solving the exterior {B}ernoulli free boundary problem using a new boundary cost functional},
	volume = {77},
	year = {2020}}

@article{RabagoAzegami2018,
	author = {J. F. T. Rabago and H. Azegami},
	journal = {Japan J. Indust. Appl. Math.},
	number = {1},
	pages = {131-176},
	title = {Shape optimization approach to defect-shape identification with convective boundary condition via partial boundary measurement},
	volume = {31},
	year = {2018}}

@book{Neuberger1997,
	address = {Berlin},
	author = {J. W. Neuberger},
	publisher = {Springer-Verlag},
	title = {Sobolev Gradients and Differential Equations},
	year = {1997}}

@article{IKP2008,
	author = {K. Ito and K. Kunisch and G. H. Peichl},
	journal = {ESAIM Control Optim. Calc. Var.},
	pages = {517-539},
	title = {Variational approach to shape derivatives},
	volume = {14},
	year = {2008}}

@book{HenrotPierre2018,
	address = {Z\"{u}rich},
	author = {A. Henrot and M. Pierre},
	publisher = {European Mathematical Society},
	series = {Tracts in Mathematics},
	title = {Shape Variation and Optimization: A Geometrical Analysis},
	volume = {28},
	year = {2018}}

@article{Hecht2012,
	author = {F. Hecht},
	journal = {J. Numer. Math.},
	pages = {251-265},
	title = {New development in {F}ree{F}em++},
	volume = {20},
	year = {2012}}

@article{HIKKP2009,
	author = {J. Haslinger and K. Ito and T. Kozubek and K. Kunisch and G. H. Peichl},
	journal = {Interfaces Free Bound.},
	pages = {317-330},
	title = {On the shape derivative for problems of {B}ernoulli type},
	volume = {11},
	year = {2009}}

@book{DelfourZolesio2011,
	address = {Philadelphia},
	author = {M. C. Delfour and J.-P. Zol\'{e}sio},
	edition = {2nd},
	publisher = {SIAM},
	series = {Adv. Des. Control},
	title = {{S}hapes and {G}eometries: {M}etrics, {A}nalysis, {D}ifferential {C}alculus, and {O}ptimization},
	volume = {22},
	year = {2011}}

@article{AfraitesRabago2025,
	author = {L. Afraites and J. F. T. Rabago},
	journal = {Discrete Contin. Dyn. Syst. Ser. S},
	number = {1},
	pages = {43--76},
	title = {Shape optimization methods for detecting an unknown boundary with the {R}obin condition by a single measurement},
	volume = {18},
	year = {2025}}



\end{document}